\documentclass[11pt, a4paper]{amsart}

\title[Operator algebras and Subproduct Systems from Stochastic Matrices]{
Operator algebras and subproduct systems arising from stochastic matrices
}
\author{Adam Dor-On}
\author[Daniel Markiewicz]{Daniel Markiewicz*}
\address{Adam Dor-On and Daniel Markiewicz,
Department of Mathematics,
Ben-Gurion University of the Negev, 
P.O.B. 653, Beersheva 84105,
Israel.} 
\email{adamd@math.bgu.ac.il, danielm@math.bgu.ac.il}
\thanks{The first author was partially supported by 
GIF (German-Israeli Foundation) research grant No. 2297-2282.6/201, and the second author was 
partially supported by
grant 2008295 from the U.S.-Israel Binational Science Foundation.}
\thanks{*corresponding author}

\date{\today}

\usepackage[a4paper, margin=2.3cm]{geometry}
\usepackage{amssymb}
\usepackage{xspace}
\renewcommand{\MR}[1]{} % remove MR info from bibliography
\newcommand{\la}{\langle}
\newcommand{\ra}{\rangle}
\newcommand{\vnM}{\mathcal{M}}
\newcommand{\vnN}{\mathcal{N}}
\newcommand{\vnMaS}{\ell^{\infty}(\Omega)}
\newcommand{\AS}{Arveson-Stinespring}
\newcommand{\St}{\Omega}
\DeclareMathOperator{\mindeg}{mindeg}
\DeclareMathOperator{\Res}{Res}

\DeclareMathOperator{\Span}{Span}
\DeclareMathOperator{\Ad}{Ad}
\newcommand{\nn}{\mathbb{N}}

\newcommand{\cc}{\mathbb{C}}
\newcommand{\tensor}{\mathcal{T}_+}

\DeclareMathOperator{\Diag}{Diag}
\newcommand{\cala}{\mathcal{A}}
\newcommand{\calb}{\mathcal{B}}
\newcommand{\cals}{\mathcal{S}}
\newcommand{\calm}{\vnM}
\newcommand{\streamlined}{streamlined\xspace}

\newcommand{\LL}{\mathcal{L}}
\newcommand{\ee}{\mathcal{E}}
\newcommand{\rp}{{\mathcal{R}_P}}
\newcommand{\calr}{\mathcal{R}}

\newcommand{\torus}{\mathbb{T}}
\newcommand{\disk}{\mathbb{D}}

\newtheorem{theorem}{Theorem}[section]
\newtheorem{defi}[theorem]{Definition}
\newtheorem{lemma}[theorem]{Lemma}

\newtheorem{proposition}[theorem]{Proposition}
\newtheorem{corollary}[theorem]{Corollary}
\newtheorem{remark}[theorem]{Remark}
\newtheorem{notation}[theorem]{Notation}
\newtheorem{example}[theorem]{Example}
\numberwithin{equation}{section}
\begin{document}

\begin{abstract}
We study subproduct systems in the sense of Shalit and 
Solel arising from stochastic matrices on countable
state spaces, and their associated operator algebras.
We focus on the non-self-adjoint tensor algebra, and
Viselter's generalization of the Cuntz-Pimsner 
C*-algebra to the context of subproduct systems.
Suppose that $X$ and $Y$ are Arveson-Stinespring 
subproduct systems associated to two stochastic matrices 
over a countable set $\Omega$, and 
let $\tensor(X)$ and $\tensor(Y)$ be their tensor
algebras. We show that every algebraic
isomorphism from $\tensor(X)$ onto $\tensor(Y)$
is automatically bounded. Furthermore, 
$\tensor(X)$ and $\tensor(Y)$
are isometrically isomorphic if and only if 
$X$ and $Y$ are unitarily isomorphic up to
a *-automorphism of $\vnMaS$. When $\Omega$
is finite, we prove that $\tensor(X)$ and 
$\tensor(Y)$ are algebraically isomorphic if and only 
if there exists a similarity between $X$ and $Y$ up
to a *-automorphism of $\vnMaS$.
Moreover, we provide an explicit description of the Cuntz-Pimsner algebra $\mathcal{O}(X)$ 
in the case where $\Omega$ is finite and the stochastic matrix is 
essential.
\end{abstract}
\maketitle

% ****************
% * Introduction *       
% ****************

\section{Introduction}

In this paper we study the structure of tensor and Cuntz-Pimsner algebras (in the sense of
 Viselter~\cite{Vis}) associated to subproduct systems, 
 and to what extent these algebras provide invariants 
 for their subproduct systems. 
These algebras generalize the tensor and Cuntz-Pimsner 
operator algebras associated to 
C*-correspondences, which have been the focus of considerable interest by many researchers.
tensor  algebras of a C*-correspondence, in particular, have been the subject of
a deep study by 
Muhly and Solel~\cite{muhly-solel-tensor, 
muhly-solel-morita, MS}, which has led
into a far-reaching non-commutative generalization of
function theory. 
We will focus on subproduct systems associated to stochastic matrices,
and in this context we prove several results which have a close parallel
in the work of Davidson, Ramsey and Shalit~\cite{DRS} 
on the isomorphism problem of tensor algebras of 
subproduct systems over $\cc$ with finite dimensional 
(Hilbert space) fibers.

A subproduct system over a W*-algebra $\vnM$ (and
over the additive semigroup $\nn$) is a family
$\{ X_n \}_{n\in\nn}$ of W*-correspondences over $\vnM$
endowed with an isometric comultiplication  
$X_{n+m} \to X_n \otimes X_m$ which is an adjointable 
$\vnN$-bimodule map for every $n,m$. 
 Subproduct
systems were first defined and studied for their own
sake by Shalit and Solel~\cite{SS}, and 
in the special case of $\vnM=\cc$
they were also independently studied under the name
of inclusion systems by
Bhat and Mukherjee~\cite{BM}. 
Subproduct systems had appeared implicitly earlier in the work of many
researchers 
 in the study of dilations of semigroups
of completely positive maps (cp-semigroups for short) 
on von Neumann algebras and later C*-algebras (see for
example \cite{BhS, MS, Dan}).
The study of cp-semigroups is closely related
to the analysis of E$_0$-semigroups and
product systems
pioneered by Arveson and Powers
(for a comprehensive introduction
 see \cite{arv-monograph},
and also 
 \cite{skeide-habilitation} for product systems of Hilbert
 modules).

Given a correspondence $E$ over a 
C*-algebra $\mathcal{A}$, the Toeplitz C*-algebra $\mathcal{T}(E)$ and the
Cuntz-Pimsner C*-algebra $\mathcal{O}(E)$ were introduced by 
Pimsner~\cite{Pims},  and
modified by Katsura~\cite{Kats} in the case of non-injective
left action of $\mathcal{A}$. As is well-known, in general 
the Cuntz-Pimsner
algebra does not provide a very strong invariant of the
underlying correspondence. However, some information does
remain. In the case of graph C*-algebras, for example, if
a graph is row-finite, then its C*-algebra is simple if and only if the graph is cofinal and every cycle has an entry. And it is easy to find  two graphs with $d$ vertices and
irreducible adjacency matrix whose C*-algebras are not isomorphic (see \cite{raeburn-graph-algebras}). In contrast, in section~\ref{section:cuntz-pimsner-stochastic} we show that
if $X$ is the Arveson-Stinespring subproduct system of a $d \times d$ irreducible stochastic matrix, then $\mathcal{O}(X) \cong C(\mathbb{T})\otimes M_d(\cc)$. More generally,
we also provide an explicit description for the
Cuntz-Pimsner algebra of a subproduct system associated to  essential finite stochastic matrices.

On the other hand, the non-self-adjoint tensor
algebra $\mathcal{T}_+(E)$ of a C*-correspondence $E$ over 
$\mathcal{A}$ has 
often proven to be a strong invariant of the correspondence.
Muhly and Solel~\cite{muhly-solel-morita} proved that
if $E$ and $F$ are aperiodic C*-correspondences, then
$\tensor(E)$ is  isometrically isomorphic to $\tensor(F)$ if
and only if $E$ and $F$ are isometrically isomorphic as $\mathcal{A}$-bimodules. Similarly, Katsoulis and Kribs~\cite{katsoulis-kribs} and Solel~\cite{solel-quiver}
proved that if $G$ and $G'$ are countable directed graphs,
then the tensor algebras $\tensor(G)$ and $\tensor(G')$ are
isomorphic as algebras if and only if $G$ and $G'$ 
are isomorphic as directed graphs. 
See also Davidson and Katsoulis~\cite{davidson-katsoulis} 
for another important
example of this phenomenon of increased acuity of the normed (non-self-adjoint)
algebras as opposed to C*-algebras, perhaps first recognized in
Arveson~\cite{arv-non-self-adjoint} and Arveson and Josephson~\cite{arv-josephson}.

The tensor algebras of subproduct
systems were first considered by
Solel and Shalit~\cite{SS} in the special case
of $\vnM=\cc$, and they analyzed the
problem of \emph{graded} isomorphism of their tensor algebras.
The general isomorphism problem for such subproduct systems
was resolved by Davidson, 
Ramsey and Shalit~\cite{DRS}. They proved that
if $X, Y$ are subproduct systems of finite-dimensional
Hilbert space fibers, then $\tensor(X)$ and $\tensor(Y)$
are isometrically isomorphic if and only if $X$ and $Y$
are (unitarily) isomorphic. 

On the other hand, the recent work of Gurevich~\cite{G}
provides a useful contrast. 
Although in this paper we focus on subproduct systems
over $\nn$, it is possible to consider
more general semigroups. Gurevich 
studied subproduct systems over the semigroup $\mathbb{N} \times \mathbb{N}$, with finite dimensional Hilbert space fibers. He proved in \cite{G}
that subproduct systems of that type in a certain large
class can be distinguished by their tensor algebras, however 
he also provided an example of two non-isomorphic 
subproduct systems over $\nn\times\nn$ of
finite dimensional Hilbert space fibers
whose tensor algebras are isometrically isomorphic.

The following are our main results. 
Suppose that $X$ and $Y$ are Arveson-Stinespring 
subproduct systems over $\vnM=\vnMaS$, 
associated to two stochastic matrices over a 
countable set $\Omega$, and 
let $\tensor(X)$ and $\tensor(Y)$ be their tensor
algebras. Then $\tensor(X)$ and $\tensor(Y)$
are isometrically isomorphic if and only if 
$X$ and $Y$ are unitarily isomorphic up to
a *-automorphism of $\vnMaS$. Every algebraic
isomorphism from $\tensor(X)$ onto $\tensor(Y)$
is automatically bounded. Furthermore, when
 $\Omega$ is finite, $\tensor(X)$ and 
$\tensor(Y)$ are algebraically isomorphic if and only 
if there exists a similarity between $X$ and $Y$ up
to a *-automorphism of $\vnMaS$, in the appropriate sense.

We now describe the structure of this paper.
In section~\ref{section:preliminaries} we review some preliminary material.
In section~\ref{section:subproduct-systems-stochastic} we describe the subproduct system
of a stochastic matrix, and in Theorem~\ref{theorem:subproduct-systems-isomorphism} we provide an effective isomorphism theorem for such objects.
In section~\ref{section:cuntz-pimsner} we review some basic facts about the
Cuntz-Pimsner algebra of a subproduct system, and in
section~\ref{section:cuntz-pimsner-stochastic} we characterize the Cuntz-Pimsner 
C*-algebra of the subproduct system of 
essential stochastic matrices. Finally, in 
section~\ref{section:tensor-algebras} we begin the study of the tensor 
algebra of a general subproduct system, 
culminating in the main results for the case
of stochastic matrices in 
section~\ref{section:tensor-algebras-stochastic}.

% *****************
% * Preliminaries *       
% *****************

\section{Preliminaries}\label{section:preliminaries}

% ****************************************
% * Preliminaries on Stochastic matrices *       
% ****************************************

\subsection*{Stochastic Matrices}
We review the basic terminology and facts about stochastic matrices that are relevant to our study.  

\begin{defi}
Let $\St$ be a countable set. A stochastic matrix is a function $P : \St \times \St \rightarrow \mathbb{R}_+$ such that for all $i\in \St$ we have
$ \sum_{j\in \St}P_{ij} = 1 $. The set $\St$ is called the \emph{state set or space} of the matrix $P$, and elements of $\St$ are called \emph{states} of $P$.
\end{defi}

Two stochastic matrices $P$ and $Q$ can be multiplied to obtain a new stochastic matrix $PQ$ defined 
by $(PQ)_{ik} = \sum_{j\in \St}P_{ij}Q_{jk}$. Henceforth, we 
will denote by $P^n$ the product of $P$ with itself $n$ times, and by $P^{(n)}_{ij}: = (P^n)_{ij}$ the $(i,j)$-th entry of $P^n$.

\begin{defi}
Let $P$ be a stochastic matrix on a state space $\St$. Denote by $Gr(P)$ the matrix representing the directed graph of $P$ which is defined to be 
$$ 
Gr(P)_{ij} = \left\{     
\begin{array}{lr}       
1, &  P_{ij} > 0 \\       
0, &  P_{ij} = 0     
\end{array}   \right.
$$
\end{defi}

\begin{defi}
Let $P$ and $Q$ be stochastic matrices on a state space $\St$, let  $\sigma :\St \rightarrow \St$
be a permutation with corresponding permutation matrix $R_{\sigma} = \big[\delta_{i}(\sigma(j))\big]$.

We say that $P$ is \emph{graph isomorphic} to $Q$ through 
$\sigma$ and write $P \sim_{\sigma} Q$ if $R_{\sigma}^{-1} Gr(Q) R_{\sigma} = Gr(P) $. We will also say that $P$ is  \emph{isomorphic} to $Q$ through $\sigma$ and write $P \cong_{\sigma} Q$ if $R_{\sigma}^{-1} Q R_{\sigma} = P$.

\end{defi}

Thus we have that $P\sim_\sigma Q$ if the directed 
graphs of $P$ and $Q$ are isomorphic,
while $P\cong_\sigma Q$ if the \emph{weighted} directed graphs of $P$ and $Q$ are isomorphic.

\begin{defi} \label{definition:paths}
Let $P$ be a stochastic matrix over a state set $\St$. A \emph{path} in $P$ is a path in the directed graph of $P$, that is a function $\gamma : \{0,...,\ell \} \rightarrow \St$ such that $P_{\gamma(k)\gamma(k+1)} > 0$ for every $0\leq k \leq \ell - 1$. The path $\gamma$ is said to be a cycle if $\gamma(0) = \gamma(\ell)$. We will say that two states $i, j$ 
\emph{communicate} if and only if there exists a path from
$i$ to $j$ and vice-versa.
\end{defi}

It is clear that the communication relation is an equivalence relation. 
Note also that a path from $i$ to $j$ of length $n$ exists if and only if $P^n_{ij} > 0$.

\begin{defi}
Let $P$ be a stochastic matrix over a state set $\St$, and let $i\in \St$.
\begin{enumerate}
\item
The \emph{period} of $i$ is $r(i) = \gcd \{ \ n \ | \ P^{(n)}_{ii} >0 \ \}$. If no such $r(i)$ exists, or if $r(i) = 1$ we say that $i$ is \emph{aperiodic}.
\item
$i\in \St$ is said to be \emph{inessential} in $P$ if there is some $j \in \St$ and $n\in \mathbb{N}$ such that $P^n_{ij} > 0$ but $P^m_{ji} = 0 $ for all $m\in \mathbb{N}$. $P$ is said to be \emph{essential} if it has no inessential states.
\item
$P$ is said to be \emph{irreducible} if any pair $i,j \in \St$ communicates in $P$.
\end{enumerate}
\end{defi}

Clearly, irreducible stochastic matrices are automatically essential. Further note that for an essential state $i\in \St$, the number $r(i)$ is always well-defined.

\begin{defi}
Let $P$ be a stochastic matrix over a state set $\St$. A state $i\in \St$ is said to be \emph{transient} if the series $\sum_{n\in \mathbb{N}}P^{(n)}_{ii}$ converges, and otherwise \emph{recurrent}. If $i \in \St$ is recurrent then it is \emph{null-recurrent} if $P^{(n)}_{ii} \rightarrow_{n \rightarrow \infty} 0$, and \emph{positive-recurrent} otherwise. We say that $P$ is transient/recurrent if all states are transient/recurrent in $P$, respectively.
\end{defi}

In an irreducible stochastic matrix, all states have 
the same period and classification in terms of recurrence type.  
We also note that recurrent states are essential (see
 \cite[Part I, Section 4, Theorem 4]{Chung}). 

The next two theorems can be found in various forms in the literature, see \cite[Part I, Section 3]{Chung}, \cite[Chapter XV, Section 6]{Feller}. We restate them in a form convenient for our purposes.

\begin{theorem} \label{theorem:graph-decomposition}
(Irreducible decomposition for essential matrices) \\
Let $P$ be an essential stochastic matrix over a state set $\St$. Let 
$(\Omega_\alpha)_{\alpha \in A}$ be the partition of $\Omega$ into
equivalence classes of communicating states. 
With the appropriate enumeration of $\Omega$, the
matrix $P$ decomposes into a block diagonal matrix whose diagonal
blocks are irreducible stochastic matrices corresponding to the restriction
of $P$ to $\Omega_\alpha \times \Omega_\alpha$, for $\alpha \in A$.
%
%$ P(1),P(2),... $ irreducible stochastic matrices such that $P$ is isomorphic %
%(via some permutation of $\St$) to
%$$ 
%		\begin{bmatrix}
%       P(1) & 0 & 0 & \cdots \\
%       0 & P(2) & 0 & \cdots \\
%       0 & 0 & P(3) & \cdots \\
%       \vdots & \vdots & \vdots & \ddots  \\
%     \end{bmatrix}  
%$$
\end{theorem}

As a corollary, we observe that complete reducibility is equivalent to essentiality.

\begin{theorem} \label{theorem:cyclic-graph-decomposition}
(Cyclic decomposition for periodic irreducible matrices)\\
Let $P$ be an irreducible stochastic matrix over a state set $\St$ with period $r$, and let $\omega \in\Omega$.  For each $\ell=0, \dots r-1$, let
$\Omega_\ell = \{ j \in \Omega \mid P_{\omega j}^{(n)}>0 \implies  n\equiv \ell 
\mod r \}$. Then the family $(\Omega_\ell)_{\ell=0}^{r-1}$ is a partition of 
$\Omega$. Furthermore if $j\in \Omega_{\ell}$ then there exists $N(j)$ such that for all $n\geq N(j)$ we have $P^{(nr + \ell)}_{\omega j}>0$. In fact,
with an appropriate enumeration of $\Omega$, there exist stochastic matrices $P_0,...P_{r-1}$ such that $P$ has the 
following cyclic block decomposition:
\begin{equation*}
\left[\begin{smallmatrix}
       0 & P_0  & \cdots & 0 \\
       \vdots  & \ddots  & \ddots & \vdots  \\
       0 &  \cdots & 0 & P_{r-2} \\
       P_{r-1}  & \cdots & 0 & 0
     \end{smallmatrix} \right]
\end{equation*}
Where the rows of $P_{\ell}$ in this matrix decomposition are indexed by $\Omega_{\ell}$ for all $0 \leq \ell < r$.
\end{theorem}

\begin{remark}
We emphasize that the columns of $P_{\ell}$ in the above matrix decomposition are indexed by $\Omega_{s(\ell)}$, where $s(\ell) = \ell+1 (\mod r)$. Further, if $P$ is finite $d\times d$, then $P_0, ... ,P_{r-1}$ are $\frac{d}{r} \times \frac{d}{r}$.
\end{remark}

We note that the cyclic decomposition is independent of our initial choice of $\omega$ in the sense that
for $P$ irreducible with period $r$ and cyclic decomposition $\St_0,..., \St_{r-1}$ given
by the previous theorem, if we were to pick a different $\omega' \in \Omega$ and a new partition $\Omega'_0,... \Omega'_{r-1}$ induced by it, a cyclic permutation of this partition would yield our original partition $\Omega_0,...,\Omega_{r-1}$. Furthermore, if $i\in \St_{l_1}$ and $j\in \St_{l_2}$ are two states, let $0 \leq \ell < r$ be such that
$\ell \equiv l_2 - l_1  \mod r$. Now if $P^{(k)}_{ij} > 0$ then one must have $k=mr + \ell$ for some $m\in \mathbb{N}$, and there exists some $m_0\in \mathbb{N}$ such that for all $m \geq m_0$ one has $P^{(mr+\ell)}_{ij} > 0$. (See \cite[Part I, Section 3, Theorems 3 \& 4]{Chung})

Recall that an irreducible stochastic $P$ is positive-recurrent if and only if it possesses a stationary distribution, i.e.
 a vector $\pi \in \ell^1(\St)$ 
such that $\pi_j \geq 0$ for all
$j\in\St$, $\sum_i \pi_i = 1$ and $\pi_j = \sum_i \pi_i P_{ij}$ for all $j\in \St$.
For the proof of the following theorem, see 
\cite[Theorem 6.7.2]{DUR}. 

\begin{theorem} \label{theorem:convergence-theorem-positive-recurrent}
(Convergence theorem for positive-recurrent irreducible matrices) \\
Let $P$ be an irreducible positive-recurrent stochastic matrix with stationary distribution $\pi$ and period $r\geq 1$. Let $\St_0,...,\St_{r-1}$ be a cyclic decomposition of $\Omega$ with respect to $P$ as in Theorem~\ref{theorem:cyclic-graph-decomposition}. Given $i\in \St_{l_{1}}$ and $j\in \St_{l_{2}}$, let $0 \leq \ell < r$ be
such that $\ell \equiv (l_2-l_1) \mod r$. Then
$$ \lim_{m\rightarrow \infty}P^{(mr+\ell)}_{ij} = \pi_jr.$$
\end{theorem}

\begin{defi}
Let $P$ be a stochastic matrix over $\St$. $i\in \St$ is said to be \emph{amenable} in $P$ if $\limsup_{n\rightarrow \infty} \sqrt[n]{P_{ii}^{(n)}} = 1$.
$P$ is said to be amenable if all states in $\St$ are amenable in $P$.
\end{defi}

The proof of the following fact is well-known and will be omitted.

\begin{lemma} \label{lemma-rec-iterate}
Let $P$ be a recurrent irreducible stochastic matrix, then $P$ is amenable and for every $n\in \mathbb{N}$ we have that $P^n$ is recurrent.
\end{lemma}

% **************************************************************
% * Preliminaries on Hilbert W*-modules and subproduct systems *      
% **************************************************************

\subsection*{Hilbert modules} 
We assume that the reader is familiar with the basic theory of Hilbert C*-modules, which can be found in \cite{Pasc, Lance, Manu}.
We only give a quick summary of basic notions and terminology 
to clarify our conventions.

As usual inner product modules are right modules. 
The \emph{dual module} of an inner product module $E$ over $\mathcal{A}$ is the set of all bounded $\mathcal{A}$-module maps from $E$ to $\mathcal{A}$, and it is denoted by $E'$.
A Hilbert C*-module is called \emph{self-dual} if the canonical embedding of $E$ into $E'$ is surjective.  

\begin{defi}
Let $\vnM$ be a W*-algebra let $E$ be a Hilbert C*-module over $\vnM$. The \emph{$\sigma$-topology} on $E$ is defined by the functionals
$f(\cdot) = \sum_{n=1}^{\infty}w_n(\la \xi_n , \cdot \ra) $
where $\xi_n \in E$ and $w_n \in \vnM_{*}$ such that $\sum_{n=1}^\infty ||w_n|| \, ||\xi_n|| < \infty$.
\end{defi}

Hilbert C*-modules over a W*-algebra will be called Hilbert W*-modules. If $E$ is a self-dual W*-module, then it is a dual 
Banach space (see \cite{Pasc}), and the associated weak-* topology  
coincides with its $\sigma$-topology. 

If $E$ is an inner product module over a W*-algebra $\vnM$, then the inner product module structure of $E$
can be naturally extended to $E'$, which makes $E'$ into
a self-dual Hilbert W*-module over $\vnM$. In this case we will refer to $E'$ as the \emph{self-dual extension of $E$}.
Furthermore, 
the canonical embedding of $E$ into $E'$ maps onto a dense subset
in the $\sigma$-topology of $E'$, so that $E$ is self dual if and only if it is $\sigma$-topology closed in $E'$.

For $E$ and $F$ Hilbert W*-modules, let $\mathcal{L}(E,F)$ denote the set of adjointable maps from $E$ to $F$. If $E$ and $F$ are \emph{self-dual} W*-modules, then all bounded $\vnM$-module maps from $E$ to $F$ are adjointable. It also 
turns out that bounded module maps between any two inner product 
modules over a W*-algebra behave well with respect to self dual 
completions, as the following proposition states.

\begin{proposition} \label{proposition:extension-self-dual}
Let $E$ and $F$ be inner-product modules over a W*-algebra $\vnM$, and let $T:E\rightarrow F$ be a bounded module map. Then $T$ has a unique extension to a bounded module map $\tilde{T} : E' \rightarrow F'$, and $||\tilde{T}|| = ||T||$. If $E=F$, then the map $T \rightarrow \tilde{T}$, restricted to the algebra of adjointable operators on $E$, is a faithful *-homomorphism from $\mathcal{L}(E)$ to $\mathcal{L}(E')$.
\end{proposition}

\begin{defi}
If $E$ is a Hilbert C*-module over $\mathcal{B}$, and $\mathcal{A}$ is another C*-algebra, then $E$ is called a \emph{C*-correspondence} from $\mathcal{A}$ to $\mathcal{B}$ if $E$ is also a left $\mathcal{A}$-module such that the left action is determined by a *-homomorphism $\phi : \mathcal{A} \rightarrow \mathcal{L}(E)$. In the case when $\mathcal{A}=\mathcal{B}$, $E$ is called a \emph{C*-correspondence} over $\mathcal{A}$. If $\vnN$ and $\vnM$ are W*-algebras then $E$ is called a Hilbert W*-correspondence from $\vnN$ to $\vnM$ if in addition $E$ is a self-dual module over $\vnM$ and the left action $\phi$ of $\vnN$ is normal.
\end{defi}

A key notion of C* and W*-correspondences is the internal tensor product. If $E$ is a C*-correspondence from $\mathcal{A}$ to $\mathcal{B}$ with left action $\phi$, and $F$ is a C*-correspondence from $\mathcal{B}$ to $\mathcal{C}$ with left action $\psi$, then on the algebraic tensor product $E\otimes_{alg}F$ one defines a $\mathcal{C}$-valued pre-inner product satisfying $\la x_1 \otimes y_1 , x_2 \otimes y_2 \ra = \la y_1 , \psi(\la x_1 , x_2 \ra) y_2 \ra$ on simple tensors. The usual quotient and completion process yields the internal Hilbert C*-module tensor product of $E$ and $F$, denoted by $E\overline{\otimes}F$ or $E\overline{\otimes}_{\psi}F$, which is a C*-correspondence from $\mathcal{A}$ to $\mathcal{C}$. In the case where $\mathcal{A}$, $\mathcal{B}$ and $\mathcal{C}$ are W*-algebras, taking the self dual completion yields the self-dual tensor product denoted also by $E\otimes F$ or $E\otimes_{\psi}F$, which is a W*-correspondence from $\mathcal{A}$ to $\mathcal{C}$.

The notion of self-dual direct sums of Hilbert C*-modules over W*-algebras was developed by Paschke.

\begin{defi}
Let $\{ E_i \}_{i\in I}$ be a family of self-dual W*-modules over $\vnM$. The ultraweak direct sum of $\{ E_i \}_{i\in I}$ is the subset $X$ of the Cartesian product of $\{ E_i \}_{i\in I}$ such that $\{ x_i \}$ is in $X$ if the sum $\sum_{i\in I}\la x_i , x_i \ra$ converges ultraweakly. The inner product on $X$ is defined to be $\la \{x_i\} , \{y_i\} \ra = \sum_{i\in I}\la x_i , y_i \ra$, where the sum converges ultraweakly in $\vnM$. This direct sum is denoted by $\oplus^{uw}_{i}E_i$ or by $\oplus_{i}E_i$ when the context is that of self-dual modules. 
\end{defi}

\subsection*{Base change}
Suppose that $E$ is a Hilbert 
C*-module (correspondence) over $\mathcal{A}$
and $\rho$ is 
a *-automorphism of $\mathcal{A}$. Then one can define
a new C*-module (correspondence) $E^\rho$ over $\mathcal{A}$. As a set, $E^\rho = E$, but  its operations
are defined as follows: 
 $ \xi \cdot_\rho a =  \xi \cdot \rho(a) $
(also $a \cdot_{\rho} \xi = \rho(a) \cdot \xi$ for correspondences) and $\la \xi,\eta \ra_{\rho} = \rho^{-1}(\la \xi,\eta \ra)$ for all $\xi, \eta \in E$ and $a\in \mathcal{A}$. 
In algebra, this operation on modules
is sometimes called a change of rings or base change.

\begin{defi}
Let $E$ and $F$ be two C*-modules over $\mathcal{A}$, and
let $\rho : \mathcal{A} \rightarrow \mathcal{A}$ be a *-
automorphism. We will say that a bounded linear map $V: E \to F$ 
is a \emph{$\rho$-module ($\rho$-correspondence) morphism} if 
$V: E \to F^\rho$ is 
an $\mathcal{A}$-linear ($\mathcal{A}$-correspondence)
map, i.e. for all $a \in \mathcal{A}$ and $\xi \in E$ one has 
$V(\xi a) = V(\xi) \rho(a)$ (also $V(a \xi) = \rho(a) V(\xi)$ for 
correspondences). We will say that $V$ is
\emph{$\rho$-adjointable} if $V$ is adjointable as a map 
$E \to F^\rho$,
and we will denote by $V^{(*,\rho)}$ its adjoint, i.e. 
$V^{(*,\rho)}:F \rightarrow E$ satisfies
$\la \xi , V(\eta) \ra_{\rho} =
\rho^{-1} (\la \xi ,V(\eta) \ra) = \la V^{(*,\rho)}(\xi), \eta \ra$
 for all $\xi \in F, \eta \in E$.
\end{defi}

 Note that a $\rho$-adjointable $V$ 
must be a bounded $\rho$-module morphism and $V^{(*,\rho)}$ 
is a $\rho^{-1}$-adjointable map with 
$(V^{(*,\rho)})^{(*,\rho^{-1})} = V$. Furthermore, if $E$, $F$ and $G$ are C*-modules over $\mathcal{A}$ and we 
have $V: E\rightarrow F$ a $\rho$-module/correspondence 
morphism and $W: F \rightarrow G$ a $\tau$-module/correspondence morphism, then $W\circ V : E 
\rightarrow G$ is a $(\tau \circ \rho)$-
module/correspondence morphism respectively. Further, if $V$ 
is $\rho$-adjointable and $W$ is $\tau$-adjointable, then 
$W\circ V$ is $(\tau \circ \rho)$-adjointable with $(W\circ 
V)^{(*, \tau \circ \rho)} = V^{(*,\rho)} \circ W^{(*,\tau)}$.

Furthermore, $\sigma$-topology continuity is automatic for $\rho$-adjointable maps. Indeed, suppose that $V : E \rightarrow F$ is a $\rho$-adjointable map, and let 
$\{ \eta_{\alpha} \}_{\alpha}$ be a net in $E$ converging in the $\sigma$-topology to $\eta$. Let $\xi_n \in F$ and $w_n \in \vnM_*$ be such that $\sum||w_n|| \cdot ||\xi_n|| < \infty$. Then we have that
$$
\sum_{n=1}^{\infty}w_n ( \la \xi_n, V(\eta_{\alpha} - \eta) \ra) = \sum_{n=1}^{\infty}(w_n \circ \rho)( \la V^{(*,\rho)}(\xi_n), \eta_{\alpha} - \eta \ra)
$$
And since $\sum ||w_n \circ \rho || \cdot ||V^{\rho}(\xi_n)|| \leq ||V^{\rho}|| \sum||w_n|| \cdot ||\xi_n|| < \infty$, we have convergence  of the net $(V\eta_\alpha)$ 
to $V\eta$ in the $\sigma$-topology in $F$.

We also note that the identity map $\iota_\rho: (E^\rho)'\to 
E'$ is a 
$\rho$-module isometric isomorphism.  It follows that if $E$ is self-dual, 
then the same holds for $E^\rho$. It also leads to the 
following fact.

\begin{proposition} \label{proposition:extension-rho-self-dual}
Let $E$ and $F$ be Hilbert W*-modules over a W*-algebra $\vnM$,
and let $\rho$ be a *-automorphism of $\vnM$. 
Suppose that $V: E\rightarrow F$ is a $\rho$-module morphism. Then $V$ has a unique extension to a bounded $\rho$-module morphism $\overline{V}: E'\rightarrow F'$, and $||V|| = ||\overline{V}||$.
\end{proposition}

\begin{proof}
Using Proposition \ref{proposition:extension-self-dual} we 
obtain a unique W*-module map $\tilde{V} : E' \rightarrow 
(F^{\rho})'$. By composing with $\iota_\rho$ described
in the  paragraph preceding the propostion, we obtain
that $\overline{V} = \iota_\rho\circ \tilde{V}$ is a $\rho$-module morphism from $E'$ to $F'$ extending $V$. Since $\iota_\rho$
is a bijection, we also obtain uniqueness of the extension.
The norm condition holds because $\iota_\rho$ is isometric.
\end{proof}

\begin{defi} \label{definition:rho-coisometric-unitary}
Let $E$ and $F$ be Hilbert W*-modules over $\vnM$ and let 
$\rho :\vnM \rightarrow \vnM$ be a *-automorphism. We say that a $\rho$-adjointable map $W: E\rightarrow F$ is a \emph{$\rho$-coisometry} if $W^{(*,\rho)}$ is an isometry,
and we will say that $W$ is a $\rho$-unitary if it it is an isometric surjective $\rho$-module morphism.
\end{defi}

We observe that if  $U$ is a  $\rho$-unitary then $U^{-1}$ 
is a $\rho^{-1}$-module morphism and $U^{(*,\rho)} = 
U^{-1}$. Further note that a $\rho$-module map $U$ is a 
$\rho$-unitary if and only if it is a $\rho$-adjointable 
module map satisfying $UU^{(*,\rho)} = Id_{F}$ and $U^{(*,
\rho)}U = Id_{E}$. This follows from the analogous and 
well-known theorem for ($Id$-)unitaries, and the preceding 
discussion.

\subsection*{Subproduct systems}

\begin{defi}
Let $\vnM$ be a W* algebra, let
$X=\{X_n\}_{n\in \nn}$ be a family of 
of Hilbert W*-correspondences over $\vnM$ and let 
$U = \{U_{n,m}:X_n \otimes X_m \rightarrow X_{n+m}\}_{n,m \in\nn}$ 
be a family of maps. 
We will say that  $(X,U)$ is a \emph{subproduct system  over $\vnM$}
(and over the semigroup $\mathbb{N}$) if the following conditions are satisfied:
\begin{enumerate}
\item 
$X_0 = \vnM$
\item $ U_{n,m}$ is a coisometric mapping of W*-correspondences over $\vnM$
for every $n,m \in \nn$

\item The maps $U_{0,n}$ and $U_{n,0}$ are given by the left and right actions of $\vnM$ on $X_n$ respectively, and for every
$n,m \in \nn$ we have the following identity:
$$ U_{n+m,p}(U_{n,m} \otimes I_{X_p}) = U_{n,m+p}(I_{X_n}\otimes U_{m,p}) $$
\end{enumerate}
In case the maps $U_{n,m}$ are unitaries, we say that $X$ is a \emph{product system}.
\end{defi}

When there is no ambiguity, we will suppress the reference
to the family $U$ of multiplication maps, and refer 
simply to a subproduct system $X$.

\begin{example} \label{regular-product-system}
If $E$ is a W*-correspondence over $\vnM$ such that $\vnM \cdot E = E$ (essential), the \emph{product} system $X^E$ over $\mathbb{N}$ defined by $X^E_n = E^{\otimes n}$ where $U_{n,m}$ is the natural identification between $E^{\otimes n}\otimes E^{\otimes m}$ and $E^{\otimes (n+m)}$, is obviously a subproduct system. We call it the \emph{full} product system.
\end{example}

\begin{defi} \label{definition:subproduct-direct-sum}
Let $(X,U^X)$ and $(Y,U^Y)$ be subproduct systems over $\mathcal{M}$ and $\mathcal{N}$ respectively. We define a subproduct system $X\oplus Y$ over $\mathcal{M} \oplus \mathcal{N}$, which
we will call the direct sum of $X$ and $Y$.
\begin{enumerate}
\item
The $n$-th fiber W*-correspondence is given by $(X\oplus Y )_n := X_n \oplus Y_n$ with left and right multiplication of $\mathcal{M} \oplus \mathcal{N}$ given by $(m\oplus n) \cdot  (\xi \oplus \eta) \cdot (m' \oplus n') = m\xi m' \oplus n\eta n'$, and inner product given by 
$ \la \xi \oplus \eta , \xi' \oplus \eta' \ra = \la \xi , \xi' \ra \oplus \la \eta, \eta' \ra$.
\item
The subproduct maps $U^{X\oplus Y}_{n,m}$ are defined by $$U^{X\oplus Y}_{n,m}((\xi_n \oplus \eta_n) \otimes (\xi_m \oplus \eta_m)) = U^X_{n,m}(\xi_n \otimes \eta_n) \oplus U^Y_{n,m}(\xi_m \otimes \eta_m)$$
\end{enumerate}
\end{defi}

\begin{defi}
Let $(X,U^X)$ and $(Y,U^Y)$ be two subproduct systems over a W*-algebra $\vnM$. A family $V=\{V_{n}\}$ of maps $V_n : X_n \rightarrow Y_n$ is called a \emph{morphism} of subproduct systems if 

\begin{enumerate}
\item
$\rho = V_0: X_0 \rightarrow Y_0$ is a *-automorphism, 

\item
For all $n \neq 0$ the map $V_n$ is a $\rho$-coisometric $\rho$-correspondence morphism.

\item
For all $n,m \in \nn$ the following identity hold:
$$ V_{n+m} \circ U^X_{n,m} = U^Y_{n,m} \circ (V_n \otimes V_m) $$
\end{enumerate}

When the family $\{ V_n \}$ is a family of $\rho$-unitaries, we say that $X$ and $Y$ are (unitarily) isomorphic via $\rho = V_0$ and write $X \cong_{\rho} Y$.

\end{defi}

In \cite{SS} the notion of isomorphism of subproduct systems was defined so that the map $V_0$ is the identity on $\vnM$, yet the above variation will be more convenient for our purposes.

We now describe the general construction of \AS\ subproduct systems.

Let $\vnM$ be a W*-algebra, and let 
$\theta$ be a completely positive contractive and normal map on $\vnM$. If $\rho : \vnM \rightarrow B(H)$ is a normal representation, then we can define a new normal representation
of $\vnM$, which we will call a dilation of $\theta$ via $\rho$, as follows. We define $\vnM \otimes_\theta H$ to 
be the
Hausdorff completion of the algebraic tensor product $\vnM \otimes H$ with respect to the sesquilinear form defined
by
$$ 
\la T_1\otimes h_1 , T_2 \otimes h_2 \ra = \la h_1 , \rho(\theta(T_1^*T_2))h_2 \ra , \qquad \text{for all } T_1, T_2 \in \vnM, h_1, h_2 \in H.
$$
Complete positivity of $\theta$ ensures that this sesquilinear
form is positive semidefinite. We define a normal representation
$\pi_\theta$ 
of $\vnM$ on $\vnM \otimes_\theta H$ by 
$\pi_\theta(S) (T \otimes h)
= ST \otimes h$.  Moreover, the map
$W_\theta: H \to \vnM \otimes H$ given by $W_\theta(h) = I \otimes h$ is a contraction which satisfies
$$
\rho(\theta(T)) = W_{\theta}^*\pi_{\theta}(T)W_{\theta}.
$$
One easily checks that $\vnM \otimes_{\theta} H$ is minimal in the sense that it is the smallest subspace of $\vnM \otimes_{\theta} H$ containing $W_{\theta}H$ and reducing $\pi_{\theta}$. It follows that if $(\pi, K , W)$
is a triple where $\pi$ is a representation of $\vnM$ on 
$K$ such that $\rho(\theta(T)) = W^*\pi(T)W$ for all $T\in \vnM$ and $K$ is minimal, then there is a unitary $U: K \to \vnM \otimes_\theta H$
such that $U$ implements a unitary equivalence between $\pi$ and $\pi_\theta$ and 
$UW = W_\theta U$.  We also define the associated
intertwiner space 
$$
\mathcal{L}_{\vnM}(H, \vnM \otimes_{\theta}H) : = \{ \ X \in B(H,\vnM \otimes_{\theta}H) \ | \ X\rho(T) = \pi_{\theta}(T)X \ \ \forall T\in \vnM  \ \}.
$$

Suppose now that $\vnM$ is a von Neumann subalgebra of 
$B(H)$ and $\rho$ is the inclusion map. In this case 
the triple $(\pi_\theta, \vnM \otimes_\theta H, W_\theta)$
will be called the minimal Stinespring dilation of $\theta$. 
In \cite{MS}, the intertwiner space
$ \mathcal{L}_{\vnM}(H, \vnM \otimes_{\theta}H)$
was shown to be a W*-correspondence over $\vnM'$, 
with left and right actions given by $S\cdot X = (I\otimes S) 
\circ X$ and $X\cdot S = X\circ S$ for $S \in \vnM '$ and $X 
\in \mathcal{L}_{\vnM}(H, \vnM \otimes_{\theta}H)$, and the 
$\vnM '$-valued inner product is given by $\la X , Y \ra = 
X^*Y$, for $X,Y \in \mathcal{L}_{\vnM}(H, \vnM 
\otimes_{\theta}H)$. 

Now let $\phi$ be another completely positive contractive normal map on $\vnM$. Then we obtain a representation
$\pi_{\theta, \phi}$ of $\vnM$ on  $\vnM \otimes_{\phi} (\vnM \otimes_{\theta}H)$ via dilation using $\pi_\theta$, and
in an analogous way we have that
the $\vnM'$-correspondence intertwiner space 
$\mathcal{L}_{\vnM}(H, \vnM \otimes_{\phi} (\vnM \otimes_{\theta}H))$,
between the identity representation of $\vnM$ and $\pi_{\theta,\phi}$.

\begin{defi} \label{subproduct-prelim}
Let $\theta$ be a normal completely positive map on a W*-algebra $\vnM$. Let 
$$
\mathcal{L}^\theta(n) = \mathcal{L}_{\vnM}(H, \vnM \otimes_{\theta^n}H) \quad \text{and} \quad
\mathcal{L}^\theta(n,m)  = \mathcal{L}_{\vnM}(H, \vnM \otimes _{\theta^n} \vnM \otimes_{\theta^m} H), \qquad n,m \in\nn.
$$
The \AS\ subproduct system of $\theta$ is defined as follows:
\begin{enumerate}
\item
The fibers are $(\mathcal{L}^\theta(n))_{n \in\nn}$ with the aforementioned 
W*-correspondence structure over $\vnM'$.
\item 
The subproduct maps $U^{\theta}_{n,m}: \mathcal{L}^\theta(n) \otimes \mathcal{L}^\theta(m) \rightarrow \mathcal{L}^\theta(n+m)$ are defined by $U^{\theta}_{n,m} = V_{n,m}^*\Psi_{n,m}$ where,
\begin{enumerate}
\item $V_{n,m} : \mathcal{L}^\theta(n+m) \to 
\mathcal{L}^\theta(m,n)$
is defined by $V_{n,m}(X) = \Gamma_{n,m} \circ X$ where $\Gamma_{n,m} : \vnM \otimes_{\theta_{n+m}}H \rightarrow \vnM \otimes_{\theta_m} (\vnM \otimes_{\theta_n}H)$ is defined to be $\Gamma_{n,m}(S\otimes_{n+m} h) = S \otimes_m I \otimes_n h$.
\item
$\Psi_{n,m}: 
\mathcal{L}^\theta(n) \otimes \mathcal{L}^\theta(m)
\to \mathcal{L}^\theta(m,n)$
is defined by $\Psi_{n,m}(X\otimes Y) = (I \otimes X)Y$.
\end{enumerate}
\end{enumerate}
We denote this subproduct system by $(\mathcal{L}^{\theta}, U^{\theta})$.
\end{defi}

It was proven in \cite{SS} that this is indeed a subproduct system. This fact relies on the work in \cite{MS},
where it was shown that $V_{n,m}$ is an isometric correspondence morphism, and that $\Psi_{n,m}$ is a correspondence isomorphism.

A fundamental property of \AS\ subproduct systems is that every subproduct system is (Id-)isomorphic to an \AS\ subproduct system of a cp-semigroup (see \cite[Corollary 2.10]{SS}).

% *********************************************************
% * Subproduct systems arising from stochastic matrices   * 
% *********************************************************

\section{Subproduct systems arising from stochastic matrices}\label{section:subproduct-systems-stochastic}

Let $\St$ be a countable set, and let $\vnMaS$ be the von Neumann algebra of bounded sequences indexed by $\St$ acting on the Hilbert space 
$\ell^2(\Omega)$. Let us denote by $\{e_i\}_{i \in \St}$ the canonical 
orthogonal basis for $\ell^{2}(\St)$, and by $\{p_j\}_{j\in \St}$ the collection 
of rank one pairwise perpendicular projections in $\ell^{\infty}(\St)$ defined by $p_j(e_i) = \delta_{ij}e_i$.

In this section we will compute the \AS\ subproduct system of the $cp$-semigroup generated by a single unital positive normal map on the von-Neumann algebra $\vnMaS$. It is easy to see that such a map is determined uniquely by a stochastic matrix on $\Omega$.
This simple observation will be used repeatedly, hence we record it here for
emphasis. We omit the straightfoward proof.

\begin{proposition} \label{corresp-cp-stoch}
There is a 1-1 correspondence between unital positive normal maps $\theta : \vnMaS \rightarrow \vnMaS$ and stochastic matrices $P$ over $\Omega$, where the relationship is given by 
$$\langle e_i , \theta(p_j)e_i \rangle = P_{ij}$$
The map just defined sends the composition of unital positive normal maps into the product of their respective stochastic matrices.
\end{proposition}

Of course representations of stochastic matrices on $\Omega$ are 
dependent on its enumerations, hence permutations of $\Omega$, 
or alternatively  *-automorphisms of $\vnMaS$,
 will play a role in the continuation.   Recall that the *-automorphisms of 
$\vnMaS$ are in 1-1 correspondence with the permutations of $\St$, 
because minimal projections must be sent to minimal projections
via *-automorphisms. We can associate to every permutation  $\sigma : \St \rightarrow \St$  the automorphism $\rho_{\sigma} : \vnMaS   \rightarrow \vnMaS$  given by $\rho_\sigma(f) = f \circ \sigma^{-1}$. The inverse map is obtained
as follows: if $\rho$ is a *-automorphism, then 
for every $j \in \Omega$ there exists a unique
 $\sigma_\rho(j) \in \Omega$
such that $\rho(p_j) = p_{\sigma_\rho(j)}$ and $\sigma_\rho$ is a well-defined
permutation.

\begin{notation}
We denote by $*$ the Schur(entrywise) multiplication of matrices $A = [a_{ij}]$ and $B = [b_{lk}]$ given by $A * B =[a_{ij}b_{ij}]$, and let $\Diag$ be the map on matrices
given by $\Diag([a_{ij}]) = [\delta_{ij} a_{ij}]$.
\end{notation}

\begin{notation}
Let $P$ and $Q$ be non-negative matrices indexed by $\St$. Define the set $E(P) : = \{ \ (i,j) \ | \ P_{ij}>0 \ \}$ which is the collection of edges in the weighted directed graph defined by $P$, and the set $E(P,Q) : = \{ \ (i,j,k) \ | \ P_{ij}Q_{jk} > 0\}$.  Also denote by $\sqrt{P}$ and $P^{\flat}$ the matrices with $(i,k)$-th entry given by
$$ 
(\sqrt{P})_{ik} := \sqrt{P_{ik}}, 
\qquad \text{and} \qquad
(P^{\flat})_{ik} := 
\begin{cases} 
  (P_{ik})^{-1},  & \mbox{if }(i,k) \in E(P) \\
  0, & \mbox{else} 
\end{cases}
$$
\end{notation}

\begin{theorem}\label{theorem:subproduct-computation}
Let $P$ be a stochastic matrix over a state space $\St$, and 
let $\theta$ be the unital positive normal map associated to 
$P$ by the previous proposition. The \AS\ subproduct system 
associated to $\theta$ is naturally (Id-)isomorphic to the 
following subproduct system, which 
will be denoted by $Arv(\theta)$ or $Arv(P)$:
\begin{enumerate}
\item
The $n$-th fiber is a W*-correspondence over 
$\vnMaS$, given by 
$$
Arv(P)_n = \{ \ [a_{ij}] \ | \ a_{ij}= 0 \ \forall \ (i,j)\notin E(P^n) \ , \ \sup_{j\in \St} \sum_{i\in \St}|a_{ij}|^2 < \infty \ \}
$$
where left and right actions are given  
by multiplication as diagonal matrices.
Given $A, B \in Arv(P)_n$, their W*-correspondence 
inner-product is given by
$$
\la A, B \ra = \Diag(A^* B).
$$

\item
The subproduct maps are given by
$$U_{n,m}(A \otimes B) = (\sqrt{P^{n+m}})^{\flat}*\big[(\sqrt{P^n} * A)\cdot(\sqrt{P^m} * B)\big] $$
for $n \neq 0$ and $m \neq 0$, $A\in Arv(P)_n$ and $B\in Arv(P)_m$. The maps $U_{0,n}$ and $U_{m,0}$ are given by left and right multiplication by elements of $\vnMaS$ respectively, considered as diagonal matrices.
\end{enumerate}
We call this presentation of $Arv(P)$ the standard presentation of the \AS\ subproduct system associated to the stochastic matrix $P$.
\end{theorem}

\begin{proof}
For the computation of the $n$-th fibers, we fix an $n\in \nn$. We will follow the notation and construction described in Definition~\ref{subproduct-prelim}. For convenience
we will write $\LL(n)$ instead of $\LL^\theta(n)$.

%  *****************
%  * Orthogonality *
%  *****************

Let us denote $H = \ell^{2}(\St)$, and consider
 the canonical inclusion of $\vnMaS$ into $B(\ell^2(\St))$. Notice that
the set $\{ p_j \otimes e_i \}_{(i,j) \in E(P^n)}$ constitutes
an orthogonal set in $ \vnMaS \otimes_{\theta^n}H$
since for $i,j, k, \ell \in \nn$,
$$ 
\langle p_k\otimes e_\ell , p_j\otimes e_i \rangle =
\langle e_\ell , \theta^n(p_k^*p_j)e_i \rangle = 
 \delta_{kj} \delta_{i\ell}\; \langle e_i , \theta^n(p_j)e_i \rangle = \delta_{kj} \delta_{i\ell}\; P^{(n)}_{ij}.
$$
Furthermore, it is straightforward to check that $\{ p_j \otimes e_i \}_{(i,j)\in E(P^n)}$ is in fact an orthogonal \emph{basis} for $\vnMaS \otimes_{\theta^n}H$.

%  ****************************
%  * Subproduct system fibers *
%  ****************************

We now show that $\mathcal{L}(n)$ and $Arv(P)_n$
are isomorphic as correspondences over $\vnM$.  Let $X\in \mathcal{L}(n)$. Then
there exist unique scalars $(a_{ijk})_{i,j,k\in\St}$ such that 
$a_{ijk} = 0$ for $(i,j) \not\in E(P^n)$ and 
for all $k\in \St$,
$$ 
X(e_k) = \sum_{(i,j)\in E(P^n)}a_{ijk}p_j\otimes e_i. 
$$
Since  $X \in \LL(n)$, it is a continuous linear map satisfying 
$(T\otimes I) X = X T$ for all $T \in \vnMaS$ (see definition \ref{subproduct-prelim}). On the other hand, if $T = (c_k)_{k\in \St}\in \vnMaS$, then
\begin{align*} 
XT(e_k) & =  \sum_{(i,j)\in E(P^n)}a_{ijk}c_kp_j\otimes e_i \\
 (T\otimes I)X(e_k) & =  \sum_{(i,j)\in E(P^n)}a_{ijk}c_jp_j\otimes e_i
\end{align*}
Thus by uniqueness of representation, we must have for $j \neq k$ that $a_{ijk}=0$.  Hence, if we define $a_{ij} :=a_{ijj}$, and denote $A = [a_{ij}]$, we obtain,
$$
X(e_j) = \sum_{i \ : \ (i,j) \in E(P^n)} a_{ij} p_j \otimes e_i = p_j\otimes Ae_j
$$
where $a_{ij} = 0$ for $(i,j) \notin E(P^n)$, and $Ae_j = (a_{ij})_{i\in \St} \in \ell^2(\St)$ for each $j\in \St$.

The boundedness condition on $X$ ensures that 
$$ ||X|| \geq ||X(e_j)|| = || p_j \otimes Ae_j|| $$
Since $X(e_j) \perp X(e_{j'})$ for all $j\neq j'$, we also have by Pythagoras that for every $b  \in \ell^2(\St)$:
$$ ||X(b)||^2 = ||\sum_{j\in \St}b_jX(e_j) ||^2 = \sum_{j\in \St}|b_j|^2||X(e_j)||^2 \leq ||b||_2^2 \;\sup_{j\in \St}||p_j \otimes Ae_j||^2 $$
Thus $||X|| = \sup_{j\in \St}||p_j\otimes Ae_j||$ and,
$$ ||p_j\otimes Ae_j||^2 = \la p_j\otimes Ae_j,p_j\otimes Ae_j \ra = \sum_{i\in \St}|a_{ij}|^2\la e_i,\theta^n(p_j)e_i \ra = \sum_{i\in \St}|a_{ij}|^2P^{(n)}_{ij}.
$$

In this way each $X\in \mathcal{L}(n)$ has a unique matrix $A=[a_{ij}]$ in the set $\ee(n)$ of matrices indexed
by $\Omega$  satisfying  $a_{ij} = 0$ for $(i,j) \notin E(P^n)$ and $\sup_j\sum_i|a_{ij}|^2P^{(n)}_{ij} < \infty$. Conversely, it is easy to see that any matrix in $\ee(n)$ is obtained in this fashion, and this implements 
a bijection of $\ee(n)$ with $\mathcal{L}(n)$.  
In the remainder, we will denote by
$X_A = X_{[a_{ij}]} = X^{(n)}_A = X^{(n)}_{[a_{ij}]}$ the unique element associated to
$A \in \ee(P)_n$  determined by 
the identity
$$
X_A(e_k) = p_k \otimes A e_k, \qquad k\in\nn.
$$

A brief computation shows that the left and right actions are given by $T \cdot X_{A} = (I\otimes T) \circ X_{A} = X_{T \cdot A} $ and $X_{A} \cdot T = X_{A \cdot T}$, where $T\in \vnMaS$ is thought of as a diagonal matrix when multiplied with the matrix $A$. Furthermore, given $A, B \in Arv(P)_n$, 
notice that $\la X_A, X_B \ra = X_A^*X_B$ is a an element
of $\vnMaS$ hence a diagonal matrix. A direct computation 
shows that for every $k\in \St$
$$ 
X_A^*X_B e_k = X^*_{A}(p_k\otimes Be_k) = \la p_k\otimes 
Ae_k,\; p_k\otimes Be_k \ra \cdot e_k
$$
Hence, 
\begin{align*}
(\la X_{A}, X_{B} \ra)_{jj} 
& =\la p_j\otimes A e_j , 	p_j\otimes B e_j \ra
	= \big\la Ae_j , 		
	\theta^n(p_j)(Be_j) \big\ra \\
&= 	\sum_{i\in \St}\overline{a_{ij}} 
	P^{(n)}_{ij}b_{ij} = \Big( \Diag((\sqrt{P^n} * A)^*(\sqrt{P^n} * B)) \Big)_{jj}
\end{align*}
and it follows that $\la X_A, X_B\ra = \Diag((\sqrt{P^n} * A)^*(\sqrt{P^n} * B))$. Establishing the fiberwise correspondence inner product for $\ee(n)$ We denote the family $\ee = \{ \ee(n) \}_{n\in \nn}$.

%  *****************
%  * Fiber tensor  *
%  *****************

Let us now focus on the subproduct maps. For that purpose, fix $n,m \in \nn$. Let us observe
that in analogy with the situation above, the set
$\{p_k \otimes p_j \otimes e_i\}_{(i,j,k) \in E(P^n,P^m)}$ is is an orthogonal basis for $\vnMaS\otimes_{\theta^m}\vnMaS\otimes_{\theta^n}H$,
since for all $i,j,k,i',j',k'$,
\begin{align*} 
\la p_k\otimes p_j \otimes e_i , p_{k'} \otimes p_{j'} \otimes e_{i'} \ra & = 
\delta_{ii'} \delta_{jj'} \delta_{kk'} \; \la e_i , \theta^n(p_j\theta^n(p_k))e_i \ra = 
\delta_{ii'} \delta_{jj'} \delta_{kk'} \; \la e_i, \theta^n(p_j)e_i \ra P^{(m)}_{jk} \\
& = \delta_{ii'} \delta_{jj'} \delta_{kk'} \; P^{(n)}_{ij}P^{(m)}_{jk}.
\end{align*} 
Furthermore, given $Y \in \mathcal{L}(m,n)$, by a computation
analogous to the one above involving the intertwiner condition,
there exist scalars $c_{ijk}$ for $(i,j,k) \in E(P^n, P^m)$ such that for each $k \in \Omega$
$$ 
Y(e_k) = \sum_{(i,j,k) \in E(P^n, P^m)} c_{ijk} \; p_k\otimes p_j \otimes e_i 
$$
We may also define $c_{ijk} = 0$ for $(i,j,k) \notin E(P^n,P^m)$. Furthermore, the norm of $Y$ given by $\|Y\| = \sup_{k\in \St}\|Y(e_k)\|$. We denote such $Y$ as $Y = Y_{[c_{ijk}]} = Y^{(m,n)}_{[c_{ijk}]}$ where $c_{ijk} = 0$ for $(i,j,k) \notin E(P^n,P^m)$.

One can similarly compute $Y^*_{[c_{ijk}]}$ and the inner product on $\mathcal{L}(m,n)$ which would make it into a W* correspondence along with the usual left and right actions.

As we shall see, these computations are unnecessary for the computation of our subproduct system.

%  ***********************
%  * Multiplication maps *
%  ***********************

Define $V_{n,m} : \mathcal{L}(n+m) \rightarrow \mathcal{L}(m,n)$
by the usual formula $V_{n,m}(X) = \Gamma_{n,m} \circ X$, where $\Gamma_{n,m}: \vnMaS \otimes_{\theta^{n+m}}H \rightarrow \vnMaS \otimes_{\theta^m} \vnMaS \otimes_{\theta^n}H$ is defined by $\Gamma_{n,m}(a\otimes h) = a\otimes I \otimes h$.

It is evident that $V_{n,m}^*(X) = \Gamma_{n,m}^* \circ X$ so in order to compute $V_{n,m}^*$, all one needs to do is compute $\Gamma_{n,m}^*$. So indeed, we compute $\Gamma_{n,m}^*$ by computing the projection $Q = \Gamma^*_{n,m}\Gamma_{n,m}$ onto the image of $\Gamma_{n,m}$ which is exactly $\vnMaS \otimes _{\theta^m}I \otimes_{\theta^n} H$ that has $\{p_k \otimes I \otimes e_i \}_{(i,k) \in E(P^{n+m})}$ as an orthogonal basis. In fact, as Hilbert spaces with the corresponding bases, $\vnMaS \otimes_{\theta^m} I \otimes_{\theta^n} H \cong \vnMaS \otimes_{\theta^{n+m}} H$ via $\Gamma^*_{n,m}$.

We run the aformentioned computation. Indeed,
$$ Q(p_k \otimes p_j \otimes (c_{ijk})_i) = \sum_{i \ : \ (i,k) \in E(P^{n+m})}\frac{\la p_k\otimes I \otimes e_i, p_k \otimes p_j \otimes (c_{ijk})_i \ra}{||p_k\otimes I \otimes e_i||^2}p_k \otimes I \otimes e_i $$
And since, $\la p_k\otimes I \otimes e_i, p_k \otimes p_j \otimes (c_{ijk})_i \ra = c_{ijk}P^{(n)}_{ij}P^{(m)}_{jk}$ and $||p_k\otimes I \otimes e_i||^2 = ||p_k\otimes e_i||^2 = P^{(n+m)}_{ik}$
We obtain that,
$$ Q\big(\sum_{j\in \St}p_k \otimes p_j \otimes (c_{ijk})_i\big) = \sum_{i \ : \ (i,k) \in E(P^{n+m})}\frac{\sum_{j\in \St}c_{ijk}P^{(n)}_{ij}P^{(m)}_{jk}}{P^{(n+m)}_{ik}}p_k \otimes I \otimes e_i $$
Now, since $\Gamma_{n,m}^* = \Gamma_{n,m}^* Q$, we have that
$$ \Gamma_{n,m}^*\big(\sum_{j\in \St}p_k \otimes p_j \otimes (c_{ijk})_i \big) = \sum_{i \ : \ (i,k) \in E(P^{n+m})}\frac{\sum_{j\in \St}c_{ijk}P^{(n)}_{ij}P^{(m)}_{jk}}{P^{(n+m)}_{ik}}p_k \otimes e_i $$
\\
We define the usual $\Psi: \mathcal{L}(n) \otimes \mathcal{L}(m) \rightarrow \mathcal{L}(m,n)$ by $\Psi_{n,m}(X^{(n)}_A\otimes X^{(m)}_B) = (I\otimes X^{(n)}_A)\circ X^{(m)}_B$ and obtain by a simple computation that for $A= [a_{ij}] \in \ee(n)$ and $B = [b_{lk}]\in \ee(m)$,
$$ \Psi_{n,m}(X^{(n)}_{[a_{ij}]}\otimes X^{(m)}_{[b_{lk}]}) = Y^{(m,n)}_{[a_{ij}b_{jk}]} $$
So the multiplication maps are $U_{n,m}: \mathcal{L}(n) \otimes \mathcal{L}(m) \rightarrow \mathcal{L}(n+m)$ given by $U_{n,m} = V_{n,m}^* \Psi_{n,m}$ and we obtain:
$$ U_{n,m}(X^{(n)}_{[a_{ij}]}\otimes X^{(n)}_{[b_{lk}]})(e_k) = (\Gamma_{n,m}^* \circ Y^{(m,n)}_{[a_{ij}b_{jk}]})(e_k) = $$
$$ \sum_{i \ : \ (i,k) \in E(P^{n+m})}\frac{\sum_{j\in \St}a_{ij}P^{(n)}_{ij}b_{jk}P^{(m)}_{jk}}{P^{(n+m)}_{ik}}p_k \otimes e_i $$
In other words,
$$ U_{n,m}(X^{(n)}_{A}\otimes X^{(m)}_{B}) = X^{(n+m)}_{(P^{n+m})^{\flat}*\big[(P^n * A)\cdot(P^m * B)\big]} $$
Defining $\widetilde{U}_{n,m}: \ee(n) \otimes \ee(m) \rightarrow \ee(n+m)$ by the rule 
$$\widetilde{U}_{n,m}(A\otimes B) = (P^{n+m})^{\flat}*\big[(P^n * A)\cdot(P^m * B)\big]$$
Yields that $\big(\ee, \widetilde{U} \big)$ is a subproduct system naturally isomorphic to $\big( \mathcal{L}, U \big)$ via the fiberwise map $A \mapsto X_A$ (for $A\in \ee(n)$).
Now a computation yields that the (Id-)isomorphism $V: \ee \rightarrow Arv(P)$ defined fiberwise by $V_n(A) = \sqrt{P^n}*A$ imposes a structure of a subproduct system on $Arv(P)$, with the aformentioned subproduct maps in the theorem, induced from that of $(\ee, \widetilde{U})$ .
\end{proof}

Note that the structure of the W*-correspondences in a standard presentation depends only on the \emph{graph} structure of the stochastic matrix. Information on the weighted graph is contained only in the subproduct maps.

Recall Definition \ref{definition:subproduct-direct-sum} and Theorem \ref{theorem:graph-decomposition}.

\begin{proposition} \label{proposition:essential-is-direct-sum}
Let $P$ be a finite and essential stochastic matrix over $\St$. Assume that $P$ decomposes into block diagonal form with irreducible stochastic blocks $P(1),..., P(\ell)$, and that $\St(1), ..., \St(\ell)$ are the state sets corresponding to the rows of $P(1),..., P(\ell)$ respectively. Then $Arv(P(k))$ is a subproduct system over $\ell^{\infty}(\St_k)$ for every $1\leq k \leq \ell$, and $Arv(P)$ is canonically Id-isomorphic to $Arv(P(1))\oplus ... \oplus Arv(P(\ell))$, when identifying $\ell^{\infty}(\St(1)) \oplus ... \oplus \ell^{\infty}(\St(\ell))$ with $\vnMaS$ in the natural way.
\end{proposition}

\begin{proof}
Considering $P$ with the mentioned decomposition, every $A\in Arv(P)_n$ decomposes uniquely to block diagonal form with blocks $A(1), ..., A(\ell)$ along the diagonal with $A(k) \in Arv(P(k))_n$ for all $1\leq k \leq \ell$. Since the subproduct $U^{Arv(P)}$ is matrix multiplication (up to Schur products), the block diagonal form is preserved, and we must have that $Arv(P) \cong_{Id} Arv(P(1))\oplus ... \oplus Arv(P(\ell))$ via the map sending $A$ to $A(1)\oplus ...\oplus A(\ell)$.
\end{proof}

\begin{remark}
There is another construction of subproduct systems from completely positive normal maps called the GNS subproduct system mentioned in \cite[Section 3]{SS}. The GNS subproduct system associated to a stochastic matrix was computed for finite stochastic matrices in \cite{Vis}. In our special case of the concrete von-Neumann algebra $\vnMaS \subseteq B(\ell^{2}(\St))$, both the GNS and \AS\ subproduct systems are over the same von-Neumann algebra, and 
although we do not include the proof here, it turns out that the GNS subproduct system for recurrent $P$ is naturally isomorphic to the \AS\ subproduct system for the \emph{time reversal} of $P$. In some sense, this is a duality
phenomenon in the context of subproduct systems arising from stochastic matrices. We use the word ``duality" in analogy with the well-known duality between GNS and \AS\ \emph{product} systems (see \cite[Remark 3.4]{SS} and \cite{DMS}) and product systems in general (see  \cite{skeide-commutant}).  In any case, given this duality
phenomenon, the choice of which  framework to use in the analysis of the subproduct systems arising from stochastic
matrices
is a matter of convenience. We will proceed with the  framework of \AS\ subproduct systems.
\end{remark}

% **************************************
% * Distinguishing stochastic matrices *
% **************************************

The main point of the following is to tell exactly when there exists an isomorphism between two \AS\ subproduct systems arising from a stochastic matrix, in terms of the matrices, and to recognize a certain class of stochastic matrices distinguishable by their \AS\ subproduct systems. Recall the 1-1 correspondence between *-automorphisms of $\vnMaS$ and permutations of $\St$ preceeding Theorem \ref{theorem:subproduct-computation}.

\begin{notation}
Let $P$ and $Q$ be stochastic matrices. If for a permutation $\sigma$ we have $P\sim_{\sigma}Q$ and for all $(i,j,k) \in E(P^n,P^m)$ we have
\begin{equation} \label{eq:subproduct-matirx-equiv}
\frac{P^{(n)}_{ij} \cdot P^{(m)}_{jk} }{P^{(n+m)}_{ik}} = 
\frac{ Q^{(n)}_{\sigma(i) \sigma(j)} \cdot Q^{(m)}_{\sigma(j) \sigma(k)}}{Q^{(n+m)}_{\sigma(i) \sigma(k)}} 
\end{equation}
We say that $P$ and $Q$ satisfy equation \eqref{eq:subproduct-matirx-equiv} via $\sigma$.
\end{notation}

\begin{theorem}\label{theorem:subproduct-systems-isomorphism}
Let $P$ and $Q$ be two stochastic matrices. 
\begin{enumerate}
\item
If $Arv(P) \cong_{\rho} Arv(Q)$ then $P$ and $Q$ satisfy equation \eqref{eq:subproduct-matirx-equiv} via $\sigma_{\rho}$.
\item
If $P$ and $Q$ satisfy equation \eqref{eq:subproduct-matirx-equiv} via $\sigma$, then $Arv(\theta) \cong_{\rho_{\sigma}} Arv(\phi)$.
\end{enumerate}
\end{theorem}
\begin{proof}
When $\sigma$ is a permutation of $\St$, we denote for brevity, $i' = \sigma(i)$.

$(1)$: Assume $V : Arv(P) \rightarrow Arv(Q)$ is a given (unitary) isomorphism of the subproduct systems. Then $\rho = V_0 : \vnMaS \rightarrow \vnMaS$ is induced by a permutation $\sigma = \sigma_{\rho}$. Now for all $(i,j)\in E(P^n)$ denote
$E_{ij}$ to be the element in $Arv(P)_n$ which is $1$ at $(i,j)$ and zero otherwise, and define $E_{i'j'}$ similarly in $Arv(Q)_n$. Due to $V_n$ being a $\rho$-correspondence morphism, we have,
$$
V_n(E_{ij})  = V_n(p_i E_{ij} p_j) = \rho(p_i) V_n(E_{ij}) \rho(p_j) = p_{i'}V_n(E_{ij}) p_{j'}
$$
So we must have that $V_n(E_{ij}) = b_{ij}^{(n)} \cdot E_{i'j'}$ for some $b_{ij}^{(n)} \in \mathbb{C}$. 
Due to $V_n$ being isometric we have that,
$$1 = ||\la E_{ij}, E_{ij} \ra || = ||E_{ij}||^2 = ||b_{ij}^{(n)}\cdot E_{i'j'}||^2 = ||\la b_{ij}^{(n)} \cdot E_{i'j'}, b_{ij}^{(n)} \cdot E_{i'j'} \ra || = |b_{ij}^{(n)}|^2 $$ 
By the formula for the subproducts maps in $Arv(P)$ and $Arv(Q)$ we have that,
$$ U^P_{n,m}(E_{ij} \otimes E_{jk}) = \sqrt{\frac{P_{ij}^{(n)} P_{jk}^{(m)}}{P_{ik}^{(n+m)}}} \cdot E_{ik} \ , \ \ 
U^Q_{n,m}(E_{i'j'} \otimes E_{j'k'}) = \sqrt{\frac{Q_{i'j'}^{(n)} Q_{j'k'}^{(m)}}{Q_{i'k'}^{(n+m)}}} \cdot E_{i'k'}
$$
But since $V$ is preserves subproducts, we obtain,
$$ b_{ik}^{(n+m)} \sqrt{\frac{P_{ij}^{(n)} P_{jk}^{(m)}}{P_{ik}^{(n+m)}}} = b_{ij}^{(n)}b_{jk}^{(m)} 
\sqrt{\frac{Q_{i'j'}^{(n)} Q_{j'k'}^{(m)}}{Q_{i'k'}^{(n+m)}}} $$
So we obtain equation \eqref{eq:subproduct-matirx-equiv} by squaring the absolute value.

$(2)$:
Define $V_0 : \vnMaS \rightarrow \vnMaS$ by $V_0 = \rho_{\sigma} = \rho$, that is, $V_0(p_i) = p_{\sigma(i)}$. Define $V_n : Arv(P)_n \rightarrow Arv(Q)_n$ by the change of variables
$V_n(A) = R_{\sigma}AR^{-1}_{\sigma}$.
We need to show that $V_n$ is a $\rho$-unitary between the two W*-correspondences, and preserves the subproduct maps.

Indeed, it is immediate that $V_n$ is a $\rho$-unitary since it preserves the inner product via $\rho$ and the left and right actions via $\rho$.
Thus we use equation \eqref{eq:subproduct-matirx-equiv} to show that $V_n$ preserves subproducts maps. Let $A \in Arv(P)_n$ and $B \in Arv(P)_m$. Then,
$$
U_{n,m}^Q(V_n(A) \otimes V_m(B)) = (\sqrt{Q^{n+m}})^{\flat}*\big[(\sqrt{Q^n} * R_{\sigma}AR^{-1}_{\sigma})\cdot(\sqrt{Q^m} * R_{\sigma}BR^{-1}_{\sigma})\big] = $$
$$ R_{\sigma} \Bigg[ (R^{-1}_{\sigma} (\sqrt{Q^{n+m}})^{\flat} R_{\sigma})*\Big\{ \big[(R^{-1}_{\sigma}\sqrt{Q^n}R_{\sigma}) * A \big]\cdot \big[(R^{-1}_{\sigma}\sqrt{Q^m}R_{\sigma}) * B \big]\Big\} \Bigg] R^{-1}_{\sigma}
= 
$$
$$
R_{\sigma} \Big[ (\sqrt{P^{n+m}})^{\flat}*\big\{(\sqrt{P^n} * A)\cdot(\sqrt{P^m} * B)\big\} \Big] R^{-1}_{\sigma}= V_{n+m}U_{n,m}^P (A \otimes B) 
$$
Where the second last equality is due to equation \eqref{eq:subproduct-matirx-equiv} being satisfied with respect to $\sigma$.

\end{proof}

\begin{example}
\AS\ subproduct systems are unable to distinguish all finite stochastic matrices.
Define for $r \in (0,1)$ the stochastic matrix over $\St = \{1 , 2 , 3 \} $:
		 $$ P(r) = \begin{bmatrix}
        0 & r & 1-r           \\[0.3em]
        0 & 1           & 0 \\[0.3em]
        0           & 0 & 1
      \end{bmatrix} $$
Note that for all $r, q\in (0, 1)$ the matrices $P(r)$ and $P(q)$ are graph isomorphic via $\sigma = Id_{\St}$ and that $P(r) = P(r)^n$ for all $n\geq 1$. Thus, to check that equation \eqref{eq:subproduct-matirx-equiv} is satisfied with respect to $\sigma = Id_{\St}$, we need only show that for $(i,j,k) \in E(P(r),P(r)) = E(P(q),P(q))$ we have that
$$
\frac{P(r)_{ij} \cdot P(r)_{jk} }{P(r)_{ik}} = 
\frac{ P(q)_{i j} \cdot P(q)_{jk}}{P(q)_{i k}}
$$
Which is readily seen to be satisfied for any $(i,j,k) \in E(P,P)$ and $r, q \in (0,1)$. So \AS\ subproduct system can't distinguish a continuous family of finite reducible stochastic matrices, since for $r\neq q$ in $(0,\frac{1}{2}]$ we have that $P(r)$ and $ P(q)$ are not isomorphic (via any $\sigma$) due to the matrices having different sets of probabilities.
\end{example}

Note that in the last example, $i = 1$ is inessential.

\begin{example}
\AS\ subproduct systems cannot distinguish all infinite irreducible stochastic matrices.\\
Take $r\in (0,1)$ and denote by $P(r)$ the matrix over $\St = \mathbb{Z}$ with entries 
$P(r)_{i(i+1)} = r \ , \ P(r)_{i(i-1)} = 1-r $, and all other entries zero. 
Then one can check that any two paths of the same length $\ell \geq 2$ both starting at $i$ and ending at $k$ have the same probabilities. This means that, if $\gamma$ and $\gamma'$ are two such paths, then 
$$
\Pi_{m=1}^{\ell}P(r)_{\gamma(m)\gamma(m+1)} = \Pi_{m=1}^{\ell}P(r)_{\gamma'(m)\gamma'(m+1)}
$$
Thus, if we sum over all paths of length $n+m$ starting at $i$ and ending at $k$ on the right, and on all paths of length $n+m$ starting at $i$, ending at $k$ and passing through $j$ on the left, such that $(i,j,k) \in E(P^n,P^m)$, we obtain that
$$
d_{n,m} \cdot P(r)^n_{ij}P(r)^m_{jk} = P(r)^{n+m}_{ik}
$$
Where $d_{n,m}$ is the number of $j' \in \nn$ such that $(i,j',k) \in E(P^n,P^m)$ (which is independent of $r$). Thus equation \eqref{eq:subproduct-matirx-equiv} holds for $P(r)$ and $P(q)$ for all $r, q\in (0,1)$, and we must have that $Arv(P(r)) \cong_{Id} Arv(P(q))$. This means that \AS\ subproduct systems cannot distinguish a continuous family of irreducible stochastic matrices, since for different $r \in (0,\frac{1}{2}]$, the associated matrices $P(r)$ are not isomorphic (via any $\sigma$), as they have different sets of probabilities. Further note that if you take $r < \frac{1}{2}$ then $P(r)$ is transient while $P(\frac{1}{2})$ is recurrent.
\end{example}

So a question arises, what known classes of stochastic matrices can \AS\ subproduct systems distinguish up to isomorphism?

\begin{theorem} \label{theorem:telling-apart}
Let $P$ and $Q$ be \emph{recurrent} stochastic matrices. If $Arv(P) \cong_{\rho} Arv(Q)$ then $P \cong_{\sigma_{\rho}} Q$.
\end{theorem}

\begin{proof}
Take $\sigma = \sigma_{\rho}$ to be the permutation on $\St$ which induces the graph isomorphism between $P$ and $Q$ satsifying equation \eqref{eq:subproduct-matirx-equiv}. We denote for brevity, $i' = \sigma(i)$. First assume that $P$ and $Q$ are both irreducible and assume also that $P$ is $r$-periodic.

We now show that $P^{(n)}_{ii} = Q^{(n)}_{i'i'}$ for every $n \in \mathbb{N}$ and $i \in \St$.
Indeed, if $P_{ii}^{(n)} = 0$ then $Q^{(n)}_{i'i'} =0$ due to the graph isomorphism. Now if $P_{ii}^{(n)}  > 0$ then again $Q^{(n)}_{i'i'} >0$ due to the graph isomorphism, and $i$ and $i'$ have the same period $r$ in $P$ and $Q$ respectively, again due to the graph isomorphism. Note that one has that $Arv(P^{n}) \cong_{\rho} Arv(Q^{n})$, that $P^n$ and $Q^n$ are aperiodic (since $n=rn'$ for some $n'\in \mathbb{N}$), and that both $P^n$ and $Q^n$ are recurrent (due to Lemma \ref{lemma-rec-iterate}). This reduces the problem to showing that $P_{ii} = Q_{i'i'}$ where $P$ is replaced by $P^n$ and $Q$ is replaced by $Q^n$.
Thus, for all $i \in \St$ we have
$$
 \frac{P_{ii} \cdot P^{(m)}_{ii} }{P^{(m+1)}_{ii}} = 
\frac{ Q_{i' i'} \cdot Q^{(m)}_{i' i'}}{Q^{(m+1)}_{i' i'}}
$$
Where these expressions are always well defined. It follows that,
$$ \frac{ (P_{ii})^M}{P^{(M)}_{ii}} = \prod_{m=1}^{M-1} \frac{P_{ii}P_{ii}^{(m)}}{P_{ii}^{(m+1)}} = \prod_{m=1}^{M-1} \frac{Q_{i'i'}Q_{i'i'}^{(m)}}{Q_{i'i'}^{(m+1)}} = \frac{(Q_{i'i'})^M}{Q^{(M)}_{i'i'}} $$
Since $P$ and $Q$ are recurrent chains, they must be amenable. Thus by taking an $M$-th root and a limit in $M$, we obtain that $P_{ii} = Q_{i'i'}$.

Now by taking $i=k$ in equation \eqref{eq:subproduct-matirx-equiv} and $n=1$ we obtain for all $i,j \in \St$
\begin{equation*}
P_{ij} P^{(m)}_{ji} = 
Q_{i' j'}Q^{(m)}_{j' i'} 
\end{equation*}
By taking sums over $m$ we obtain
\begin{equation*}
P_{ij} \sum_{m=1}^MP^{(m)}_{ji} = Q_{i' j'} \sum_{m=1}^MQ^{(m)}_{j' i'} 
\end{equation*}
Now since $P_{ii}^{(m)} = Q_{i'i'}^{(m)}$ for all $m \in \mathbb{N}$, we must have that,
\begin{equation}\label{eq:before-limit}
P_{ij} \cdot \frac{\sum_{m=1}^MP^{(m)}_{ji}}{\sum_{m=1}^MP^{(m)}_{ii}} = Q_{i' j'} \cdot \frac{\sum_{m=1}^MQ^{(m)}_{j' i'}}{\sum_{m=1}^MQ^{(m)}_{i' i'}}
\end{equation}

Now if $P$ is a recurrent irreducible stochastic matrix, by 
Doeblin's ratio limit theorem (see Part I, Section 9, Theorem 5 in \cite{Chung})
we have for $P$ (and also for $Q$) that for any three states $i,j,k \in \St$,
$$ 
\lim_{M\rightarrow \infty} \frac{\sum_{m=1}^MP^{(m)}_{ij}}{\sum_{m=1}^MP^{(m)}_{kj}} = 1 
$$
Thus, by taking $M \to \infty$ in equation \eqref{eq:before-limit} and using Doeblin's theorem for both $P$ and $Q$, 
we must have that $P_{ij} = Q_{i' j'}$.

Now, for general (reducible) recurrent stochastic matrices $P$ and $Q$, Theorem \ref{theorem:graph-decomposition} 
enables us to decompose $P$ into $P(1),P(2),...$ in block diagonal form, which induces a decomposition of $Q$ into $Q(1),Q(2),...$ in block diagonal form via the graph isomorphism $\sigma$, such that for all $0<k\in \mathbb{N}$ we have $Arv(P(k)) \cong_{\rho_{k}} Arv(Q(k))$ for some *-automorphism $\rho_{k}$, since equation \eqref{eq:subproduct-matirx-equiv} is satisfied via $\sigma_{\rho}$ (restricted to the appropriate set of indices in $\St$) for the pairs $P(k)$ and $Q(k)$ for every $0<k\in \mathbb{N}$, and thus a reduction is made to the general case.
\end{proof}

% *************************************************
% * Cuntz-Pimsner algebras for subproduct systems * 
% *************************************************

\section{Cuntz-Pimsner algebras for subproduct systems}
\label{section:cuntz-pimsner}

We describe the construction of Toeplitz and Cuntz-Pimsner algebras arising from subproduct systems  (see \cite{V, Vis}). Let $X = (X_n)_{n\in \mathbb{N}}$ be a subproduct system. We use the following notations throughout this work. The $X$-Fock correspondence is the W*-correspondence (weak) direct sum of the fibers of the subproduct system:
\begin{equation*}
\mathcal{F}_X := \bigoplus_{n\in \mathbb{N}}X_n 
\end{equation*}

Denote by $Q_n \in \mathcal{L}(\mathcal{F}_X)$ the projection of $\mathcal{F}_X$ onto the $n$th fiber $X_n$. Define $Q_{[0,n]} = Q_0 + Q_1 + ... + Q_n$ and $Q_{[n,\infty)} = Id - Q_{[0,n-1]}$.

The $X$-shifts are the operators $S^{(n)}_{\xi} \in \mathcal{L}(\mathcal{F}_X)$ for $n\in \mathbb{N}$, $\xi \in X_n$ given by 
$$ S^{(n)}_{\xi}(\eta) : = U_{n,m}(\xi \otimes \eta) $$ 
For $m \in \mathbb{N}$, $\eta \in X_m$. Since $\mathcal{F}_X$ is a W*-correspondence, $S^{(n)}_{\xi}$ are adjointable.

\begin{defi}
The Toeplitz algebra $\mathcal{T}(X)$ is the C*-subalgebra of $\mathcal{L}(\mathcal{F}_X)$ generated by all $X$-shifts and $\vnM$,
$$ 
\mathcal{T}(X) : = C^*(\vnM \cup \{ \ S^{(n)}_{\xi} \ | \ \xi \in X_n, \ n\in \mathbb{N} \ \})
$$
\end{defi}

\begin{remark}
One can show that $S^{(n)}_{\xi}$ are adjointable even if one takes the C*-correspondence direct sum for the Fock direct sum correspondence, and due to proposition \ref{proposition:extension-self-dual}, one obtains the same Toeplitz algebra as before.
\end{remark}

The algebra $\mathcal{L}(\mathcal{F}_X)$ admits a natural action $\alpha$ of the unit circle $\mathbb{T}\subseteq \mathbb{C}$, called the gauge action, defined by $\alpha_{\lambda}(T) = W_{\lambda}TW_{\lambda}^*$ for all $\lambda \in \mathbb{T}$ where $W_{\lambda}: \mathcal{F}_X \rightarrow \mathcal{F}_X$ is the unitary defined by 
$$W_{\lambda}(\oplus_{n\in \mathbb{N}} \xi_n) = \oplus_{n\in \mathbb{N}} \lambda^n \xi_n$$
Since $\alpha_{\lambda}(S^{(n)}_{\xi}) = S^{(n)}_{\lambda^n \xi}$, it follows that the Toeplitz algebra is an $\alpha$-invariant closed C*-subalgebra, so the circle action can be restricted to a circle action on the Toeplitz algebra. One then shows that for every $S\in \mathcal{T}(X)$, the function $f(\lambda) = \alpha_{\lambda}(S)$ is norm continuous. This enables the definition of a conditional expectation $\Phi$ defined by
$$ \Phi(S) =\int_{\mathbb{T}}\alpha_{\lambda}(S)d \lambda $$
Where $d \lambda$ is the normalized Haar measure on $\mathbb{T}$.

Let $\{ k_n \}_{n=1}^{\infty}$ denote Fejer's kernel function defined for $\lambda \in \mathbb{T}$ by 
$$k_n(\lambda) = \sum_{j=-n}^n\big(1 - \frac{|j|}{n+1}\big) \lambda^j$$
Note that for $S\in \mathcal{T}(X)$, the existence of the canonical conditional expectation $\Phi$ permits the definition of Fourier coefficients for an element $S \in \mathcal{T}(X)$ by
$$\Phi_n(S) = \int_{\mathbb{T}} \alpha_{\lambda}(S)\lambda^{-n}d \lambda$$
Then define the Cesaro sums,
$$ \sigma_n(S) := \sum_{j = -n}^n\big(1 - \frac{|j|}{n+1}\big) \Phi_j(S) =
\int_{\mathbb{T}} \sum_{j = -n}^n\big(1 - \frac{|j|}{n+1}\big) \alpha_{\lambda}(S)\lambda^{-j}d \lambda
= \int_{\mathbb{T}} \alpha_{\lambda}(S)k_n(\lambda)d \lambda
$$

\begin{proposition}
For every $S\in \mathcal{T}(X)$, the Cesaro sums $\sigma_n(S)$ converge in norm to $S$.
\end{proposition}

The proof is an easy adaptation of the proof of Theorem VIII.2.2 in \cite{Dav}.

\begin{defi}
For each $k \in \mathbb{Z}$ the $k$-th spectral subspace for $\alpha$ is defined by 

$$\mathcal{T}(X)_k = \{ \ T \in \mathcal{T}(X) \ | \ \alpha_{\lambda}(T) = \lambda^kT  \ \} $$

\end{defi}

\begin{defi}
Let $X = (X_n)_{n\in \mathbb{N}}$ be a subproduct system. A monomial in $\mathcal{T}(X)$ is a composition of finitely many of the operators $S^{(n)}_{\xi}$ and their adjoints. Every such operator can be written as $\Pi_{i = 1}^t S^{(m_i)*}_{\xi_i}S^{(n_i)}_{\zeta_i}$ for suitable $0<t\in \mathbb{N}$, and $n_i, m_i \in \mathbb{N}$, $\xi_i \in X_{m_i}$, $\zeta_i \in X_{n_i}$. A monomial of this form is said to be of degree $\sum_{i=1}^t(n_i - m_i)$. For $k\in \mathbb{Z}$ define $\mathcal{T}_k(X)$ to be the closure of all homogeneous polynomials of degree $k$. Note that $\mathcal{T}_0(X)$ is a C*-subalgebra of $\mathcal{T}(X)$.
\end{defi}

The next corollary gives us a characterization of the graded structure of Toeplitz algebras, in terms of the natural circle action.

\begin{corollary}
Let $X$ be a subproduct system. Then, $\mathcal{T}_k(X) = \mathcal{T}(X)_k$.
\end{corollary}

\begin{proof}
It is trivial to show that $\mathcal{T}_k(X) \subseteq \mathcal{T}(X)_k$. To show the reverse inclusion, take $T \in \mathcal{T}(X)_k$, and let $T_n$ be a sequence of polynomials in $\mathcal{T}(X)$ converging to $T$ in norm. Since
$$
\| T - \Phi_k(\sigma_n(T_n)) \| \leq  \| T - \sigma_n(T) \| + \| \Phi_k(\sigma_n(T)) - \Phi_k(\sigma_n(T_n)) \| \leq \| T - \sigma_n(T) \| + \| T - T_n \|
$$
We must have that $T$ is a norm limit of homogeneous polynomials of degree $k$, and we are done.
\end{proof}

A closed subspace $M\leq \mathcal{T}(X)$ is invariant under $\alpha$ or simply \emph{invariant} if for all $\lambda \in \mathbb{T}$ one has $\alpha_{\lambda}(M) \subseteq M$, along with the previous proposition we obtain:

\begin{corollary} \label{corollary:invariant-subspace}
Let $M\leq \mathcal{T}(X)$ be a closed invariant subspace. Then $M$ is the closed linear span of homogeneous polynomials in $M$. That is, $M = \overline{Sp} \{\bigcup \big(M \cap \mathcal{T}_k(X)\big) \}$.
\end{corollary}
\begin{proof}
Let $S \in M$. Then for every $k\in \mathbb{Z}$ denote $\Phi_k(S): = \int_{\mathbb{T}}\alpha_{\lambda}(S)\lambda^{-k}d \lambda$ which is in $M \cap \mathcal{T}_k(X)$. Since $\sigma_n(S)$ is a linear combination of $\{ \Phi_k(S) \}_{k= -n}^{n}$, we are done.
\end{proof}

\begin{proposition}
Define a subset $\mathcal{J} \subseteq \mathcal{L}(\mathcal{F}_X)$ by
$$ \mathcal{J} = \{ \ T\in \mathcal{L}(\mathcal{F}_X) \ | \ \lim_{n \rightarrow \infty}||TQ_n|| = 0 \ \} $$
Then $\mathcal{J}$ is a closed left invariant ideal in $\mathcal{L}(\mathcal{F}_X)$.
\end{proposition}

\begin{proof}
$\mathcal{J}$ is obviously a left ideal, and is closed in norm. Indeed, take $S\in \bar{\mathcal{J}}$, and $\epsilon > 0$, choose $T\in \mathcal{J}$ such that $||S - T|| < \epsilon$. Thus we have that $||SQ_n - TQ_n|| < \epsilon$ for every $n\in \mathbb{N}$, so we must have that also $||SQ_n|| \rightarrow_{n\rightarrow \infty} 0$.
Now, if $T \in \mathcal{J}$, we note that for each $\lambda \in \mathbb{T}$ and each $m\in \mathbb{N}$ $W_{\lambda}Q_m = Q_mW_{\lambda}$ and thus,
$$|| \alpha_{\lambda}(T)Q_m|| = ||W_{\lambda}TW_{\lambda}^*Q_m|| = ||W_{\lambda}TQ_mW_{\lambda}^*|| = || TQ_m || \underset{m\rightarrow \infty}{\longrightarrow} 0 $$
\end{proof}

\begin{notation}
Let $X$ be a subproduct system, and let $\mathcal{A}$ be an invariant C*-subalgebra of $\mathcal{L}(\mathcal{F}_X)$. Let $\mathcal{J}(\mathcal{A}): = \mathcal{A} \cap \mathcal{J}$, which is an invariant closed left ideal in $\mathcal{A}$.
\end{notation}

\begin{proposition}
$\mathcal{J}(\mathcal{T}(X))$ is a \emph{two} sided closed invariant ideal in $\mathcal{T}(X)$ and one has
$$ \mathcal{J}(\mathcal{T}(X)) = \{ T \in \mathcal{T}(X) \ | \ \lim_{n\rightarrow \infty}||TQ_{[n,\infty)}|| = 0 \ \} $$ 
\end{proposition}
\begin{proof}
The fact that $\mathcal{J}(\mathcal{T}(X))$ is a left ideal is shown in \cite[Theorem 2.5]{Vis}, and the equality of sets above follows from \cite[Corollary 2.7]{Vis}.
\end{proof}

\begin{defi}
Let $X = (X_n)_{\mathbb{N}}$ be a subproduct system, and let $\mathcal{T}(X)$ be its corresponding Toeplitz algebra. We define the Cuntz ideal of $\mathcal{T}(X)$ to be $\mathcal{J}(\mathcal{T}(X)) = \mathcal{J} \cap \mathcal{T}(X)$. The Cuntz-Pimsner algebra of $X$ is then defined to be $\mathcal{O}(X) : = \mathcal{T}(X) / \mathcal{J}(\mathcal{T}(X))$.
\end{defi}

Going back to the full product system in example \ref{regular-product-system}, we note that the Toeplitz and Cuntz-Pimsner algebras of $X_E$ are the usual Toeplitz and Cuntz-Pimsner algebras of the W*-correspondence $E$, in the sense of Katsura~\cite{Kats}.

See \cite[Example 3.6 and Remark 3.7]{Vis} for some motivation as to why Viselter defined the Cuntz-Pimsner algebra in this way for subproduct systems.

We denote by $\overline{T}$ the image of $T\in \mathcal{T}(X)$ under the canonical quotient from $\mathcal{T}(X)$ onto $\mathcal{O}(X)$.

\begin{proposition} \label{bound-norm-cuntz}
If $Q_n \in\mathcal{T}(X)$ for all $n\in \mathbb{N}$, then $$||\overline{T}||_{\mathcal{O}(X)} = \lim_{n\rightarrow \infty}||TQ_{[n,\infty)}||_{\mathcal{T}(X)}$$
Further, if $T\in \mathcal{T}(X)_k$ then,
$$
||\overline{T}||_{\mathcal{O}(X)} = \limsup_{n\rightarrow \infty}||TQ_n||_{\mathcal{T}(X)}$$
 \end{proposition}
\begin{proof}
Indeed, for every $\epsilon > 0$ there exists $Q\in \mathcal{J}(\mathcal{T}(X))$ such that 
$$||\overline{T}|| \geq ||T + Q|| - \epsilon \geq \limsup_{n\rightarrow \infty}||(T+Q)Q_{[n,\infty)}|| - \epsilon \geq $$
$$ \limsup_{n\rightarrow \infty}||TQ_{[n,\infty)}|| - \limsup_{n\rightarrow \infty}||QQ_{[n,\infty)}|| - \epsilon = \limsup_{n\rightarrow \infty}||TQ_{[n,\infty)}|| - \epsilon$$
Thus we have that $||\overline{T}|| \geq \limsup_{n\rightarrow \infty}||TQ_{[n,\infty)}||$. We also have $||\overline{T}|| = \inf_{Q\in \mathcal{J}(\mathcal{T}(X))}||T + Q||$, so by taking $Q = - T Q_{[0,n]} \in \mathcal{J}(\mathcal{T}(X))$ we obtain that $||\overline{T}|| \leq \liminf_{n\rightarrow \infty}||TQ_{[n,\infty)}||$.

For the second equality, note that if $T\in \mathcal{T}(X)_k$ for $k\in \mathbb{Z}$ then for $n \geq k$ we have $T \upharpoonright _{X_n} : X_n \rightarrow X_{n+k}$ and thus $||TQ_{[n,\infty)}||_{\mathcal{T}(X)} = \sup_{m\geq n}||TQ_m||_{\mathcal{T}(X)} $

\end{proof}

\begin{remark} \label{cuntz-homogeneous-action}
Note that the circle action on $\mathcal{T}(X)$ passes naturally to $\mathcal{O}(X)$ due to $\mathcal{J}(\mathcal{T}(X))$ being invariant. This means that we still have $\mathcal{O}_k(X) = \mathcal{O}(X)_k$ where these sets are defined to be the images of $\mathcal{T}_k(X)$ and $\mathcal{T}(X)_k$ under the quotient map from $\mathcal{T}(X)$ to $\mathcal{O}(X)$.
\end{remark}

The following is a useful proposition used to establish *-isomorphisms between C*-algebras with a gauge action, it is a special case of Proposition 4.5.1 in \cite{BO} for when the compact group is the torus $\mathbb{T}$.

\begin{proposition}\label{proposition:injectivity-fixed-point-alg}
Let $\mathcal{A}$ and $\mathcal{B}$ be two C*-algebras with gauge (torus) actions $\alpha$ and $\beta$ respectively. Assume further that $\pi : \mathcal{A} \rightarrow \mathcal{B}$ is a *-homomorphism satisfying for every $a \in \mathcal{A}$ and $\lambda \in \mathbb{T}$ that $\beta_{\lambda}(\pi(a)) = \pi( \alpha_{\lambda}(a))$. Then $\pi$ is injective if and only if it is injective on the fixed point algebra $\mathcal{A}_0 : = \{ \ a\in \mathcal{A} \ | \ \alpha_{\lambda}(a) = a  \ \}$.
\end{proposition}

We now relate the theory of finite direct sums of subproduct systems, to that of finite direct sums of Toeplitz and Cuntz algebras.

Let $X$ and $Y$ be subproduct systems over $\mathcal{M}$ and $\mathcal{N}$ respectively. Define the unitary element $W_{X,Y} : \mathcal{F}_{X\oplus Y} \rightarrow \mathcal{F}_X \oplus \mathcal{F}_Y$ by $W_{X,Y}(\oplus_{n=0}^{\infty}(\xi_n \oplus \eta_n)) = (\oplus_{n=0}^{\infty}\xi_n) \oplus (\oplus_{n=0}^{\infty}\eta_n)$ for $\xi_n \in X_n$ and $\eta_n \in Y_n$. Further note that we have the natural inclusion $\mathcal{T}(X) \oplus \mathcal{T}(Y) \subseteq \mathcal{L}(\mathcal{F}_X \oplus \mathcal{F}_Y)$.

We omit the straightforward proof of the following proposition.

\begin{proposition} \label{proposition:essential-cuntz-direct-sum}
Let $(X,U^X)$ and $(Y,U^Y)$ be subproduct systems over $\mathcal{M}$ and $\mathcal{N}$ respectively. Then the *-isomorphism $Ad_{W_{X,Y}}: \mathcal{L}(\mathcal{F}_{X\oplus Y})\rightarrow \mathcal{L}(\mathcal{F}_X \oplus \mathcal{F}_Y)$ given by $Ad_{W_{X,Y}}(T) = W_{X,Y}T W^*_{X,Y}$ restricted to $\mathcal{T}(X\oplus Y)$, induces a *-isomorphism $\pi : \mathcal{T}(X\oplus Y) \rightarrow \mathcal{T}(X) \oplus \mathcal{T}(Y)$ such that $\pi(\mathcal{J}(\mathcal{T}(X\oplus Y))) = \mathcal{J}(\mathcal{T}(X)) \oplus \mathcal{J}(\mathcal{T}(Y))$. In particular $\mathcal{O}(X\oplus Y) \cong \mathcal{O}(X) \oplus \mathcal{O}(Y)$.
\end{proposition}

% *************************************************************
% * Cuntz algebras of \AS\ subproduct systems over $\vnMaS$ * 
% *************************************************************

\section{Cuntz-Pimsner algebras of subproduct systems arising from finite essential stochastic matrices}\label{section:cuntz-pimsner-stochastic}

In this section we compute the Cuntz-Pimsner algebras of subproduct systems arising from stochastic matrices, in the sense defined above, where $\St$ is finite and $P$ essential. We sometimes denote $\mathcal{T}(P) : = \mathcal{T}(Arv(P))$ and $\mathcal{O}(P) : = \mathcal{O}(Arv(P))$.

\begin{remark} \label{remark:recurrence-finite}
In the case of finite matrices, a matrix $P$ is recurrent if and only if it is positive-recurrent if and only if it is essential. In particular, finite irreducible stochastic matrices, are positive-recurrent.
\end{remark}

Note that the subproduct system associated to a finite stochastic matrix in standard presentation has no boundedness condition on elements of the fibers since $P$ is finite. And we also have that $\Diag : M_d(\mathbb{C}) \rightarrow \vnMaS \cong \mathbb{C}^d$ is a faithful conditional expectation.

We look at the shift operators of $Arv(P)$. For every $n\in \mathbb{N}$ and $A \in Arv(P)_n$ we denote by
$S^{(n)}_{A}$ the operator defined for every $B\in Arv(P)_m$ for $m>0$ by
$$ 
S^{(n)}_{A}(B) = U_{n,m}(A \otimes B) 
= (\sqrt{P^{n+m}})^{\flat}*\big[(\sqrt{P^n} * A)\cdot(\sqrt{P^m} * B)\big]
$$
Where for $m=0$ it is simply the right multiplication of $A \in Arv(P)_n$ by an element of $\vnMaS$.

The operator $S^{(n)*}_{A}$ for $A \in Arv(P)_n$ is given by the following formula
$$
S^{(n)*}_{A}(B) = \sqrt{P^m}*\big[(\sqrt{P^n} * A)^* \cdot((\sqrt{P^{n+m}})^{\flat} * B)\big] $$
for $B \in Arv(P)_{n+m}$ and $m>0$. Further satisfying $S^{(n)*}_{A}(B) = 0$ for $B \in Arv(P)_{m}$ where $m < n$, and $S^{(n)*}_{A}(B) = \Diag \big[A^*B \big]$ for $B \in Arv(P)_n$.

The Toeplitz algebra is defined to be $$\mathcal{T}(P): = C^*\Big(M \cup \big\{ \ S^{(n)}_{A} \ | \ n\in \mathbb{N}, \  \ A \in Arv(P)_n \ \big\}\Big)$$
We also recall that $Q_n \in \mathcal{L}(\mathcal{F}_{Arv(P)})$ is the projection onto the $n$-th summand of the Fock module direct sum.
\begin{proposition}
Let $P$ be finite and stochastic. Then $Q_n \in \mathcal{T}(P)$ for every $n\in \mathbb{N}$.
\end{proposition}
\begin{proof}
Let $E_{ij}$ be the matrix with 1 in the $(i,j)$-th coordinate and zeros everywhere else. Then we show directly that 
\begin{equation*}
Q_{[n,\infty)} = \sum_{(i,j) \in E(P^n)}S^{(n)}_{E_{ij}}S^{(n)*}_{E_{ij}} \in \mathcal{T}(P) 
\end{equation*}
Indeed, since $S^{(n)*}_{E_{ij}}(A) = 0$ for all $m< n$ and $A\in Arv(P)_m$, it would suffice to show that the right hand side in the above equation is the identity on $E_{lk} \in Arv(P)_m$ for all $(l,k) \in E(P^m)$ and $m \geq n$. For $E_{lk} \in Arv(P)_m$ such that $m > n$ we have
$$S^{(n)*}_{E_{ij}}(E_{lk}) = \delta_{i,l} \cdot \sqrt{\frac{P^{(n)}_{ij}P^{(m)}_{jk}}{P^{(n+m)}_{lk}}}E_{jk}, \ \ and \ \
S^{(n)}_{E_{ij}}(E_{jk}) = \sqrt{\frac{P^{(n)}_{ij}P^{(m)}_{jk}}{P^{(n+m)}_{ik}}}E_{ik}
$$
So we have that
$$
S^{(n)}_{E_{ij}}S^{(n)*}_{E_{ij}}(E_{lk}) = \delta_{i,l} \cdot \frac{P^{(n)}_{ij}P^{(m)}_{jk}}{P^{(n+m)}_{ik}}E_{lk}
$$
Where $\delta_{i,l} = 1$ if $i=l$ and zero otherwise. Now by taking sums over $i,j \in \St$ and noting that a similar computation works for $m=n$, we obtain our description of $Q_{[n,\infty)}$ above.

Thus, one sees that $Q_0 = I - Q_{[1,\infty)} \in \mathcal{T}(P)$ and $Q_n = Q_{[n,\infty)} - Q_{[n+1,\infty)} \in \mathcal{T}(P)$.
\end{proof}

\begin{defi}
We define for every $A \in Arv(P)_n$ the operators 
$$T^{(n)}_{A} : Arv(P)_m \rightarrow Arv(P)_{n+m} \ \ by \ \ T^{(n)}_{A}(B) = S^{(n)}_{(\sqrt{P^n})^{\flat} * A}(B)$$
And
$$W^{(n)}_{A} : Arv(P)_m \rightarrow Arv(P)_{n+m} \ \ by \ \ W^{(n)}_{A}(B) = A\cdot B$$
for $B\in Arv(P)_m$.
\end{defi}

\begin{proposition}
Let $P$ be finite and stochastic.
$T^{(n)}_{A}$ and $W^{(n)}_{A}$ are well-defined bounded operators in $\mathcal{L}(\mathcal{F}_{Arv(P)})$ for every $A \in Arv(P)_n$.
\end{proposition}

\begin{proof}
During this proof, summation is taken over all $i\in \St$ such that $(i,j) \in E(P^n)$, yet we abuse notation and simply write $i\in \St$ for brevity. Assume $A = [a_{ij}]$ and $B = [b_{lk}]$.
$$||T^{(n)}_{A}(B)|| = ||U_{n,m}\Big( \big((\sqrt{P^n})^{\flat} * A\big) \otimes B \Big)|| \leq || (\sqrt{P^n})^{\flat} * A || \cdot || B ||
$$
Since one has, 
$$
|| (\sqrt{P^n})^{\flat} * A ||^2 = \sup_{j\in \St} \sum_{i \in \St}\frac{|a_{ij}|^2}{P^{(n)}_{ij}}
$$
The fact that $P$ is finite entails that $T^{(n)}_{A}$ is bounded and well-defined. We now show well-definededness and boundedness of $W^{(n)}_{A}$. Indeed, we have,
$$||W^{(n)}_{A}(B)||^2 = ||\Big[ \sum_{j\in \St}a_{ij}b_{jk} \Big]||^2 = \sup_k \sum_{i\in \St} \Big| \sum_{j\in \St}a_{ij}b_{jk} \Big|^2 \leq 
$$
$$
\sum_{i\in \St} \Big( \sum_{j\in \St}|a_{ij}|^2 \Big) \sup_k \big( \sum_{l\in \St}|b_{lk}|^2 \big) =
\Bigg( \sum_{i,j\in \St}|a_{ij}|^2 \Bigg) ||B||^2
$$
Where we've used the Cauchy Schwarz inequality in the inequality above.
\end{proof}

The last proposition ensures that in the finite case, we have $$\{ \ T^{(n)}_{A} \ | \ A \in Arv(P)_n, \ n\in \mathbb{N} \ \} = \{ \ S^{(n)}_{A} \ | \ A \in Arv(P)_n, \ n\in \mathbb{N} \ \}$$

Thus,
$$\mathcal{T}(P) = C^*\Big(M \cup \big\{ \ T^{(n)}_{A} \ | \ n\in \mathbb{N}, \  \ A \in Arv(P)_n \ \big\}\Big)$$

We denote 
$$\mathcal{T}^{\infty}(P) = C^*\Big(M \cup \big\{ \ W^{(n)}_{A} \ | \ n\in \mathbb{N}, \  \ A \in Arv(P)_n \ \big\}\Big)$$

Note that by Theorem~\ref{theorem:subproduct-computation} the structure of $\mathcal{T}^{\infty}(P)$ depends only on the graph structure of $P$, and that $\mathcal{T}^{\infty}(P)$ is invariant under the circle action on $\mathcal{L}(\mathcal{F}_{Arv(P)})$ since $\alpha_{\lambda}(W^{(n)}_A) = W^{(n)}_{\lambda^n A}$.

Recall that $\mathcal{J} = \{ \ T\in \mathcal{L}(\mathcal{F}_{Arv(P)}) \ | \ \lim_{n \rightarrow \infty}||TQ_n|| = 0 \ \}$ is an invariant closed left ideal of $\mathcal{L}(\mathcal{F}_{Arv(P)})$.

\begin{proposition} \label{proposition:in-cuntz-ideal}
Let $P$ be a finite irreducible stochastic matrix. For every $n \in \mathbb{N}$ and $A \in Arv(P)_n$ we have $T^{(n)}_{A} - W^{(n)}_{A} \in \mathcal{J}$.
\end{proposition}

\begin{proof}
Assume $P$ has period $r \geq 1$. Let $A= [a_{ij}]\in Arv(P)_n$ and $B\in Arv(P)_m$. Evaluating $||(T^{(n)}_{A} - W^{(n)}_{A})(B)||^2$ using Cauchy-Schwarz inequality yields
$$
||(T^{(n)}_{A} - W^{(n)}_{A})(B)||^2 \leq \sup_k \Bigg( \sum_{i,j} |a_{ij}|^2c_m(i,j,k) \Bigg) \cdot ||B||^2
$$
Where
\begin{displaymath}   
c_m(i,j,k) = \left\{     
\begin{array}{lr}       
\Big| \sqrt{\frac{P^{(m)}_{jk}}{P^{(n+m)}_{ik}}} - 1 \Big|^2 & : (i,j,k)\in E(P^n,P^m) \\       
0 & : (i,j,k)\notin E(P^n,P^m)     
\end{array}   \right.
\end{displaymath}

So in order to show that $||(T^{(n)}_{A} - W^{(n)}_{A})Q_m|| \underset{m\rightarrow \infty}{\longrightarrow} 0$, due to finiteness of $P$, it would suffice to show that for all $(i,j,k) \in \St^3$ one has $c_m(i,j,k) \underset{m\rightarrow \infty}{\longrightarrow} 0$. 

Decompose $P$ into stochastic matrices $P_0, ..., P_{r-1}$ as in Theorem \ref{theorem:cyclic-graph-decomposition} and let $\St_0,...,\St_{r-1}$ be the state sets corresponding to the rows of $P_0,...,P_{r-1}$ in the decomposition. Choose any $(i,j,k) \in \St^3$ and assume $i\in \St_{\ell}$, $j\in \St_{\ell+\ell_1}$, $k\in \St_{\ell+\ell_1+\ell_2}$. 
Assume by negation that we have a subsequence $\{ c_{m_{\alpha}}(i,j,k) \}$ composed of non-zero entries only, such that $c_{m_{\alpha}}(i,j,k)$ does not tend to $0$ as ${\alpha}\rightarrow \infty$. Then we must have that $(i,j,k) \in E(P^{n,m_{\alpha}})$ for every ${\alpha}\in \mathbb{N}$. Due to r-periodicity of all states in our matrix, we must have that $n = n'r + \ell_1$ and $m_{\alpha} = m'_{\alpha}r + \ell_2$. So we can find arbitrarily large $m'_{\alpha} \in \mathbb{N}$ such that
$$
\Bigg| \sqrt{\frac{P^{(m'_{\alpha}r + \ell_2)}_{jk}}{P^{((n' +m'_{\alpha})r + \ell_1+\ell_2)}}} - 1 \Bigg|^2 \geq \epsilon
$$
For some fixed $\epsilon >0$. This contradicts the fact that 
$P^{(t r + \ell_2)}_{jk} \underset{t \rightarrow \infty}{\longrightarrow} \pi_k r $ and $P^{(t r + \ell_1 + \ell_2)}_{ik} \underset{t \rightarrow \infty}{\longrightarrow} \pi_k r $, which is due to Theorem \ref{theorem:convergence-theorem-positive-recurrent}.
\end{proof}

Recall that if $\mathcal{A}$ is a gauge invariant C*-subalgebra of $\mathcal{L}(\mathcal{F}_{Arv(P)})$ then $\mathcal{J}(\mathcal{A}):=\mathcal{A} \cap \mathcal{J}$ is an invariant closed left ideal in $\mathcal{A}$.

\begin{theorem} Let $P$ be an irreducible finite stochastic matrix. 
Let
$$\mathcal{T}^c(P) := C^*\Big(\vnMaS \cup \big\{ \ T^{(n)}_{A}, W^{(n)}_{A} \ | \ n\in \mathbb{N}, \  \ A \in Arv(P)_n \ \big\}\Big)$$
Then $\mathcal{J}(\mathcal{T}^c(P))$ is an invariant two sided closed ideal in $\mathcal{T}^c(P)$ so that both
$\mathcal{J}(\mathcal{T}(P))$ and $\mathcal{J}(\mathcal{T}^{\infty}(P))$ are invariant two sided closed ideals in $\mathcal{T}(P)$ and $\mathcal{T}^{\infty}(P)$ respectively, and if $P$ is also irreducible then
$$
\mathcal{O}(P) : = \mathcal{T}(P) / \mathcal{J}(\mathcal{T}(P)) \cong \mathcal{T}^{\infty}(P) / \mathcal{J}(\mathcal{T}^{\infty}(P)) $$
\end{theorem}

\begin{proof}
$\mathcal{T}^c(P)$ is the closure of the linear span of all homogeneous polynomials in the variables $T^{(n)}_{A}, W^{(n)}_{A}$ and their adjoints. Now $\alpha_{\lambda}(W^{(n)}_{A}) = \lambda^nW^{(n)}_{A}$ and $\alpha_{\lambda}(T^{(n)}_{A}) = \lambda^nT^{(n)}_{A}$, and thus one must have that $\mathcal{T}^c(P)$ is an invariant $C$*-subalgebra of $\mathcal{L}(\mathcal{F}_{Arv(P)})$.

We show that $\mathcal{J}(\mathcal{T}^c(P))$ being a closed invariant left ideal, is also a right ideal in $\mathcal{T}^c(P)$. Indeed, Take $S \in \mathcal{J}(\mathcal{T}^c(P))$ and $T \in \mathcal{T}^c(P)$. Assume first that $T$ is homogeneous of degree $m\in \mathbb{Z}$. Then $T: Arv(P)_n \rightarrow Arv(P)_{n+m}$ when $n+m \geq 0$. Thus,
$$ ||STQ_n|| = ||SQ_{n+m}TQ_n|| \leq ||SQ_{n+m}||\cdot ||T|| \underset{n\rightarrow \infty}{\longrightarrow} 0 $$
Now since $\mathcal{T}^c(P)$ is the closure of the linear span of homogeneous polynomials, one has this for general $T \in \mathcal{T}^c(P)$.
This shows that $\mathcal{J}(\mathcal{T}^c(P))$ and $\mathcal{J}(\mathcal{T}^{\infty}(P))$
are closed two sided ideals in their respective algebras.
By Proposition \ref{proposition:in-cuntz-ideal}, we have that 
$$\mathcal{T}^c(P) = \mathcal{T}(P) + \mathcal{J}(\mathcal{T}^c(P)) = \mathcal{T}^{\infty}(P) + \mathcal{J}(\mathcal{T}^c(P))$$
And by Corollary I.5.6 in \cite{Dav} we have that $$\mathcal{T}(P) /\mathcal{J}(\mathcal{T}(P)) \cong \mathcal{T}^{\infty}(P) / \mathcal{J}(\mathcal{T}^{\infty}(P))$$ 
as desired.
\end{proof}

Note that since for all $B \in Arv(P)_m$ we have $W^{(n)}_{A}(B) = A\cdot B$, then a simple calculation shows that $W^{(n)*}_{A} := \big(W_{A}^{(n)}\big)^* : Arv(P)_{n+m} \rightarrow Arv(P)_m $ is given by $$W^{(n)*}_{A}(B) = Gr(P^m) * \big[A^*\cdot B\big]$$
Usually $A^*B$ may not be an element of $Arv(P)_m$ since it may have non zero entries outside the support of $P^m$. Schur multiplication with $Gr(P^m)$ ensures that the output element is in $Arv(P)_m$.

\begin{proposition} \label{proposition:cuntz-direct-sum}
Let $P$ be a finite and essential stochastic matrix over $\St$. Assume that $P$ decomposes into irreducible stochastic matrices $P(1),P(2),..., P(\ell)$ in block diagonal form, then
$$\mathcal{O}(P) \cong \mathcal{O}(P(1))\oplus ... \oplus \mathcal{O}(P(\ell))$$
\end{proposition}

\begin{proof}
This follows immediately from Propositions \ref{proposition:essential-is-direct-sum} and the iteration of Proposition \ref{proposition:essential-cuntz-direct-sum}.
\end{proof}

The last proposition enables us to reduce the problem of computing the Cuntz-Pimsner algebra of $Arv(P)$ for finite essential $P$, to that of computing the Cuntz-Pimsner algebra of $Arv(P)$ when $P$ is finite and irreducible. Thus, we assume throughout the following discussion that $P$ is irreducible, unless stated otherwise. Let $P$ be a $d \times d$ irreducible stochastic matrix with period $r\geq 1$. Denote $q = \frac{d}{r}$.

Let $P$ be an irreducible $r$-periodic $d\times d$ matrix. Use Theorem \ref{theorem:cyclic-graph-decomposition} to write $P$ in the following form with appropriate enumeration:
$$ 
P = \left[\begin{smallmatrix}
       0 & P_0  & \cdots & 0 \\
       \vdots  & \ddots & \ddots & \vdots  \\
       0 & \cdots  & 0 & P_{r-2} \\
       P_{r-1} & \cdots & 0  & 0
\end{smallmatrix}\right]
$$

\begin{notation}
With reference to the decomposition of $P$ above, for 
$A_k \in M_{q}(\mathbb{C})$ for every $0 \leq k \leq r-1$ we define the matrices
$$
A_0 \oplus ... \oplus A_{r-1} := 
\left[\begin{smallmatrix}
       A_0 & 0 & \cdots & 0 \\
       0 & A_1 & \cdots & 0 \\
       \vdots & \vdots & \ddots & \vdots  \\
       0 & 0 & \cdots & A_{r-1}
\end{smallmatrix}\right], \qquad
F : = 
\left[\begin{smallmatrix}
       0 & I_{q} & \cdots & 0 \\
       \vdots   & \ddots & \ddots & \vdots  \\
       0 & \cdots & 0 & I_{q} \\
       I_{q} & \cdots & 0 & 0
\end{smallmatrix}\right]
$$
Where $F$ is $d \times d$ with $I_{q} \in M_{q}(\mathbb{C})$ the $q\times q$ identity matrix.
\end{notation}

Thus, we note that for every $A \in Arv(P)_n$ one has the following decomposition
\begin{equation} \label{eq:decomp}
A = \big( A_0 \oplus ... \oplus A_{r-1} \big) \cdot F^n
\end{equation}
For unique $A_0, ... , A_{r-1} \in M_{q}(\mathbb{C})$, due to the matrix decomposition of $P$ assumed previously. 

\begin{remark}
Note that the matrices $A_0,...,A_{r-1}$ can have non-zero entries only in non-zero entries of the matrices comprising the cyclic decomposition of $P^n$ (For $n=1$ these are just $P_0,...,P_{r-1}$). Due to finiteness and positive-recurrence of $P$ (recall remark \ref{remark:recurrence-finite}) along with Theorem \ref{theorem:convergence-theorem-positive-recurrent}, there exists $n_0$ such that for all $n \geq n_0$, the matrices comprising the cyclic decomposition of $P^n$ have \emph{all} entries non-zero. Thus for $n \geq n_0$ \emph{any} $ A_0,...,A_{r-1} \in M_{q}(\mathbb{C})$ can be taken in equation \eqref{eq:decomp} so as to obtain an element in $Arv(P)_n$.
\end{remark}

\begin{remark} \label{remark:multiplication-in-quotient-red}
Due to the previous remark on equation \eqref{eq:decomp}, there exists $m_0 \in \mathbb{N}$ such that for all $m \geq m_0$ and $B\in Arv(P)_m$ one actually has
$$W^{(n)*}_{A}(B) = A^*\cdot B$$
Which simply means that the matrix $A^*\cdot B$ is already in $Arv(P)_m$, and one does not have to Schur-multiply it with $Gr(P^m)$, to obtain an element in $Arv(P)_m$.
\end{remark}

\begin{defi}
Let $P$ be an irreducible $r$-periodic $d \times d$ stochastic matrix with cyclic decomposition as above. For $A_k \in M_{q}(\mathbb{C})$ with $0\leq k \leq r-1$ and $n\in \mathbb{N}$ the operators $M_{A_0 \oplus ... \oplus A_{r-1}} : Arv(P)_m \rightarrow Arv(P)_m$ and $S_n : Arv(P)_m \rightarrow  Arv(P)_{n+m}$ by
$$ 
M_{A_0 \oplus ... \oplus A_{r-1}}(B) = 
Gr(P^m) * \big[
(A_0 \oplus ... \oplus A_{r-1})  \cdot B \big]
 \ \ and \ \  
S_n(B) = 
Gr(P^{n+m}) * \big[
F^n  \cdot B \big]
$$
for $B \in Arv(P)_m$.
\end{defi}

\begin{proposition}
Let $P$ be an irreducible $r$-periodic $d \times d$ stochastic matrix with cyclic decomposition as above. For every $ A_k \in M_{q}(\mathbb{C}) $ with $0 \leq k \leq r-1$ and $n\in \mathbb{N}$, we have that $M_{A_0 \oplus ... \oplus A_{r-1}} \ , \ \ S_n \in \mathcal{T}^{\infty}(P)$.
\end{proposition}

\begin{proof}
Let $k\in \mathbb{N}$ be large enough so that,
$$F^n \in Arv(P)_{kr+n} \ \ and \ \ I_d \  , \ 
A_0 \oplus ... \oplus A_{r-1} \in Arv(P)_{kr} $$
Where $I_d$ is the identity matrix in $M_{d}(\mathbb{C})$. Then we have,
$$ M_{A_0 \oplus ... \oplus A_{r-1}} = W^{(kr)*}_{I_d} \circ W^{(kr)}_{A_0 \oplus ... \oplus A_{r-1}}, \qquad
S_n = W^{(kr)*}_{I_d} \circ W^{(kr+n)}_{F^n} 
$$
hence both belong to $\mathcal{T}^{\infty}(P)$.
\end{proof}

Denote for $A_0\oplus ...\oplus A_{r-1} \in \oplus_0^{r-1}M_{q}(\mathbb{C})$ the element $\sigma(A_0\oplus ...\oplus A_{r-1}) = A_1\oplus ... \oplus A_{r-1} \oplus A_0 \in \oplus_0^{r-1}M_{q}(\mathbb{C})$ which is the backward cyclic shift. 

\begin{corollary} \label{corollary:spatial-action}
Let $P$ be an irreducible $r$-periodic $d \times d$ stochastic matrix with cyclic decomposition as above. For every $n\in \mathbb{N}$ and $A \in Arv(P)_n$ as in equation \ref{eq:decomp} for unique $A_k \in M_{q}(\mathbb{C})$ with $0\leq k \leq r-1$, one has
$$W^{(n)}_{A} - M_{A_0 \oplus ... \oplus A_{r-1}} \circ S_n \in \mathcal{J} \ \ and \ \ S_1 \circ M_{A_0 \oplus ... \oplus A_{r-1}} \circ S^*_1  - M_{\sigma(A_0 \oplus ... \oplus A_{r-1})} \in \mathcal{J}$$
\end{corollary}

\begin{proof}
Using \ref{remark:multiplication-in-quotient-red} one sees that there exists $m_0 \in \mathbb{N}$ such that for all $m\geq m_0$ and all $B \in Arv(P)_m$ one has
$$ W^{(n)}_{A}(B) = (M_{A_0 \oplus ... \oplus A_{r-1}} \circ S_n)(B) \ \ and \ \ (S_1 \circ M_{A_0 \oplus ... \oplus A_{r-1}} \circ S^*_1)(B)  = M_{\sigma(A_0 \oplus ... \oplus A_{r-1})}(B)$$
\end{proof}

Recall the definition of $\mathcal{O}_0(P)$ given in remark \ref{cuntz-homogeneous-action} along with the surrounding discussion.
We denote by $\overline{T} \in \mathcal{O}(P)$ the image of $T\in \mathcal{T}^c(P)$ under the canonical quotient map $q: \mathcal{T}^c(P) \rightarrow \mathcal{O}(P)$

\begin{corollary} \label{proposition:element-decomp}
Let $P$ be an $r$-periodic finite $d\times d$ irreducible stochastic matrix with cyclic decomposition as above. Then,
$$\mathcal{O}_{0}(P) = \{ \ \overline{M_{A_0 \oplus ... \oplus A_{r-1}}} \ | \ A_k \in M_{q}(\mathbb{C}) \ \}$$
\end{corollary}

\begin{proposition} \label{proposition:irreducible-periodic-cuntz}
Let $P$ be an $r$-periodic finite $d\times d$ irreducible stochastic matrix with cyclic decomposition as above. Then,
$ \mathcal{O}_{0}(P) \cong \bigoplus^r_1 M_{q}(\mathbb{C})$.
Furthermore, if we let $\sigma : \mathbb{Z} \rightarrow Aut(\mathcal{O}_{0}(P))$ be a $\mathbb{Z}$-action given by 
$\sigma_1(\overline{M_{A_0\oplus ... \oplus A_{r-1}}}) = \overline{M_{A_1\oplus ... \oplus A_{r-1} \oplus A_0}}$. Then we have that 
$$\mathcal{O}(P) \cong \Big( \bigoplus^r_1 M_{q}(\mathbb{C}) \Big) \rtimes_{\sigma} \mathbb{Z} $$

\end{proposition}

\begin{proof}
Define a map $\varphi : \oplus_0^{r-1}M_q(\mathbb{C}) \rightarrow \mathcal{O}_{0}(P)$ by $\varphi(A_0 \oplus ... \oplus A_{r-1}) = \overline{M_{A_0 \oplus ... \oplus A_{r-1}}}$ where $A_0\oplus ... \oplus A_{r-1} \in \oplus_0^{r-1}M_q(\mathbb{C})$. It follows that $\varphi$ is a surjective, bounded *-homomorphism, we show that it is 1-1. Assume $0 \neq C = A_0\oplus ... \oplus A_{r-1} \in \oplus_0^{r-1}M_q(\mathbb{C})$. There exists $k_0 \in \mathbb{N}$ such that for all $k \geq k_0$ one has $I_d \in Arv(P)_{kr}$, and thus in $Arv(P)_{kr}$ ,
$$||M_{C}(I_d)||^2 = ||C||^2 = ||\Diag[C^*C]|| $$
Due to Proposition \ref{bound-norm-cuntz} we have that 
$$||\overline{M_C}||^2 \geq \limsup_{k\rightarrow \infty}||M_C \circ Q_{kr}||^2 \geq ||\Diag[C^*C]|| > 0$$
We must have that $\varphi$ is 1-1, and thus it is a *-isomorphism.

For the second part assume $\mathcal{O}(P) \subseteq B(H)$ for some Hilbert space $H$. Due to Corollary \ref{corollary:spatial-action}, one has a faithful $\sigma$-covariant representation $(S_1, \varphi^{-1}, H)$ (in which the action of $\sigma$ is spatially implemented). Thus, by the universal property of crossed products, one attains a *-homomorphism $\Phi: \Big( \bigoplus^r_1 M_{q}(\mathbb{C}) \Big) \rtimes_{\sigma} \mathbb{Z} \rightarrow B(H)$ satisfying $\Phi(\sum_{s\in \mathbb{Z}}A_s s) = \sum_{s\in \mathbb{Z}}\varphi^{-1}(A_s)(S_1)^s$ the image of which is $\mathcal{O}(Arv(P))$, due to Proposition the above. This assures us that this *-homomorphism respects the gauge actions and is injective on $\bigoplus^r_1 M_{q}(\mathbb{C})$ (which is the fixed point algebra of the crossed product). So by Proposition \ref{proposition:injectivity-fixed-point-alg} this *-homomorphism must be injective.
\end{proof}

\begin{corollary} \label{corollary:cuntz-irreducible}
Let $P$ be an irreducible $d\times d$ stochastic matrix. Then $\mathcal{O}(P) \cong M_d(\mathbb{C}) \otimes C(\mathbb{T}) \cong C(\mathbb{T};M_d(\mathbb{C})) $
\end{corollary}

\begin{proof}
By Proposition \ref{proposition:irreducible-periodic-cuntz} we have that $\mathcal{O}(P) \cong \Big( \bigoplus^r_1 M_{q}(\mathbb{C}) \Big) \rtimes_{\sigma} \mathbb{Z}$, and we 
have that
$$
\Big( \oplus^r_1 M_{q}(\mathbb{C}) \Big) \rtimes_{\sigma} \mathbb{Z} \cong \Big( C(\mathbb{Z} / r\mathbb{Z}) \otimes M_{q}(\mathbb{C}) \Big) \rtimes_{\widetilde{\delta}} \mathbb{Z}
$$
where $\widetilde{\delta}$ on $C(\mathbb{Z} / r\mathbb{Z}) \otimes M_{q}(\mathbb{C})$ is given by $\delta \otimes I$, and $\delta$ is the backward cyclic shift of the domain of functions in $C(\mathbb{Z} / r\mathbb{Z})$. Now let $G = \mathbb{Z}$, $H = r\mathbb{Z}$, $A = M_{q}(\mathbb{C})$ and let $\tau$ be the trivial action of $G$ on $A$. Then
by \cite[Corollary 2.8]{PG} we have that
\begin{align*}
\Big( C(\mathbb{Z} / r\mathbb{Z}) \otimes M_{q}(\mathbb{C}) \Big) \rtimes_{\widetilde{\delta}} \mathbb{Z} 
& \cong (C_0(G/H) \otimes A) \rtimes_{\widetilde{\delta}} G \simeq (A \rtimes_\tau H) \otimes K(L^2(G/H, \mu)) \\
& \cong (M_{q}(\mathbb{C}) \rtimes_{\tau} r \mathbb{Z}) \otimes K(L^2( \mathbb{Z} / r \mathbb{Z} ))
\cong C(\mathbb{T}) \otimes M_{q}(\mathbb{C}) \otimes M_{r}(\mathbb{C}) 
\end{align*}
Therefore 
$\mathcal{O}(P) \cong
C(\mathbb{T}) \otimes M_{q}(\mathbb{C}) \otimes M_{r}(\mathbb{C}) \cong C(\mathbb{T}) \otimes M_{d}(\mathbb{C})
\cong C(\mathbb{T};M_d(\mathbb{C}))$.
\end{proof}

\begin{corollary}[Cuntz-Pimsner C*-algebras in the finite essential case] \label{corollary:cuntz-computation} 
Let $P$ be a finite and essential stochastic matrix over $\St$. Then 
$$\mathcal{O}(P) \cong C\Big(\mathbb{T} ; \oplus_{k=1}^{\ell} M_{d_k}(\mathbb{C})\Big)$$
\end{corollary}

\begin{proof}
Assume that $P$ decomposes into irreducible stochastic matrices $P(1),P(2),..., P(\ell)$ in block diagonal form. Due to Proposition \ref{proposition:cuntz-direct-sum} and Corollary \ref{corollary:cuntz-irreducible} we have
$$\mathcal{O}(P) \cong \oplus_{k=1}^{\ell}\mathcal{O}(P(k)) \cong 
\oplus_{k=1}^{\ell} C\big(\mathbb{T};M_{d_k}(\mathbb{C})\big) \cong C\Big(\mathbb{T} ; \oplus_{k=1}^{\ell} M_{d_k}(\mathbb{C})\Big)$$
\end{proof}

Thus,  we obtain
the following result. We omit the 
proof, which is straightforward.

\begin{theorem}[Isomorphism between Cuntz-Pimsner C*-algebras in the finite essential case]
Let $P$ and $Q$ be two finite essential stochastic matrices. Assume
$P(1) \oplus P(2) ... \oplus P(\ell)$ and $Q(1) \oplus Q(2) \oplus ... \oplus Q(s)$
are irreducible decompositions for $P$ and $Q$ respectively, where $P(k) \in M_{d_k}(\mathbb{C})$ and $Q(k')\in M_{t_{k'}}(\mathbb{C})$ are irreducible, and such that $d_1 \leq d_2 \leq ... \leq d_{\ell}$ and $t_1 \leq t_2 \leq ... \leq t_s$. Then $\mathcal{O}(P) \cong \mathcal{O}(Q)$ if and only if $\ell = s$ and $d_k = t_k$ for all $1\leq k \leq \ell$.
\end{theorem}

% *******************************************************************
% *  Tensor algebras for subproduct systems over commutative VN alg *
% *******************************************************************

\section{Tensor algebras and their graded structure}\label{section:tensor-algebras}

In this section we begin the study of the \emph{non-self-adjoint} norm-closed tensor algebra associated to a subproduct system over $\mathbb{N}$. 

\begin{defi}
Let $X = (X_n)_{n\in \mathbb{N}}$ be a subproduct system over a W*-algebra  
$\vnM$. The \emph{tensor algebra associated to $X$} is the norm-closed subalgebra
 of $\mathcal{L}(\mathcal{F}_X)$ given by
$$
\mathcal{T}_{+}(X): = \overline{\mathrm{Alg}}\left\{ \vnM \cup \{  S^{(n)}_{\xi} \; | \; \xi \in X_n, \; n\in \mathbb{N}   \} \right\}
$$
The tensor algebra is clearly a gauge invariant closed
 subalgebra of $\mathcal{T}(X)$,
and the Fourier coefficient maps $\Phi_n$ restrict to idempotents.
Therefore the tensor algebra is graded by the
 spaces $\mathcal{T}_+(X)_n  =  
\Phi_n(\mathcal{T}_+(X)) =
\overline{\Span} \{\, S^{(n)}_{\xi} \mid \xi \in X_n \, \}$.
\end{defi}

For a subproduct system $X$, we denote by $\Omega_X$ the vacuum vector of the Fock module of $X$, which is simply $1_{\vnM} \oplus 0 \in \vnM \oplus \bigoplus_{n=1}^{\infty}X_n$.

\begin{proposition} \label{proposition:banach-bimodule}
Let $(X,U)$ be a subproduct system. For every $x\in \mathbb{N}$ 
we have that $X_n$ is isometrically isomorphic 
as a Banach $\vnM$-bimodule to $\mathcal{T}_+(X)_n$
via the map $\xi \mapsto S^{(n)}_\xi$.

Therefore, every element $T\in \mathcal{T}_+(X)$ 
has a unique representation as an infinite series
$T = \sum_{n = 0}^{\infty}S^{(n)}_{\xi_n}$
where $\xi_n \in X_n$ satisfies $\Phi_n(T) = S^{(n)}_{\xi_n}$
(called its Fourier series representation for short),  and the series 
converges Cesaro to $T$ in norm: 
if 
$\sigma_N(T) = \sum_{n=0}^{N}\Big(1 - \frac{n}{N+1} \Big)S^{(n)}_{\xi_n}$,
then we have that
$\lim_{N\to \infty} \| \sigma_N(T) - T\| = 0$.
Furthermore, if $T, T' \in \mathcal{T}_+(X)$ have Fourier series representations
$T = \sum_{i=0}^{\infty}S^{(i)}_{\xi_i}$ and $T' = \sum_{i=0}^{\infty}S^{(i)}_{\eta_i}$, then 
$$
TT' = \sum_{n=0}^{\infty}S^{(n)}_{\zeta},
\quad \text{where} \quad
\zeta= \sum^{n}_{k=0}U^X_{k,n-k}(\xi_k \otimes \eta_{n-k}).
$$
\end{proposition}

\begin{proof}
Define $\Psi_n : X_n \rightarrow \mathcal{T}_+(X)_n$ by
$\Psi_n(\xi) = S^{(n)}_{\xi}$ for all $\xi \in X_n$. For every $\sum^M_{m=0}\eta_m$ with $\eta_m \in X_m$,
$$
|S^{(n)}_{\xi}\big(\sum^M_{m=0}\eta_m \big)|^2 = |\sum^M_{m=0}U_{n,m}(\xi \otimes \eta_m)|^2 = \sum^M_{m=0}|U_{n,m}(\xi \otimes \eta_m) |^2 \leq \sum^M_{m=0}|\xi \otimes \eta_m|^2 = |\xi \otimes \big(\sum^M_{m=0}\eta_m \big)|^2
$$
Thus, taking norms we obtain that 
$$
\|S^{(n)}_{\xi} \big(\sum^M_{m=0}\eta_m \big)\| \leq \|\xi \otimes \big(\sum^M_{m=0}\eta_m \big)\| \leq \| \xi \| \cdot  \| \sum^M_{m=0}\eta_m \| 
$$
Thus, $\|S^{(n)}_{\xi}\| \leq \|\xi\|$.
Moreover, $\|S^{(n)}_{\xi}(\Omega_X)\| = \|\xi\|$
hence $\Psi_n$ is an isometry. In addition, $\Psi_n$ preserves the left 
and right actions of $\vnM$: for $m\in \vnM$ and $\eta \in X_m$,
\begin{align*}
(m \cdot \Psi_n(\xi))(\eta) & = m \cdot U_{n,m}(\xi \otimes \eta) = U_{n,m}(m 
\cdot \xi \otimes \eta) = \Psi_n(m \cdot \xi) \\
(\Psi_n(\xi) \cdot m)(\eta) & = \Psi_n(\xi)(m \cdot \eta) = U_{n,m}(\xi 
\otimes m\cdot \eta) = U_{n,m}(\xi \cdot m \otimes \eta) = (\Psi_n(\xi \cdot m))(\eta)
\end{align*}
Thus, $\Psi_n$ is an $\vnM$-bimodule isometry with a dense image in 
$\mathcal{T}_+(X)_n$. Since the image of an isometry is closed, $\Psi_n$ must be onto, and we have proven that $\Psi_n$ has the desired properties.

Let $T\in \mathcal{T}_+(X)$. By the previous paragraph, 
there exist unique $\xi_n \in X_n$ such that $\Phi_n(T) = S^{(n)}_{\xi_n}$, and as discussed in Section~\ref{section:cuntz-pimsner},
the series converges to $T$ in the Cesaro sense.
Thus the Fourier series decomposition is unique, in the
sense that any two elements of $\tensor(X)$ with the same
Fourier series must be identical.

Finally, let  $T, T' \in \mathcal{T}_+(X)$
have Fourier series representations
$T = \sum_{i=0}^{\infty}S^{(i)}_{\xi_i}$ and $T' = \sum_{i=0}^{\infty}S^{(i)}_{\eta_i}$. 
\begin{align*}
\sigma_N(T) \sigma_N(T') & = \bigg(\sum_{n=0}^{N}\Big(1- \frac{n}{N+1} \Big)S^{(n)}_{\xi_n} \bigg)\bigg(\sum_{j=0}^{N}\Big(1- \frac{j}{N+1} \Big)S^{(j)}_{\eta_j} \bigg) \\
& =
\sum_{m=0}^N \sum_{k=0}^m \Big(1 - \frac{k}{N+1} \Big) S^{(k)}_{\xi_k} \Big( 1- \frac{m-k}{N+1} \Big)S^{(m-k)}_{\eta_{m-k}} \\
& = 
\sum_{m=0}^N \Big( 1- \frac{m}{N+1} \Big) \sum_{k=0}^m \frac{\Big(1 - \frac{k}{N+1} \Big)\Big( 1- \frac{m-k}{N+1} \Big)}{\Big( 1- \frac{m}{N+1} \Big)}S^{(m)}_{U_{k,m-k}(\xi_k\otimes \eta_{m-k})}
\end{align*}
Thus,
$$
\Phi_n(TT') = \lim_{N\to\infty} \Phi_n(\sigma_N(T) \sigma_N(T')) = S^{(n)}_{\zeta}, \quad \text{where} \quad \zeta=\sum_{k=0}^nU_{k,n-k}(\xi_k\otimes \eta_{n-k}).
$$
By uniqueness, the proposition is proven.
\end{proof}

\subsection*{Automatic continuity}
We are interested in the study of isomorphisms between tensor algebras. Under certain conditions, we will show that
algebraic isomorphisms are automatically bounded. 
We will follow closely the ideas of Donsig, Hudson and Katsoulis~\cite{donsig-hudson-katsoulis}, Katsoulis and Kribs~\cite{katsoulis-kribs}, and
Davidson and Katsoulis~\cite{davidson-katsoulis}.

Suppose that $\cala$ and $\calb$ are Banach algebras,
and suppose that $\varphi: \cala \to \calb$ is a surjective
homomorphism. Let 
$$
\cals(\varphi) = \{ b \in \calb \mid \exists (a_n)\subseteq \cala 
\text{ such that } 
a_n \to 0 \text{ and } \varphi(a_n) \to b \}.
$$
It is easy to see that the graph of $\varphi$ is closed if
and only if $\cals(\varphi) = \{0\}$, hence by the
closed graph theorem we have that $\phi$ is continuous
if and only if $\cals(\varphi) = \{0\}$. We will use the
following adaptation of \cite[Lemma 2.1]{sinclair},
which appeared first in Donsig, Hudson and Katsoulis~\cite{donsig-hudson-katsoulis}.

\begin{lemma}[Sinclair]\label{lemma:sinclair}
Suppose that $\cala$ and $\calb$ are Banach algebras and
$\varphi: \cala \to \calb$ is a surjective 
algebraic homomorphism. Let $(b_n)_{n\in\nn}$ be
any sequence in $\calb$. Then there exists $N \in \nn$
such that for all $n\geq N$,
$$
\overline{b_1 b_2 \dots b_n \cals(\phi)} = 
\overline{b_1 b_2 \dots b_{n+1} \cals(\phi)}
\quad \text{and} \quad
\overline{\cals(\varphi)b_n b_{n-1} \dots b_1} = 
\overline{\cals(\varphi)b_{n+1} b_{n} \dots b_1}.
$$
\end{lemma}

\begin{theorem}\label{thm:automatic}
Let $X$ be a subproduct system over a
von Neumann algebra $\calm$ and let 
$\cala$ be a Banach algebra. 
Suppose that for all $d \geq 0$ and
$0\neq y\in \tensor(X)_d$ there exists a sequence
$(b_n) \in \ker \Phi_0$ such that $y b_n b_{n-1} \dots b_1 \neq 0$ for all $n\geq 1$. 
Then any surjective algebraic homomorphism 
$\varphi: \cala \to \tensor(X)$ is continuous.
\end{theorem}
\begin{proof}   
Suppose towards a contradiction that there exists 
$0 \neq z\in \cals(\varphi)$. Let $d\in \nn$ be minimal such that $y := \Phi_d(z) \neq 0$,
and let $y = \Phi_d(z)$. By assumption, there
exists a sequence $(b_n) \subseteq \ker \Phi_0$ such
that for every $n$, $y b_n b_{n-1} \cdots b_1 \neq 0$,
and therefore $z  b_n b_{n-1} \cdots b_1\neq 0$.
Notice that since $b_n \in \ker \Phi_0$ for all n, 
we have that for all $n$ and every $k < n$, and every $w \in \cals(\varphi)$,
$$
\Phi_k( w b_n b_{n-1} \cdots b_1) =  0
$$
It follows that 
$$
\overline{\cals(\varphi) b_n b_{n-1} \cdots b_1} \subseteq 
\bigcap_{k<n} \ker \Phi_k
$$
However, by Lemma~\ref{lemma:sinclair}, there exists $N$ such that for all $n\geq N$,
$$
\overline{\cals(\varphi) b_N b_{N-1} \cdots b_1} =
\overline{\cals(\varphi)b_n b_{n-1} \cdots b_1} 
\subseteq 
\bigcap_{k<n} \ker \Phi_k.
$$
Thus we obtain that
$$
\overline{\cals(\varphi)b_N b_{N-1} \cdots b_1} \subseteq 
\bigcap_{k=1}^\infty \ker \Phi_k = \{0 \},
$$
and so we reach a contradiction to the fact that
$z b_N b_{N-1} \cdots b_1 \neq 0$.
\end{proof}

This approach  was used in
 \cite{katsoulis-kribs, davidson-katsoulis}, replacing a right acting sequence by a left acting one,  in the special case when there exists $b \in \ker \Phi_0$ that satisfies $\|by\| = \|y\|$ for all $y \in \tensor(X)_d$. That is analogous to the following corollary.

\begin{corollary}\label{cor:power-sequence}
Let $X$ be a subproduct system over a
von Neumann algebra $\calm$ and let 
$\cala$ be a Banach algebra. 
Suppose that 
for every $0\neq y\in \tensor(X)$ there exists 
$b \in \ker \Phi_0$ such that $ y b^n \neq 0$ and
for all $n\geq 1$. 
Then any surjective algebraic homomorphism 
$\varphi: \cala \to \tensor(X)$ is continuous.
\end{corollary}

\subsection*{Admissible isomorphisms}
Our analysis of isomorphisms between tensor algebras
is most effective for the following class of isomorphisms.

\begin{defi}[Admissibility] \label{deifnition:admissibility}
Let $X$ and $Y$ be subproduct systems over a W*-algebra $\vnM$. Let $\varphi: \mathcal{T}_+(X) \rightarrow \mathcal{T}_+(Y)$ be an isomorphism. Denote for $n,m \in \mathbb{N}$ the maps $\varphi^{(n,m)} : = \Phi_m \circ \varphi \upharpoonright_{\mathcal{T}_+(X)_n}$. $\varphi$ is said to be \emph{admissible} if the following conditions hold:
\begin{enumerate}
\renewcommand{\labelenumi}{(A\theenumi)}
\item
The maps $\rho_{\varphi} := \Phi_0 \circ \varphi \upharpoonright_{\vnM}$ and $\rho_{\varphi^{-1}} := \Phi_0 \circ (\varphi^{-1}) \upharpoonright_{\vnM}$ are *-automorphisms of $\vnM$.
\item
For all $n,m\in \nn$ such that $m<n$ we have that the maps $\varphi^{(n,m)}$ and $(\varphi^{-1})^{(n,m)}$ are continuous in the $\sigma$-topology induced by the identification of $\mathcal{T}_+(X)_k$ as a W*-correspondence as in Proposition \ref{proposition:banach-bimodule}.
\end{enumerate}
\end{defi}

\begin{lemma}\label{lemma:isometric-automatic-admissibility}
Let $X$ and $Y$ be subproduct systems over a W*-algebra $\vnM$,
and let $\varphi : \mathcal{T}_+(X) \rightarrow \mathcal{T}_+(Y)$ be an isomorphism. If $\varphi$ is \emph{isometric} then it is admissible.
\end{lemma}

\begin{proof}
If $\varphi : \mathcal{T}_+(X) \rightarrow \mathcal{T}_+(Y)$ is an isometric isomorphism, 
since $\mathcal{T}_+(Y) \subseteq 
\mathcal{T}(Y)$ we can regard $\varphi$ as a map into
the Toeplitz algebra. In our case, $\varphi 
\upharpoonright _{\vnM} : \vnM \to
\mathcal{T}(Y)$ is a unital contractive homomorphism, hence
it is necessarily positive. Thus it preserves the involution from $\vnM$ to $\mathcal{T}(Y)$. 
Since we also have that $\varphi(\vnM) = \varphi(\vnM)^* \subseteq 
\mathcal{T}_+(Y)^*$ (considered inside $\mathcal{T}(Y)$), 
we must have that $\varphi(\vnM) 
\subseteq \mathcal{T}_+(Y) \cap \mathcal{T}_+(Y)^* = \vnM$. 
A similar argument then shows that $\varphi^{-1}(\vnM) \subseteq \vnM$, so that $\varphi(\vnM) = \vnM$ and $\varphi$ satisfies (A1). To show (A2) note that since $\varphi(\vnM) = \vnM$ and $\varphi$ is multiplicative, we must have that $\varphi^{(n,m)}$ is a $\rho_{\varphi}$-correspondence map for all $n,m\in \mathbb{N}$. But we have seen after defining $\rho$-correspondence maps that they are necessarily $\sigma$-topology continuous. The identical argument then works for $\varphi^{-1}$.
\end{proof}

When an isomorphism $\varphi : \mathcal{T}_+(X) \rightarrow \mathcal{T}_+(Y)$ fails to be isometric, then it may fail to satisfy condition (A1) even if we assume that $\rho_{\varphi}$ and $\rho_{\varphi^{-1}}$ are bijective, since $\rho_{\varphi}$ might not be *-preserving, or equivalently contractive. 
Nevertheless, if in addition $\vnM$ is commutative, the following lemma provides automatic contractivity of $\rho_{\varphi}$ and $\rho_{\varphi^{-1}}$, assuring (A1) holds.

\begin{lemma}\label{lemma:automatic-contractivity}
Let $A$ be a unital Banach algebra with $\|1_A\|=1$, and let $B$ be a unital Banach subalgebra of a commutative 
C*-algebra $B'$. If $\varphi: A \rightarrow B$ is a 
unital algebraic homomorphism, then $\varphi$ is contractive.
\end{lemma}
\begin{proof}
We may assume without loss of generality that $B'$ is unital and $1_B' \in B$
since $C^*(B) \subseteq B'$ is a unital commutative C*-algebra.
For all $a\in A$, the element $\varphi(a) \in B'$ must be normal, 
one has
$$
\|\varphi(a)\| = r_{B'}(\varphi(a)) = r_{B}(\varphi(a)) \leq r_{A}(a) \leq \|a\|,
$$
where the second equality follows from spectral permanence and the first inequality follows from unitality of $\varphi$.
\end{proof}

We were not able to prove that composition of admissible isomorphisms is always admissible. In the cases when
our analysis is most successful this is not an issue, 
because admissibility will be shown to be automatic. 
We already showed this fact for
the isometric case, and it will also hold for tensor
algebras arising from finite 
stochastic matrices (see Section~\ref{section:tensor-algebras-stochastic}).

\subsection*{Graded bounded admissible isomorphisms}
The most basic examples of bounded admissible isomorphisms 
between tensor
algebras are graded ones. These are 
easy to characterize.

\begin{defi}
Let $X$ and $Y$ be subproduct systems over a W*-algebra $\vnM$.  An isomorphism $\varphi:\tensor(X) \to \tensor(Y)$
is said to be \emph{graded} if $\Phi_n\varphi \Phi_n= \varphi \Phi_n$ for all $n$.
\end{defi}

\begin{remark} \label{remark:contractive-graded-is-admissible}
Admissibility of a graded isomorphism $\varphi$ is reduced to showing that $\rho_{\varphi} = \varphi \upharpoonright_{\vnM}$ is contractive. Indeed, in that case, we have that (A1) holds since then $\rho_{\varphi}$ is a *-automorphism (and as a consequence so is $\rho_{\varphi^{-1}} = \rho_{\varphi}^{-1}$), Furthermore, (A2) is also satisfied since $\varphi^{(n,m)} = 0 = (\varphi^{-1})^{(n,m)}$ when $m <n$. See also Proposition \ref{proposition:admissibility-semi-graded}.
\end{remark}

\begin{defi}
Let $(X,U^X)$ and $(Y,U^Y)$ be subproduct systems over the same W*-
algebra $\vnM$, and let $\rho$ be a *-automorphism of $\vnM$. A family 
$V=\{ V_n \}_{n \in \mathbb{N}}$ of maps 
$V_n :X_n \rightarrow Y_n$ is called a \emph{$\rho$-similarity}  
$V: X \to Y$ of subproduct systems if 
$V_0 =\rho$, and for $0\neq 
n \in \mathbb{N}$ the maps $V_n$ are bijective $\rho$-correspondence 
morphisms satisfying $\sup_{n\in \mathbb{N}}\{||V_n||, ||V_n^{-1}||\} 
< \infty$, such that
$$
V_{n+m}U^X_{n,m} = U^Y_{n,m}(V_n \otimes V_m).
$$
\end{defi}

If $V:X \to Y$ is a $\rho$-similarity, we automatically
have that $V^{-1} : Y \to X$ is a $\rho^{-1}$-similarity, 
since $V_0^{-1} = \rho^{-1}$ and 
$V^{-1}_{n+m}U^Y_{n,m} = U^X_{n,m}(V^{-1}_n \otimes V^{-1}_m)$
for all $n,m \in \mathcal{N}$.

Note also that if $\{V_n\}$ is a $\rho$-similarity such that 
$\|V_n\| = \|V_n^{-1}\| = 1$ for all $n\in \mathbb{N}$, 
then it is a unitary $\rho$-isomorphism of subproduct systems.  

When we have a similarity $\rho$-isomorphism $V: X \rightarrow Y$, we can define a natural $\rho$-correspondence morphism from 
$\mathcal{F}_X$ to $\mathcal{F}_Y$ as follows. The map
$$
W_V = \oplus_{n=0}^{\infty}V_n
$$
is a well-defined map between the Fock C*-direct sums,
since $\sup_{n\in \mathbb{N}}\{||V_n||, ||V_n^{-1}||\} < \infty$. 
Moreover, it extends uniquely to a map which we also denote by 
$W_V$ in $\mathcal{L}(\mathcal{F}_X, \mathcal{F}_Y)$ between the W*-direct sums, which is a $\rho$-correspondence morphism. 
Furthermore, we have that $W_V^{-1} = W_{V^{-1}}$ is  $\rho^{-1}$-correspondence morphism and 
\begin{equation}\label{eq:joining-it-all-up}
\|W_V\| \leq \sup_{n\in \mathbb{N}} \|V_n\|  \quad \text{and} \quad \|W_V^{-1}\| \leq \sup_{n\in \mathbb{N}} \|V^{-1}_n\|.
\end{equation}

\begin{proposition} \label{proposition:graded-tensor-isomorphism}
Let $X$ and $Y$ be subproduct systems over a W*-algebra $\vnM$. 
\begin{enumerate}
\item Let $\rho$ be a *-automorphism of $\vnM$.
If $V : X \rightarrow Y$ is a $\rho$-similarity, then $\Ad_V :
\mathcal{T}_+(X) \rightarrow \mathcal{T}_+(Y)$ given by $\Ad_V(T) = 
W_V T W_V^{-1}$ for every $T\in \mathcal{T}_+(X)$ is a graded completely bounded 
admissible isomorphism satisfying $\Ad_V \upharpoonright _{\vnM} = 
\rho$ and
$$
\max\{ \|\Ad_V\|_{cb}, \|\Ad_V^{-1}\|_{cb}\} \leq  \sup_{n\in \mathbb{N}} \|V_n\| \cdot
\sup_{n\in \mathbb{N}} \|V_n^{-1}\|.  
$$

\item
Let $\varphi : \mathcal{T}_+(X) \to \mathcal{T}_+(Y)$ be an admissible \emph{graded} bounded isomorphism, and let $\rho_\varphi = \varphi \restriction_\vnM$. 
Then the family  
$V^\varphi=(V^\varphi_n : X_n \to Y_n)$ uniquely
determined by $S^{(n)}_{V^{\varphi}_n(\xi)} = \varphi(S^{(n)}_\xi)$ for $\xi \in X_n$ constitutes a $\rho_\varphi$-similarity satisfying
$$
\sup_{n\in \mathbb{N}} \|V^{\varphi}_n\| 
\leq  \|\varphi\| 
\quad \text{and} \quad 
 \sup_{n\in \mathbb{N}} \|(V^{\varphi}_n)^{-1}\| \leq \|\varphi^{-1}\|.
$$
Thus, if $\varphi:\tensor(X) \rightarrow \tensor(Y)$ is an admissible graded bounded isomorphism, then $\varphi =\Ad_{V^\varphi}$ and $\varphi$ is completely bounded with $\| \varphi \|_{cb} \leq \| \varphi \| \| \varphi^{-1} \|$.
\end{enumerate}
\end{proposition}

\begin{proof}
(1) Let us assume that $\{ V_n \}$ is a family of maps constituting a 
$\rho$-similarity from $X$ to $Y$. The map $\Ad_V$ is 
multiplicative and bounded by \eqref{eq:joining-it-all-up}.
Since $V^{-1}$ is a $\rho^{-1}$-similarity, the map  $\Ad_{V^{-1}}$, 
is also a bounded graded homomorphism 
from $\mathcal{T}_+(Y)$ into $\mathcal{T}_+(X)$
and $\Ad_{V^{-1}} = 
(\Ad_V)^{-1}$, 
ensuring that $\Ad_V$ is a graded bounded admissible \emph{isomorphism} between $\mathcal{T}_+(X)$ and $\mathcal{T}_+(Y)$. The norm inequalities
follow trivially from equation~\eqref{eq:joining-it-all-up}.

(2) Let us assume that $\varphi: \mathcal{T}_+(X) \rightarrow \mathcal{T}_+(Y)$ is a  graded bounded admissible bijective homomorphism. Notice that $V^{\varphi}_n : X_n \rightarrow Y_n$ 
in the assertion can be written as 
$V^{\varphi}_n = \Psi_n^{-1}\varphi \Psi_n,$
where $\Psi_n(\xi) = S^{(n)}_\xi$ is the map considered in
Proposition~\ref{proposition:banach-bimodule}.
Notice that $V^{\varphi}_0= \rho_{\varphi}$ is a *-automorphism,
and the proof of the identities $V^{\varphi}_{n+m}U^X_{n,m} = U^Y_{n,m}(V^{\varphi}_n \otimes V^{\varphi}_m)$ 
is straightforward. Moreover,
$$
\|V^{\varphi}_n(\xi)\| = \|(\Psi_n^{-1}\varphi \Psi_n)(\xi)|| = \|S^{(n)}_{(\Psi_n^{-1}\varphi \Psi_n)(\xi)}\| = \|\varphi(S^{(n)}_{\xi})\| \leq \|\varphi\| \cdot \|S^{(n)}_{\xi}\| = \|\varphi\| \cdot \|\xi\|.
$$
Notice that $V^{\varphi^{-1}}_n = (V^{\varphi}_n)^{-1}$. Hence by a similar argument for $(V^{\varphi}_n)^{-1}$ we obtain that $\|V^{\varphi}_n\|$ and $\|(V^{\varphi}_n)^{-1}\|$ are uniformly bounded by $\|\varphi\|$ and $\|\varphi^{-1}\|$ respectively. 

For the last part of the proposition, using the norm estimates of item (2) and item (1) in tandem while noting that $\varphi = \Ad_{V^{\varphi}}$, we arrive at the completely bounded norm estimate for $\varphi$.
\end{proof}

We obtain the following immediate corollary by noting that isometric isomorphisms are automatically admissible, or by Remark \ref{remark:contractive-graded-is-admissible}.

\begin{corollary} \label{corollary:isometric-isomorphism}
Let $X$ and $Y$ be subproduct systems over a W*-algebra $\vnM$.
\begin{enumerate}
\item
Let $\rho$ be a *-automorphism of $\vnM$. 
If $V : X \rightarrow Y$ is a unitary $\rho$-isomorphism, then 
$\Ad_V :\mathcal{T}_+(X) \rightarrow \mathcal{T}_+(Y)$ given by 
$\Ad_V(T) = W_V T W_V^{-1} = W_V T W_V^*$ for every $T\in 
\mathcal{T}_+(X)$ is a completely isometric isomorphism satisfying $\Ad_V 
\upharpoonright _{\vnM} = \rho$. 
\item
Let $\varphi : \mathcal{T}_+(X) \rightarrow \mathcal{T}_+(Y)$ be a 
\emph{graded} isometric isomorphism, 
and let $\rho_\varphi = \varphi \restriction_\vnM$. Then the family 
$V^\varphi=(V^\varphi_n : X_n \to Y_n)$ uniquely
determined by $S^{(n)}_{V^{\varphi}_n(\xi)} = \varphi(S^{(n)}_\xi)$ for $\xi \in X_n$ constitutes a unitary $\rho_{\varphi}$-isomorphism of subproduct systems satisfying $\varphi = \Ad_{V^\varphi}$.
\end{enumerate}
\end{corollary}

\subsection*{Semi-graded bounded admissible isomorphisms}
We will consider a special class of bounded admissible
isomorphisms which will provide a convenient platform for the
analysis of the general case.

\begin{defi}
Let $X$  be a subproduct system. The minimal degree of an element $0\neq T \in \mathcal{T}_+(X)$, also denoted by $\mindeg(T)$ is the smallest $n\in \mathbb{N}$ such that $\Phi_n(T) \neq 0$.

Let $Y$ be another subproduct system, and suppose that $\varphi: \mathcal{T}_+(X) \rightarrow \mathcal{T}_+(Y)$ is a bounded isomorphism. We say that $\varphi$ is \emph{semi-graded} if for all $T\in \mathcal{T}_+(X)$ 
$$
\mindeg(\varphi(T)) = \mindeg(T),
$$ 
or equivalently, for every $T\in \mathcal{T}_+(X)$ 
and $S \in \mathcal{T}_+(Y)$,
$$
\mindeg(\varphi(T)) \geq \mindeg(T) \quad \text{and} \quad  \mindeg(\varphi^{-1}(S)) \geq \mindeg(S).
$$

\end{defi}

\begin{proposition} \label{proposition:admissibility-semi-graded}
Let $X,Y$ be subproduct systems over a W*-algebra $\vnM$ 
and suppose that $\varphi: 
\mathcal{T}_+(X) \rightarrow \mathcal{T}_+(Y)$ is a 
semi-graded bounded isomorphism. Then $\rho_{\varphi}^{-1} = \rho_{\varphi^{-1}}$ and if $\rho_{\varphi}$ is contractive then $\varphi$ is admissible. In particular, if $\vnM$ is
commutative, then $\varphi$ is automatically admissible.
\end{proposition}

\begin{proof}
Note that since $\varphi$ is
a semi-graded isomorphism, if $T \in \tensor(X)$, then
$$
\Phi^Y_0 \varphi(T) = \Phi^Y_0 \varphi \Phi^X_0 (T).
$$
Thus $\rho_{\varphi}$ is surjective and in fact, since the 
same argument works for $\varphi^{-1}$ we have that 
$\rho_{\varphi}^{-1} = \rho_{\varphi^{-1}}$. Indeed, for 
$m\in \vnM$ we have $(\rho_{\varphi^{-1}} \circ 
\rho_{\varphi}) (m) = \rho_{\varphi^{-1}} (\Phi_0 
\varphi(m)) = \Phi_0\varphi^{-1}\Phi_0 \varphi(m) = 
\Phi_0\varphi^{-1}\varphi(m) = m$. Thus if $\rho_{\varphi}$ 
is a contractive bijective map then both $\rho_{\varphi}$ 
and $\rho_{\varphi^{-1}}$ are *-automorphisms of $\vnM$ 
which establishes (A1). Since $\varphi$ is semi-graded, it 
follows that for all $n,m\in \nn$ with $m< n$ we have 
$\varphi^{(n,m)} = 0 = (\varphi^{-1})^{(n,m)}$ and so (A2) 
is satisfied.  Finally, notice that if $\vnM$ is 
commutative, then by 
Lemma~\ref{lemma:automatic-contractivity} we have
that $\rho_\varphi$ is contractive, hence $\varphi$
is admissible.
\end{proof}

\begin{proposition}[Criteria for semi-gradedness] \label{proposition:semigradedness}
Let $X,Y$ be subproduct systems and suppose that $\varphi: 
\mathcal{T}_+(X) \rightarrow \mathcal{T}_+(Y)$ is an admissible bounded isomorphism. Then the following are 
equivalent.
\begin{enumerate}
\item $\varphi$ is semi-graded
\item $\mindeg(\varphi(T)) \geq\mindeg(T)$
for all $T\in\tensor(X)$
\item  $\mindeg(\varphi(S^{(1)}_{\xi})) 
\geq 1$ for every $\xi \in X_1$
\item $\varphi(\ker \Phi_0) = \ker \Phi_0$
\end{enumerate}
\end{proposition}
\begin{proof}
We first prove that $(2)\iff(3)$. It is clear that
$(2)\implies(3)$. For the converse, let us suppose that (3) holds. By continuity of $\varphi$, 
in order to prove (2) it suffices to show that 
$\mindeg(\varphi(S^{(n)}_{\xi})) \geq n $ for every $n\geq 1$ and $\xi \in 
X_n$. Consider the set $\mathcal{S} \subseteq \tensor(X)_n$
given by
$$
\mathcal{S} = \Span\left\{ 
S^{(1)}_{\xi_1} \cdot \dots \cdot S^{(1)}_{\xi_n}
\;\;\mid\;\; \xi_1 , ..., \xi_n \in X_1  \right\}.
$$
By the natural identification given in 
Proposition~\ref{proposition:banach-bimodule}, we get that
$\tensor(X)_n$ can be considered as a 
W*-correspon\-dence, for which the set $\mathcal{S}$
is dense in the $\sigma$-topology. Next, since 
$\varphi$ is admissible, for all $n,m\in \nn$ such that $m<n$ the map 
$\Phi_m \circ \varphi \upharpoonright_{\mathcal{T}_+(X)_n} :\mathcal{T}_+(X)_n \rightarrow \mathcal{T}_+(Y)_m$ is  $\sigma$-topology to $\sigma$-topology continuous.  
Furthermore, when $m < n$ we have
that  for every $T \in \mathcal{S}$, 
$\Phi_m(\varphi(T)) = 0$, and hence 
$\Phi_m\varphi(\tensor(X)_n)\equiv 0$. We conclude that $\mindeg(\varphi(S^{(n)}_{\xi})) \geq n$.
 
 We prove that $(3) \implies (1)$.
Since $(2)\iff(3)$, it suffices to show that $\mindeg(\varphi^{-1}(S^{(1)}_{\eta})) \geq 1$ for $\eta \in Y_1$. Represent uniquely $\varphi^{-1}(S^{(1)}_{\eta}) = m + T$
for unique $m\in \vnM$ and $\mindeg(T) \geq 1$. Since $\varphi$ does not decrease minimal degree, we obtain that $\mindeg(\varphi(T)) \geq 1$, and thus from $S^{(1)}_{\eta} = \varphi(m) + \varphi(T)$, due to degree considerations we get $\rho_{\varphi}(m) = \Phi_0(\varphi(m)) = 0$. By injectivity of $\rho_{\varphi}$ we get that $m=0$.

Finally, we note that the implications 
$(1)\implies (3)$, $(1)\implies(4)$ and $(4) \implies(3)$ are trivial. 
\end{proof}

The concept of 
semi-gradedness in the sense of condition (4) in the 
previous Proposition, 
appeared in the work
by Muhly and Solel~\cite[section 5]{muhly-solel-morita}
in their study of tensor algebras of correspondences.

\begin{proposition}\label{proposition:bounded-semi-graded-turns-graded}
Let $X$ and $Y$ be subproduct systems over the same W*-algebra, and 
let $\varphi : \mathcal{T}_+(X) \rightarrow \mathcal{T}_+(Y)$ be a 
semi-graded bounded admissible isomorphism. Let 
$\widetilde{\varphi} : \mathcal{T}_+(X) \rightarrow \mathcal{T}_+(Y)$ be the unique bounded map satisfying
$$
\widetilde{\varphi}(S^{(n)}_{\xi}) = \Phi_n(\varphi(S^{(n)}_{\xi})), 
\qquad  \forall\xi \in X_n. 
$$
Then $\widetilde{\varphi}$ is a \emph{graded} completely bounded admissible isomorphism such that $\widetilde{\varphi^{-1}} = \widetilde{\varphi}^{-1}$, and $\|\widetilde{\varphi}\|_{cb} 
\leq \|\varphi\| \cdot \|\varphi^{-1}\|$.
\end{proposition}

\begin{proof}
We first make a few observations. Notice that since $\varphi$ is
a semi-graded isomorphism, if $T \in \tensor(X)$
and $n= \mindeg(T)$, then
$$
\Phi^Y_n \varphi(T) = \Phi^Y_n \varphi \Phi^X_n (T).
$$
It follows that for all $n$,
\begin{equation}\label{v-inverse}
T= \Phi_n(T) =  \Phi_n\varphi^{-1}\varphi(T) = \Phi_n\varphi^{-1}\Phi_n\varphi(T), 
\qquad \forall T \in \mathcal{T}_+(X)_n.
\end{equation}

Recall $\rho_{\varphi} = \Phi_0 \circ \varphi\restriction_\vnM$. Since $\varphi$
is admissible, $\rho_{\varphi}$ is a *-automorphism.
For each $n\in \mathbb{N}$, let us define $V_n: X_n \to Y_n$ 
by the 
map $V_n(\xi) = (\Psi_n)^{-1} \Phi_n \varphi \Psi_n (\xi)$,
where $\Psi_n(\xi) = S_\xi^{(n)}$ (inserted in the appropriate
Toeplitz algebra; the superscripts $X$ and $Y$ are clear from
the context and hence omitted). Let also
$V'_n: Y_n \to X_n$ be the map given by 
$V'_n(\xi) = (\Psi_n)^{-1} \Phi_n \varphi^{-1} \Psi_n (\xi)$.

Notice that $V_n$ is clearly
well-defined and $\|V_n\| \leq \| \varphi \|$. It follows
from Proposition~\ref{proposition:banach-bimodule}
that $V_n$ is a $\rho_{\varphi}$-correspondence morphism for every $n$.
The same reasoning applies to $V_n'$ to show that it is a $\rho_{\varphi^{-1}}$-correspondence morphism, and in particular
$\| V_n'\| \leq \| \varphi^{-1} \|$. Furthermore, by
equation \eqref{v-inverse}, we have that $V_n^{-1} = V_n'$.
It follows that 
$\sup_{n\in\mathbb{N}} \|V_n\| \leq \| \varphi\|$ and 
$\sup_{n\in\mathbb{N}} \|V_n^{-1}\|  \leq \| \varphi\|$.

In order to obtain that $V = (V_n)$ is a $\rho_{\varphi}$-similarity, 
it remains to show that $V_{n+m}U^X_{n,m} = U^Y_{n,m}(V_n \otimes V_m)$ for all $n, m$. So let $\xi \in X_n$, $\eta \in X_m$,
and let $\xi' = V_n \xi$ and $\eta' = V_m \eta$, so
that $S_{\xi'}^{(n)} = \Phi_n\varphi(S_\xi^{(n)})$ and $S_{\eta'}^{(m)} = \Phi_m\varphi(S^{(m)}_\eta)$.
Notice that
$$
\Phi_{n+m}(S_\xi^{(n)} \cdot S_\eta^{(m)}) = S^{(n+m)}_{U_{n,m}(\xi \otimes \eta)},
\qquad
\Phi_{n+m}(S^{(n)}_{\xi'} \cdot S^{(m)}_{\eta'}) =  S^{(n+m)}_{U_{n,m}(\xi' \otimes \eta')}
$$
and because $\varphi$ is semi-graded we also have that
$$
\Phi_{n+m}( \Phi_n\varphi (S^{(n)}_\xi) \cdot \Phi_m\varphi(S^{(m)}_{\eta}))
= \Phi_{n+m}( \varphi (S^{(n)}_\xi) \cdot \varphi(S^{(m)}_{\eta})).
$$
Therefore we obtain the desired identity as follows:
\begin{align*}
S^{(n+m)}_{U_{n,m}(\xi' \otimes \eta')} 
& = \Phi_{n+m}(S^{(n)}_{\xi'} \cdot S^{(m)}_{\eta'})
 = \Phi_{n+m}( \Phi_n\varphi (S^{(n)}_\xi) \cdot \Phi_m\varphi(S^{(m)}_{\eta}))
= \Phi_{n+m}( \varphi (S^{(n)}_\xi) \cdot \varphi(S^{(m)}_{\eta})) \\
& = \Phi_{n+m}\varphi( S^{(n)}_\xi \cdot S^{(m)}_{\eta}) = 
\Phi_{n+m}\varphi\Phi_{n+m}( S^{(n)}_\xi \cdot S^{(m)}_{\eta})
= \Phi_{n+m}\varphi( S^{(n+m)}_{U_{n,m}(\xi \otimes \eta)}).
\end{align*}
By proposition~\ref{proposition:graded-tensor-isomorphism}, we have that $\tilde{\varphi} = \Ad_V$ is a graded completely bounded admissible isomorphism from $\mathcal{T}_+(X)$ to $\mathcal{T}_+(Y)$.
It is clear that it satisfies the required property, and
its uniqueness is also clear. Finally, the norm inequality follows from Proposition~\ref{proposition:graded-tensor-isomorphism}.
\end{proof}

\begin{proposition} \label{proposition:isometric-semi-graded-turns-graded}
Let $X,Y$ be subproduct systems and let $\varphi :\mathcal{T}_+(X) \rightarrow \mathcal{T}_+(Y)$ be a semi-graded isometric isomorphism. Let
$\widetilde{\varphi} : \mathcal{T}_+(X) \rightarrow \mathcal{T}_+(Y)$
be the unique map  satisfying
$$
\widetilde{\varphi}(S^{(n)}_{\xi}) = \Phi_n(\varphi(S^{(n)}_{\xi})), 
\qquad  \forall\xi \in X_n. 
$$
Then $\widetilde{\varphi}$ is a graded completely isometric isomorphism such that $\widetilde{\varphi^{-1}} = \widetilde{\varphi}^{-1}$.
\end{proposition}

\begin{proof}
First notice that by Lemma~\ref{lemma:isometric-automatic-admissibility},
the map $\varphi$ is admissible. Therefore by the previous proposition,
we have that the map $\tilde{\varphi}$ is a completely bounded graded isomorphism
satisfying $\| \tilde{\varphi} \|_{cb} \leq \| \varphi\|
\| \varphi^{-1}\|
$ and 
$\widetilde{\varphi^{-1}} = \widetilde{\varphi}^{-1}$. Since $\varphi$
is isometric, we have that $\| \tilde{\varphi} \|_{cb} \leq 1$ and also
$\| \tilde{\varphi}^{-1} \|_{cb} \leq 1$ similarly. It follows that $\tilde{\varphi}$ is completely isometric.
\end{proof}

We will show in Section~\ref{section:tensor-algebras-stochastic} that given a bounded/isometric 
admissible isomorphism between tensor algebras \emph{arising
from stochastic matrices}, then there exists a semi-graded
bounded/isometric isomorphism between them. The two previous propositions will then allow us to reach the desired goal of the existence of a graded isometric/bounded isomorphism between the algebras.

\subsection*{Reducing projections}
We now define the notion of \emph{reducing} projection, which will  be useful in converting bounded isomorphisms into semi-graded isomorphisms, in the context of subproduct systems arising from stochastic matrices.

\begin{defi}
Let $(X,U)$ be a subproduct system over the W*-algebra $\vnM$. A projection $p \in \vnM$ is said to be \emph{reducing} for $X$ if
$$
U_{n,m}^*(pX_{n+m}p) \subseteq pX_np \otimes pX_m p.
$$
\end{defi}

\begin{proposition}
Let $X$ be a subproduct system over a W*-algebra $\vnM$, and suppose that
 $p \in \vnM$ is
a projection. Then $p$ is reducing for $X$ if and only if $pXp = \{ pX_np \}_{n\in \mathbb{N}}$ with subproduct maps $\{ U_{n,m} \upharpoonright_{pX_np \otimes pX_mp} \}$ is a subproduct system over $p\vnM p$.
Moreover, if $\vnM$ is commutative, then $p$ is reducing if and only if
for every $\xi \in X_n$ and $\eta \in X_m$, 
\begin{equation}\label{eq:commutative-reducing}
U_{n,m}(p\xi (1-p) \otimes (1-p) \eta p) = 0,
\end{equation}
or alternatively
$$
U_{n,m}(p\xi  \otimes \eta p) = U_{n,m}( p\xi p  \otimes p\eta p).
$$
\end{proposition}

\begin{proof}
$\Rightarrow$: Suppose that $p \in \vnM$ is reducing for $X$. Then $pXp$ is readily shown to be a family of W*-correspondences over $p\vnM p$ with respect to the restricted actions and
inner product. Regarding self-duality, note that if $E$ is a right Hilbert 
W*-module over $\vnM$,
and $\vnM$ is faithfully represented in a Hilbert space $H$, 
then $E$ can be regarded as a von Neumann module inside $B(H, E\otimes H)$, in  
the terminology of \cite{skeide-vNmodules}. It is easy to see that if $p \in \vnM$ is
a projection, then $Ep$ is identified with a von Neumann module over $p\vnM p$ in 
$B(pH, Ep \otimes pH) \subseteq B(H, E \otimes H)p$, hence it is self-dual.

The only thing left to check is that the restriction of $U_{n,m}$ to $pX_np \otimes pX_m p$ is 
a \emph{coisometric} W*-correspondence map onto $pX_{n+m}p$,
which is an immediate consequence of the fact that $p$ is reducing.

$\Leftarrow$: 
Notice that since $U_{n,m}$ is a coisometry onto
$pX_{n+m}p$ when restricted to
$pX_np \otimes pX_mp$, we obtain immediately that
that $p$ is reducing.

For the remaining statement of the proposition, suppose that $\vnM$ is
commutative. Note that in that case we always have the
orthogonal direct sum decomposition
$pX_n \otimes X_mp = (pX_np \otimes pX_mp) \oplus (pX_n(1-p) \otimes (1-p)X_mp)$. 
If \eqref{eq:commutative-reducing} is satisfied, then
$(pX_n(1-p) \otimes (1-p)X_mp) \subseteq \ker U_{n,m}$. It follows
by orthogonality and the fact that $U_{n,m}^*$ is a correspondence map
that $U_{n,m}^* (pX_{n+m}p) \subseteq pX_n p \otimes p X_mp$.
Hence $p$ is reducing. The proof of the converse is straightforward.
\end{proof}

\begin{proposition} \label{proposition:tensor-cutdown}
Let $X$ be a subproduct system over the W*-algebra $\vnM$, and $p\in \vnM$ reducing for $X$. Then one has a canonical graded contractive isomorphism given by the restriction map $\Res_p : p\mathcal{T}_+(X)p \rightarrow \mathcal{T}_+(pXp)$ defined by $\Res_p(pTp) = pTp \upharpoonright_{\mathcal{F}_{pXp}}$. 
\end{proposition}
\begin{proof}
$\Res_p$ is easily shown to be a contractive surjective homomorphism. We show it is injective. Indeed, if we represent $T = \sum_{n=0}^{\infty}S^{(n)}_{\xi_n}$ for unique $\xi_n \in X_n$, then $pTp = \sum_{n=0}^{\infty}S^{(n)}_{p \xi_n p}$. If $\Res_p(pTp) = 0$, then for every $n\in \mathbb{N}$ we have $\Phi_n(\Res_p(pTp)) = S^{(n)}_{p \xi_n p} \upharpoonright_{\mathcal{F}_{pXp}} = 0$. But since
$\| S^{(n)}_{p \xi_n p} \upharpoonright_{\mathcal{F}_{pXp}} \| \geq \|S^{(n)}_{p \xi_n p}\upharpoonright_{\mathcal{F}_{pXp}} (\Omega_{pXp})\| = \|p \xi_n p\|$, we must have $S^{(n)}_{p \xi_n p} = 0$ for all $n\in \mathbb{N}$, assuring that $pTp = 0$.
\end{proof}

We now focus our attention on the case of subproduct systems over a commutative W*-algebra. Note that in this case, by Proposition \ref{proposition:admissibility-semi-graded}, 
all bounded semi-graded isomorphisms of tensor algebras of subproduct systems over a commutative W*-algebra $\vnM$ are automatically admissible.

\begin{lemma} \label{lemma:perpendicular-subproduct}
Let $X$ and $Y$ be subproduct systems over a \emph{commutative} W*-algebra $\vnM$, and let $\varphi: \mathcal{T}_+(X) \rightarrow \mathcal{T}_+(Y)$ be an (algebraic) isomorphism. Then for every pair of projections $p,q \in \vnM$ with $pq =0$ and any $T \in \tensor(X)$
such that $\mindeg(T)\geq 1$ we have that
$$
\mindeg(\varphi(pTq)) \geq 1.
$$
\end{lemma}

\begin{proof}
Let $T \in \tensor(X)$ be given such that $\mindeg(T)\geq 1$
and let
$p' = \Phi_0(\varphi(p))$ and $q' = \Phi_0(\varphi(q))$.
Note that since $pq=0$ and both $\varphi$ and $\Phi_0$ are homomorphisms we have that
$p'q' = 0$. Thus,
$$
\Phi_0(\varphi(pTq)) = 
p'\Phi_0(\varphi(T)) q' =
p'q' \Phi_0(\varphi(T)) = 0
$$
where we used commutativity of $\vnM$ to interchange the order
of the elements in the product.
\end{proof}

\begin{defi}
Let $X$ and $Y$ be subproduct systems over a commutative W*-algebra $\vnM$, and let $\varphi: \mathcal{T}_+(X) \rightarrow \mathcal{T}_+(Y)$ be an (algebraic) isomorphism. We say that a projection $p\in \vnM$ is \emph{$\varphi$-singular} if for some $\xi \in X_1$ we have
$$
\mindeg(\varphi(p S^{(1)}_{\xi} p)) = 0
$$
Otherwise, we say that $p$ is \emph{$\varphi$-regular}.
\end{defi}

The following proposition will be useful in reducing the problem of finding a semi-graded bounded isomorphism even further, in the admissible case.

\begin{proposition} \label{proposition:condition-for-semi-graded}
Let $X$ and $Y$ be subproduct systems over a commutative W*-algebra $\vnM$, and let $\varphi : \mathcal{T}_+(X) \rightarrow \mathcal{T}_+(Y)$ be an admissible bounded isomorphism. If there exists a family of pairwise-perpendicular $\varphi$-regular projections $\{ p_i \}_{i\in I}$ such that $\sum_{i\in I} p_i = 1_{\vnM}$, then $\varphi$ is semi-graded.
\end{proposition}

\begin{proof}
By Proposition~\ref{proposition:semigradedness} it suffices to show that for every $\xi \in X_1$ we have that $\mindeg \big(\varphi(S^{(1)}_{\xi}) \big) \geq 1$. Since $\sum_{i\in I}p_i = 1$ by assumption, due to Lemma \ref{lemma:perpendicular-subproduct} and normality of $\rho_{\varphi} : \vnM \rightarrow \vnM$ (it being a *-automorphism) we have 
the following chain of equalities
taking place in $\vnM$, where the sums are convergent in the weak* topology:
\begin{align*}
\Phi_0(\varphi(S^{(1)}_{\xi})) & =  \rho_{\varphi}(\sum_{i \in I}p_i) \cdot \Phi_0(\varphi(S^{(1)}_{\xi}))  \cdot \rho_{\varphi} (\sum_{j \in I}p_j) 
= \sum_{i\in I} \sum_{j\in I} \rho_\varphi(p_i) \Phi_0(\varphi(S^{(1)}_{\xi}))
\rho_\varphi(p_j)
\\
& =
\sum_{i\in I} \sum_{j\in I} \Phi_0( \varphi(p_i S^{(1)}_{\xi}p_j) ) = 
\sum_{i\in I} \Phi_0( \varphi(p_i S^{(1)}_{\xi}p_i) ) = 0.
\end{align*}
\end{proof}

We will need the following facts about reducing projections in the next section, for dealing with subproduct systems arising from stochastic matrices.

\begin{proposition} \label{proposition:preservation-reducing}
Let $X$ and $Y$ be subproduct systems over a \emph{commutative} W*-algebra $\vnM$, and let $\varphi : \mathcal{T}_+(X) \rightarrow \mathcal{T}_+(Y)$ be a bounded isomorphism, and $p\in \vnM$ a nonzero projection. Let $q = \Phi_0(\varphi(p))$.
\begin{enumerate}
\item $q$ is a nonzero self-adjoint projection in $\vnM$.
\item If $p$ is reducing for $X$ then $q$ is reducing for $Y$,
and $q\varphi(p)q = q$.
\item Suppose that $p$ is reducing and let
$\varphi_{p,q} : p\mathcal{T}_+(X)p \rightarrow q\mathcal{T}_+(Y) q$ be given by
$\varphi_{p,q}(x) = q \varphi(x) q$. 
Then $\varphi_{p,q}$ is a bounded isomorphism. 
\end{enumerate}
\end{proposition}

\begin{proof}
(1) Since $p$ is an idempotent we have that $\varphi(p)^2 = \varphi(p)$, and thus $q^2 = 
\Phi_0(\varphi(p))^2
= \Phi_0(\varphi(p)^2)
= \Phi_0(\varphi(p^2)) = \Phi_0(\varphi(p)) = q$, and since $\vnM$ is commutative, we must have that $q$ is a self-adjoint projection. 
Furthermore, it must be nonzero. Indeed, $\varphi(p) \neq 0$ is an
idempotent, and we must have $\mindeg(\varphi(p)) =
\mindeg(\varphi(p)^2) = 2\mindeg(\varphi(p))$.  Hence $\mindeg(\varphi(p))=0$.

(2) We show now that if $p$ is reducing for $X$, then $q$ must be reducing for $Y$. Let $\xi \in Y_n$ and $\eta \in Y_m$, and 
let us consider the following Fourier series representations:
$$
\varphi^{-1}(\xi) = \sum_{k=0}^{\infty}S^{(k)}_{\xi_k} 
\qquad \text{and} \qquad
\varphi^{-1}(\eta) = \sum_{k=0}^{\infty}S^{(k)}_{\eta_k}.
$$
Applying Proposition~\ref{proposition:banach-bimodule}
and the reducibility of $p$  we obtain
\begin{align*}
p\varphi^{-1}(S^{(n)}_{\xi}) \varphi^{-1}(S^{(m)}_{\eta}) p 
&= p\sum_{k=0}^{\infty}S^{(k)}_{\sum_{\ell =0}^kU^X_{\ell,k - \ell}(\xi_{\ell}\otimes \eta_{k-\ell})} p \\
& = 
\sum_{k=0}^{\infty}S^{(k)}_{\sum_{\ell =0}^kU^X_{\ell,k - \ell}(p \xi_{\ell} p \otimes p \eta_{k-\ell} p)} = p\varphi^{-1}(S^{(n)}_{\xi}) p \varphi^{-1}(S^{(m)}_{\eta}) p.
\end{align*}
Now let $T = \varphi(p) - q$, and note that it is an operator of minimal degree strictly greater than 0. By applying $\varphi$ on both sides of the previous equation we obtain that
$$
(q + T) S^{(n)}_{\xi} S^{(m)}_{\eta} (q + T) = 
(q + T) S^{(n)}_{\xi} (q + T) S^{(m)}_{\eta} (q + T).
$$
By applying $\Phi_{n+m}$ on both sides of the equation, we have that
$$ 
S^{(n+m)}_{U_{n,m}(q\xi \otimes \eta q)} = q S^{(n)}_{\xi} S^{(m)}_{\eta} q = q S^{(n)}_{\xi} q S^{(m)}_{\eta} q = S^{(n+m)}_{U_{n,m}(q\xi q \otimes q \eta q)}
$$
This establishes that $q$ is reducing for $Y$.

Let us show that $q\varphi(p)q = q$. Write uniquely the Fourier
series representation $\varphi(p) = q + \sum_{m=1}^{\infty}S^{(m)}_{\zeta_m}$. 
Since $\varphi(p)^2 = \varphi(p)$ we have that
\begin{align*}
\zeta_1 & = q\zeta_1 + \zeta_1q \\
\zeta_m & = q\zeta_m + \big(\sum_{k=1}^{m-1}U_{k, m-k}(\zeta_k \otimes \zeta_{m-k}) \big) + \zeta_m q, \qquad m\geq 2
\end{align*}
By multiplying the first equation by $q$ on both sides,
we obtain $q\zeta_1q = 2 q\zeta_1q$ and so $q\zeta_1q=0$.
Similarly, multiplying the second equation by $q$ from both sides, and using induction and the reducibility of $q$, we obtain
$$
q\zeta_m q = q\zeta_m q + \big(\sum_{k=1}^{m-1}U_{k, m-k}(q \zeta_k q \otimes q \zeta_{m-k} q) \big) + q \zeta_m q = 2  q \zeta_m q
$$
and thus, $q\zeta_m q = 0$ for all $m$. 
Therefore, $q\varphi(p)q = q$.

(3) Let us suppose that $p$ is reducing, and
$\varphi$ is a bounded isomorphism. It is
clear that $\varphi_{p,q}$ is a bounded linear map.
We will show that it is multiplicative and bijective.

Let $\xi \in X_n$ and $\eta \in X_m$. Write uniquely
the Fourier series representations
$$
\varphi(S^{(n)}_{\xi}) = \sum^{\infty}_{k=0}S^{(k)}_{\xi_k}
\qquad \text{and} \qquad
\varphi(S^{(m)}_{\eta}) = \sum^{\infty}_{k=0}S^{(k)}_{\eta_k}
$$
By an argument similar to the one earlier, using 
 Proposition~\ref{proposition:banach-bimodule}
and the reducibility of $q$  we obtain
\begin{align*}
q\varphi(S^{(n)}_{\xi}) \varphi(S^{(m)}_{\eta}) q &= 
q \sum^{\infty}_{k=0}S^{(k)}_{\sum^{k}_{\ell = 0}U_{\ell,k-\ell}(\xi_{\ell} \otimes \eta_{k-\ell})} q \\
&= 
q \sum^{\infty}_{k=0}S^{(k)}_{\sum^{k}_{\ell = 0}U_{\ell,k-\ell}( q \xi_{\ell} q \otimes q \eta_{k-\ell}q)} q = q\varphi(S^{(n)}_{\xi}) q \varphi(S^{(m)}_{\eta}) q
\end{align*}
Thus by boundedness of $\varphi$, we obtain that $\varphi_{p,q}$
is multiplicative.

Let us show that $\varphi_{p,q}$ is injective. Suppose that 
$0 \neq pTp \in p\mathcal{T}_+(X) p $ satisfies $q\varphi(pTp) q = 0$. 
Notice that for all $n\geq 0$,
$$
q\Phi_n(\varphi(pTp))q = \Phi_n(q\varphi(pTp) q) = 0.
$$
On the other hand, if $d=\mindeg(\varphi(pTp))$ then
$$
\Phi_d(\varphi(pTp))=
\Phi_d(\varphi(p)\varphi(pTp)\varphi(p)) = 
q \Phi_d(\varphi(pTp))q = 0
$$
We conclude that $\Phi_n(\varphi(pTp))=0$ for all $n$, and
therefore $\varphi(pTp)=0$. By injectivity of $\varphi$,
it follows that $pTp=0$, which is a contradiction.

We now show surjectivity. Let $qTq \in q \mathcal{T}_+(Y) q$ 
be given. By surjectivity of $\varphi$ there exists $S\in \mathcal{T}_+(X)$ such that $\varphi(S) = qTq$. Using the fact that $q\varphi(p)q = q$ we obtain
$$
\varphi_{p,q}(pSp) = q\varphi(pSp)q = q\varphi(p)\varphi(S)\varphi(p) q = q \varphi(p) q T q \varphi(p) q = qTq.
$$
\end{proof}

% *****************************************************
% * Tensor algebras associated to stochastic matrices *
% *****************************************************

\section{Tensor algebras of subproduct systems arising from stochastic matrices}\label{section:tensor-algebras-stochastic}

We now head towards handling the specific case of our interest:
the characterization of isomorphisms between tensor algebras arising from subproduct systems of stochastic matrices over a countable state space $\Omega$.  In this case $\vnM = \vnMaS$, and we have a natural family of pairwise orthogonal projections summing up to 1, namely $\{ p_i \}_{i\in \St}$. Thus one can try to apply Proposition~\ref{proposition:condition-for-semi-graded} in the admissible case and Proposition~\ref{proposition:preservation-reducing}.

\subsection*{General Properties} 
Firstly, we characterize several general properties and concepts
in the particular circumstances arising from stochastic
matrices: $\rho$-similarities, admissibility and
automatic continuity of algebraic isomorphisms.

When $P$ is a stochastic matrix, we will write $\mathcal{T}_+(P) : = \mathcal{T}_+(Arv(P))$.

\begin{theorem}\label{theorem:finite-essential-subproduct}
Let $P$ and $Q$ be stochastic matrices. 
\begin{enumerate}
\item
If there exists a $\rho$-similarity of $Arv(P)$ and $Arv(Q)$ then $P\sim_{\sigma_{\rho}}Q$.
\item
If $P$ and $Q$ are finite and essential such that $P \sim_{\sigma}Q$, then there exists a $\rho_{\sigma}$-similarity of $Arv(P)$ and $Arv(Q)$. 
\end{enumerate}
\end{theorem}

\begin{proof}
When $\sigma$ is a permutation of $\St$, we denote for brevity, $i' = \sigma(i)$.

$(1)$: Assume $V : Arv(P) \rightarrow Arv(Q)$ is a given $\rho$-similarity of the subproduct systems. Then $\rho = V_0 : \vnMaS \rightarrow \vnMaS$ is induced by a permutation $\sigma = \sigma_{\rho}$. Now for all $(i,j)\in E(P)$ denote
$E_{ij}$ to be the element in $Arv(P)_1$ which is $1$ at $(i,j)$ and zero otherwise. Due to $V_1$ being a $\rho$-correspondence morphism, we have,
$$
V_1(E_{ij})  = V_1(p_i E_{ij} p_j) = \rho(p_i) V_1(E_{ij}) \rho(p_j) = p_{i'}V_1(E_{ij}) p_{j'}
$$
So we must have that $0 \neq V_1(E_{ij}) = b_{ij} \cdot E_{i'j'}$ for some $0\neq b_{ij} \in \mathbb{C}$. Hence $(i',j') \in E(Q)$, and $P \sim_{\sigma_{\rho}} Q$.

$(2)$: Assume $P\sim_{\sigma}Q$. Define $V_0 : \vnMaS \rightarrow \vnMaS$ by $V_0 = \rho_{\sigma} = \rho$, that is, $V_0(p_i) = p_{\sigma(i)}$. Define $V_n : Arv(P)_n \rightarrow Arv(Q)_n$ by $V_n(A) = (\sqrt{Q^n})^{\flat} * [R_{\sigma}(\sqrt{P^n} * A)R^{-1}_{\sigma}]$.
We need to show that $V_n$ is a bijective, subproduct preserving, $\rho$-morphism with $\sup_{n\in \nn}\{ \| V_n \| , \|V_n^{-1} \| \} < \infty$.

It is immediate that $V_n$ is bijective with $\rho^{-1}$-morphism inverse $V_n^{-1}(A) = (\sqrt{P^n})^{\flat} * [R_{\sigma^{-1}}(\sqrt{Q^n} * A)R^{-1}_{\sigma^{-1}}]$. We show that $\sup_{n\in \nn} \{ \| V_n \| \} < \infty$ and the proof for $V_n^{-1}$ is analogous.
Indeed, for $A = [a_{ij}] \in Arv(P)_n$ we have
$$
\| V_n(A) \|^2 = \sup_j \sum_i \frac{P^{(n)}_{ij}}{Q^{(n)}_{i'j'}} \cdot |a_{ij}|^2 \leq \Big( \sup_{(i,j)\in E(P^n)}\frac{P^{(n)}_{ij}}{Q^{(n)}_{i'j'}}\Big) \cdot ||A||^2
$$

Denote 
\begin{displaymath}   
c_n(i,j) = \left\{     
\begin{array}{lr}       
 \frac{P^{(n)}_{ij}}{Q^{(n)}_{i'j'}} & : (i,j)\in E(P^n) \\       
0 & : (i,j)\notin E(P^n)     
\end{array}   \right.
\end{displaymath}

Due to finiteness of $P$, in order to show $\sup_{n\in \nn} \{ \| V_n \| \} < \infty$, it would suffice to show that for all $(i,j)\in \St^2$ one has $\limsup_{n\rightarrow \infty} c_n(i,j) < \infty$. Since $P$ and $Q$ are essential, Theorem \ref{theorem:graph-decomposition} allows us to assume WLOG that $P$ and $Q$ are irreducible. Indeed, if $i$ and $j$ do not belong to the same irreducible component in a decomposition of $P$ into irreducibles, we would just have $c_n(i,j) = 0$ for all $n\in \nn$. Thus, assuming $P$ (as well as $Q$) is irreducible and $r$-preiodic, with $i\in \St_{\ell_1}$ and $j\in \St_{\ell_2}$ where $\St_0,..,\St_{r-1}$ is a cyclic decomposition for $P$ given by Theorem \ref{theorem:cyclic-graph-decomposition}, noting that the only instance for which $c_n(i,j) > 0$ is when $n$ is of the form $n=n'r+\ell$, where $\ell = \ell_2 - \ell_1 (\mod r)$, we may apply Theorem \ref{theorem:convergence-theorem-positive-recurrent} to obtain that $\limsup_{n\rightarrow \infty} c_n(i,j) < \infty$.
Thus we establish that $V_n$ is a bijective $\rho$-morphism with $\sup_{n\in \nn} \{ \| V_n \| \} < \infty$. 

We now show that $V_n$ preserves subproducts. Indeed, for $A\in Arv(P)_n$ and $B\in Arv(P)_m$ we have
$$
U^Q_{n,m}(V_n(A) \otimes V_m(B)) = (\sqrt{Q^{n+m}})^{\flat}* [(R_{\sigma} (\sqrt{P^n}*A)R_{\sigma}^{-1}) \cdot (R_{\sigma} (\sqrt{P^m}*B)R_{\sigma}^{-1})] =
$$
$$
(\sqrt{Q^{n+m}})^{\flat}* [R_{\sigma} (\sqrt{P^{n+m}} * U^P_{n,m}(A\otimes B))R_{\sigma}^{-1} ] = V_{n+m} U^P_{n,m}(A\otimes B)
$$
Thus the family $V_n: Arv(P) \rightarrow Arv(Q)$ is a $\rho_{\sigma}$-similarity, as required.
\end{proof}

Thus, for finite essential stochastic matrices, only their \emph{graph} structure determines whether or not there exists a $\rho$-similarity of the subproduct systems.

\begin{theorem}[Automatic continuity]
Let $P$ and $Q$ be a stochastic matrices. If $\varphi : \tensor(Q) \rightarrow \tensor(P)$ is an algebraic isomorphism, then 
$\varphi$ is bounded.
\end{theorem}

\begin{proof}
By Corollary \ref{cor:power-sequence}, it suffices to show that for every $0 \neq T\in \tensor(P)_d$ there exists $Z\in \ker\Phi_0$ such that $TZ^n \neq 0$ for all $n\geq 1$. For the purpose of the proof, we say that $i\in \St$ is \emph{returning} for $P$ if $P^{(\ell)}_{ii} > 0$ for some (and hence infinitely many) $1 \leq \ell \in \nn$, and we say for $i,j\in \St$ that $i$ \emph{leads} to $j$ if $P^{(\ell)}_{ij} > 0$ for some $1 \leq \ell \in \nn$.

Let $T = S^{(d)}_{A} \neq 0$ and denote $A=[a_{ij}]$. 
There exist $i,j\in \St$ such that $p_i T p_j \neq 0$ , or equivalently $p_iAp_j\neq 0$, since $p_i T p_j =a_{ij}S^{(d)}_{E_{ij}}$.
Note that for $d=0$ we must have $i=j$. We split the proof into three cases.
\begin{enumerate}
\item There is a pair $i,j \in \Omega$ such that  $p_i A p_j \neq 0$ and $j\in \St$ is returning for $P$. Then let $Z=S^{(\ell)}_{E_{jj}}$, where $1 \leq \ell \in \nn$ is such that $P^{(\ell)}_{jj} >0$. Note that for each $n\geq 1$
there exists $c_n >0$ so that 
 $(S^{(\ell)}_{E_{jj}})^n = c_n S^{(n \ell)}_{E_{jj}}$. Therefore
$
TZ^n = S^{(d)}_{A} (S^{(\ell)}_{E_{jj}})^n = c_n S^{(d)}_{A}S^{(n \ell)}_{E_{jj}} = c_n' S^{(d+n \ell)}_B
$
where $c_n'$ is a nonzero scalar and
$B=(\sqrt{P^{d+n \ell}})^{\flat} * \sqrt{P^d} * (Ap_j)$, which is non-zero since $Ap_j \neq 0$, and the matrices $P^d$ and $P^{d+n \ell}$ have non-zero entries where $Ap_j$ does.

\item 
There is no pair $i,j \in \Omega$ such that  $p_i A p_j \neq 0$ and $j\in \St$ is returning for $P$, however there is a pair
$i,j \in \Omega$ such that  $p_i A p_j \neq 0$ and $j$ leads
to a returning $k\in \St$. In that case there exist 
$m , \ell\geq 1$ such that $S^{(m)}_{jk} \neq 0$ and $S^{(\ell)}_{kk} \neq 0$. Let us take $Z = S^{(m)}_{jk} + 
S^{(\ell)}_{kk}$. Note that since 
$S^{(\ell)}_{E_{kk}}S^{(m)}_{E_{jk}} = 
S^{(m)}_{E_{jk}}S^{(m)}_{E_{jk}} = 0$, 
we have $Z^n = S^{(m)}_{E_{jk}}(S^{(\ell)}_{E_{kk}})^{n-1} + (S^{(\ell)}_{E_{kk}})^n$.
Thus, since $Ap_k = 0$ we must have that
$$
TZ^n = S^{(d)}_A (S^{(m)}_{E_{jk}}(S^{(\ell)}_{E_{kk}})^{n-1} + (S^{(\ell)}_{E_{kk}})^n) = S^{(d)}_A S^{(m)}_{E_{jk}}(S^{(\ell)}_{E_{kk}})^{n-1}
$$
Let $T' = S^{(d)}_A S^{(m)}_{E_{jk}}$. Notice that 
$p_iT'p_k \neq 0$, and in this case $k$ is returning for 
$T'$. Therefore,  by the argument of case (1) we see that $T'(S^{(\ell)}_{E_{kk}})^{n-1} \neq 0$,
and hence $TZ^n\neq 0$, for all $n\geq 1$.

\item
There is no pair $i,j \in \Omega$ such that  $p_i A p_j \neq 0$ and $j\in \St$ is returning for $P$ or $j$ leads
to a returning $k\in \St$. Let $i,j \in \Omega$ be 
such that  $p_i A p_j \neq 0$.  Then there exists a sequence $\{j_t\}_{t=1}^{\infty}$ such that $j=j_1$ and $(j_t,j_{t+1})\in 
E(P)$ for all $t\geq 1$, and $j_{t_1} \neq j_{t_2}$ for all $t_1 
\neq t_2$ since $j$ does not lead to a returning state. Define 
$B = \sum^{\infty}_{t=1}E_{j_t,j_{t+1}}$ which gives an element in $Arv(P)_1$ by Theorem \ref{theorem:subproduct-computation} since every column of $B$ contains only one non-zero entry, at which it is $1$. Set $Z = S^{(1)}_{B}$. Multiplying $T$ by $Z$ on the right shifts the columns of $A$ indexed by $j_1,j_2,...$, up to multiplying each entry by a positive scalar multiple, to the columns indexed by $j_2,j_3,...$ respectively (and deletes all other columns). Therefore, if $C$ is such that  
$TZ^n = S^{(d)}_A(S^{(1)}_B)^n = S^{(d+n)}_C$, we see that
$p_iCp_{j_{n+1}}$ is a non-zero scalar multiple of $p_iAp_j$. 
Hence $TZ_n \neq 0$.
\end{enumerate}
\end{proof}

Henceforth, there is no ambiguity when referring to isomorphisms $\varphi :\tensor(P) \rightarrow \tensor(Q)$: they are always
bounded, and there is no need to mention boundedness explicitly.

\begin{proposition}[Automatic admissibility for finite matrices] \label{proposition:bounded-automatic-admisibility} 
Let $P$ and $Q$ be stochastic matrices over \emph{finite} $\St$. Then every (algebraic) isomorphism $\varphi : \mathcal{T}_+(P) \rightarrow \mathcal{T}_+(Q)$ is admissible.
\end{proposition}

\begin{proof}
We show that $\rho_{\varphi} = \Phi_0 \circ \varphi \upharpoonright_{\ell^{\infty}(\St)}$ is injective, since then, finite dimensionality would imply bijectiveness, and Lemma \ref{lemma:automatic-contractivity} will show that $\rho_{\varphi}$ is a *-automorphism, thus satisfying (A1) in Definition \ref{deifnition:admissibility}. Indeed, to show injectivity, note that the finite family $\{ p_i \}_{i \in \St}$ of pairwise perpendicular non-zero projections, which constitutes a basis for $\ell^{\infty}(\St)$, gets sent to a family of pairwise perpendicular non-zero projections $\{ \rho_{\varphi}(p_i) \}_{i \in \St}$, due to Proposition \ref{proposition:preservation-reducing} item (1) and $\rho_{\varphi}$ being a homomorphism. This implies, in particular, that this new family is a linearly independent set. For (A2) in Definition \ref{deifnition:admissibility}, note that $\varphi^{(n,m)}$ as maps between finite dimensional W*-correspondences are automatically $\sigma$-topology continuous.
\end{proof}

\subsection*{Reducing and Singular projections}
We proceed to characterize reducing projections in $Arv(P)$. Recall the terminology of paths from 
Definition~\ref{definition:paths}. We will sometimes mention paths and cycles without mentioning the stochastic matrix $P$ when it is determined unambiguously by the context.

\begin{proposition}\label{proposition:reducing-char}
Let $P$ be a stochastic matrix and let $p$ be a projection in $\vnMaS$. Let $C_p \subset \St$ be such that $p = \sum_{i\in C_p}p_i$. Then $p$ is reducing for $Arv(P)$ if and only if for any path $\gamma:\{0,...,\ell\} \rightarrow \St$ in the directed graph of $P$ with $\gamma(0),\gamma(\ell) \in C_p$, we have that $\gamma(k) \in C_p$ for all $0\leq k \leq \ell$.

\end{proposition}

\begin{proof}
($\Rightarrow):$ We prove the contrapositive. Suppose that there exists a path $\gamma$ of length $n+m$ such that $i:=\gamma(0) , k:=\gamma(n+m) \in C_p$ while $j:=\gamma(n) \notin C_p$. Decompose $\gamma$ into two paths, $\gamma_n$ from $i$ to $j$ and $\gamma_m$ from $j$ to $k$, which assures us that $P^{(n)}_{ij}, P^{(m)}_{jk} > 0$. Thus, 
$$ U_{n,m}(p E^{(n)}_{ij} (1-p) \otimes (1-p) E^{(m)}_{jk} p) =
U_{n,m}(E^{(n)}_{ij} \otimes E^{(m)}_{jk}) = \sqrt{\frac{P^{(n)}_{ij}P^{(m)}_{jk}}{P^{(n+m)}_{ik}}}E^{(n+m)}_{ik} \neq 0
$$
Which shows that $p$ is not reducing.

($\Leftarrow):$ Again we prove the contrapositive. If $p$ is not reducing, then
$$
0 \neq U_{n,m}(p A (1-p) \otimes (1-p) B p)
$$
for some $A \in Arv(P)_n$ and $B \in Arv(P)_m$. Thus there exists some $i,k \in C_p$ and $j\notin C_p $ such that
$$
0 \neq U_{n,m}(p_i A p_j \otimes p_j B p_k)
$$
In particular, $p_i A p_j$ and $p_j B p_k$ are nonzero
in the subproduct system, so that by considering their
supports we obtain $P^{(n)}_{ij}, P^{(m)}_{jk} > 0$. Hence there exists a path $\gamma$ with $\gamma(0), \gamma(n+m) \in C_p$ while $\gamma(n) \notin C_p$.
\end{proof}

\begin{defi} \label{definition:paths2}
Let $P$ be a stochastic matrix over a state set $\St$. A path $\gamma$ of length $\ell$ is said to be \emph{\streamlined} if for every $0 \leq k \leq \ell-1$ one has $\gamma(k) \neq \gamma(k+1)$. 
\end{defi}

Given two different states $i,j \in \Omega$, we note that  if there exists a path from $i$ to  $j$, then a simple culling procedure yields a \streamlined path 
from $i$ to $j$. The situation for cycles is slightly
different: in general, a cycle $\gamma$ of length $\ell$ 
from $i$ to $i$ only yields a \streamlined cycle by culling when there exists $1 \leq n \leq \ell-1$
such that $\gamma(n) \neq i$.

\begin{corollary}\label{cor:streamlined}
Let $P$ be a stochastic matrix over a state set $\St$. A projection $p_i$ for $i\in \St$ is reducing for $Arv(P)$ if and only if there are no \streamlined cycles through $i$ in the directed graph of $P$.
\end{corollary}

\begin{remark} \label{remark:sigular-is-cyclic}
If $P$ and $Q$ are stochastic over $\St$ and $i\in \St$ is such that $p_i$ is $\varphi$-singular for some $\varphi: \tensor(P) \rightarrow \tensor(Q)$ then it must be the case that $P_{ii} > 0$. Indeed, for some $\xi \in Arv(P)_1$, we have that $\Phi_0(\varphi(p_i S^{(1)}_\xi p_i)) \neq 0$ so that $p_i S^{(1)}_\xi p_i \neq 0$.
Therefore, since $p_i S^{(1)}_\xi p_i$ is supported in $(i,i)$ we have that $S^{(1)}_{E_{ii}} \neq 0$, and
therefore $P_{ii}>0$.

\end{remark}

\begin{remark} \label{remark:reducing-by-1-dim}
If $i\in \St$ is such that $p_i$ is reducing for $Arv(P)$, then $p_i Arv(P) p_i$ is a subproduct system over $\mathbb{C}$, with all fibers being dimension 0 (except for the $0$-th fiber) or all of dimension 1. Since  $ U_{n,m} \upharpoonright_{p_i Arv(P)_n p_i \otimes p_i Arv(P)_m p_i}$ remains coisometric for all $n,m\in \mathbb{N}$, it must in fact be isometric due to dimension considerations. Thus, $p_i Arv(P) p_i$ is isomorphic to a \emph{product} system as in example \ref{regular-product-system} with $E = 0$ or $E = \mathbb{C}$. The tensor algebras are then given up to a canonical isomorphism by $\mathbb{C}$ or $\mathbb{A}(\mathbb{D})$ respectively, where the latter is the disc algebra.
\end{remark}

\begin{proposition} \label{proposition:singular-is-reducing}
Let $i\in \St$. If for some stochastic $Q$ and isomorphism $\varphi : \mathcal{T}_+(P) \rightarrow \mathcal{T}_+(Q)$ we have that $p_i$ is $\varphi$-singular then it must be reducing for $Arv(P)$.
\end{proposition}

\begin{proof}
We prove the contrapositive. Suppose that $p_i$ is not
reducing. By Corollary~\ref{cor:streamlined}
there exists a  \streamlined cycle  $\gamma$ of length $\ell \geq 2$ such that $\gamma(0) = \gamma(\ell) = i$. Hence, 
for some $ c_{\gamma} > 0$ we have
\begin{equation}\label{eq:one-way}
S^{(1)}_{E_{i\gamma(1)}} \cdot ... \cdot S^{(1)}_{E_{\gamma(\ell - 1)i}} = c_{\gamma} S^{(\ell)}_{E_{ii}}.
\end{equation}
By Remark~\ref{remark:reducing-by-1-dim}, we must have $S^{(1)}_{E_{ii}} \neq 0$. We note that for
some $d_{i,\ell} > 0$ we have
\begin{equation}\label{eq:another-way}
(S^{(1)}_{E_{ii}})^{\ell} = d_{i,\ell} \cdot S^{(\ell)}_{E_{ii}}.
\end{equation}
Since $\gamma$ is \streamlined, by 
Lemma~\ref{lemma:perpendicular-subproduct} it follows that we must have  $\mindeg(\varphi(S^{(1)}_{E_{\gamma(k)\gamma(k+1)}})) \geq 1$ for all $0 \leq k \leq \ell-1$. So, by equation~\eqref{eq:one-way},
we must have $\mindeg(\varphi(S^{(\ell)}_{E_{ii}})) \geq \ell$. Therefore, since $\vnMaS$ (being commutative) does not contain any nilpotent elements,
it follows from equation~\eqref{eq:another-way} that we must have that $\mindeg(\varphi(S^{(1)}_{E_{ii}})) \geq 1$, contradicting $\varphi$-singularity of $p_i$.
\end{proof}

Thus we see that, in the context of subproduct systems arising from stochastic matrices, singular minimal projections (with respect to some isomorphism) can only arise from reducing minimal projections. In particular, the following special case of our main goal is resolved.

\begin{corollary}\label{corollary:essential-case}
Let $P$ and $Q$ be two essential stochastic matrices over $\St$ with no $1\times 1$ irreducible block. If $\varphi: \mathcal{T}_+(P) \rightarrow \mathcal{T}_+(Q)$ is an admissible isomorphism, then it is semi-graded. Therefore there exists a completely bounded graded (admissible) isomorphism from $\mathcal{T}_+(P) \rightarrow \mathcal{T}_+(Q)$, hence there exists a $\rho$-similarity
between $Arv(P)$ and $Arv(Q)$ and $P\sim_{\sigma_\rho} Q$.
\end{corollary}

\begin{proof}
Using Proposition \ref{theorem:graph-decomposition},
we can decompose $P$ into irreducible blocks, which will
be at least $2\times 2$. For every $i\in \St$ we must have that $p_i$ is non-reducing for $Arv(P)$ (and $Arv(Q)$) since there exists a \streamlined cycle through $i$ in the directed graph of $P$. This makes $p_i$ a $\varphi$-regular projection for every $i\in \St$ due to Proposition \ref{proposition:singular-is-reducing}. Since $\varphi$ is admissible, by Proposition \ref{proposition:condition-for-semi-graded}, we obtain that $\varphi$ is semi-graded. Using Proposition~\ref{proposition:bounded-semi-graded-turns-graded} yields a the graded completely bounded admissible isomorphism, and by Proposition~\ref{proposition:graded-tensor-isomorphism} we obtain the desired $\rho$-similarity.
It follows from Theorem~\ref{theorem:finite-essential-subproduct} that $P\sim_{\sigma_\rho} Q$.
\end{proof}

We will develop further machinery to deal with the case of algebraic isomorphisms between tensor algebras of general (non-essential) stochastic matrices.

\begin{proposition} \label{proposition:isometric-one-dim}
Let $P$ and $Q$ be stochastic matrices over $\St$. Let $\varphi : \mathcal{T}_+(P) \rightarrow \mathcal{T}_+(Q)$ be an isomorphism and $i\in \St$ be such that $p_i$ is reducing. Then,
\begin{enumerate}
\item
There exists $j\in \St$ such that $p_j = \Phi_0 (\varphi(p_i))$ and $p_j$ is reducing. Furthermore, the map $\widehat{\varphi}_{p_i,p_j} := \Res_{p_j} \circ \varphi_{p_i,p_j} \circ \Res_{p_i}^{-1}:\mathcal{T}_+(p_i Arv(P) p_i)\rightarrow \mathcal{T}_+(p_j Arv(Q) p_j)$ is an isometric isomorphism.
\item
If $R$ is a stochastic matrix over $\St$ and $\psi : \mathcal{T}_+(Q) \rightarrow \mathcal{T}_+(R)$ is an isomorphism, let $k\in \St$ be such that $p_k = \Phi_0(\psi(p_j))$, then $p_k = \Phi_0((\psi \circ \varphi) (p_i))$ and $\psi_{p_j,p_k} \circ \varphi_{p_i,p_j} = (\psi \circ \varphi)_{p_i,p_k}$. In particular, $(\varphi_{p_i,p_j})^{-1} = (\varphi^{-1})_{p_j,p_i}$ and $(\widehat{\varphi}_{p_i,p_j})^{-1} = (\widehat{\varphi^{-1}})_{p_j,p_i}$
\end{enumerate}
\end{proposition}

\begin{proof}
(1): Let $q = \Phi_0(\varphi(p_i))$. We show that $q$ is minimal. If not then write $q = q_1 +q_2$ with $q_1 \perp q_2$ both non-zero in $q\vnMaS q$. Since $(\varphi_{p_i,q})^{-1}(q)$ is an idempotent in $p_i \mathcal{T}_+(P)p_i$, and $\varphi_{p_i,q}$ is an isomorphism by item (3) of Proposition \ref{proposition:preservation-reducing}, we must have, by similar arguments to those in item (1) of Proposition \ref{proposition:preservation-reducing}, that $\Phi_0 (\varphi_{p_i,q}(q_1))$ is a non-zero projection in $p_i \mathcal{T}_+(P) p_i$. Similarly, $\Phi_0 (\varphi_{p_i,q}(q_2))$ is also a non-zero projection $p_i \mathcal{T}_+(P) p_i$, which is perpendicular to $\Phi_0 (\varphi_{p_i,q}(q_1))$. This contradicts minimality of $p_i$. Now, since $\widehat{\varphi}_{p_i,p_j}$ is a map from $\mathcal{T}_+(p_i Arv(P) p_i)$ into $\mathcal{T}_+(p_j Arv(Q) p_j)$, each being isomorphic to either $\mathbb{A}(\mathbb{D})$ or $\mathbb{C}$, both embed into a commutative C* algebra, so by Lemma \ref{lemma:automatic-contractivity}, our map $\widehat{\varphi}_{p_i,p_j}$ is isometric.

(2): Using the previous item, let $j, k , r \in \St$ be such that $p_j = \Phi_0(\varphi(p_i))$, $p_k =  \Phi_0(\psi(p_k))$ and $p_r = \Phi_0(\psi(\varphi(p_i)))$.
By Proposition~\ref{proposition:preservation-reducing},
we have that $p_j = p_j \varphi(p_i) p_j$ and
$p_k = p_k\psi(p_j)p_k$. Therefore,
$$
p_k = p_k \psi(p_j) p_k = p_k \psi(p_j\varphi(p_i)p_j) p_k =
p_k \psi(p_j) \psi(\varphi(p_i)) \psi(p_j) p_k.
$$
Applying $\Phi_0$ to both sides and using the fact
that $\Phi_0$ is a homomorphism we obtain that
$$
p_k = \Phi_0(p_k) = \Phi_0(p_k) \Phi_0(\psi(p_j)) \Phi_0(\psi(\varphi(p_i))) \Phi_0(\psi(p_j)) \Phi_0(p_k) =
p_k p_r p_k.
$$
Therefore $p_k \leq p_r$. By minimality of the projections,
we obtain $p_r = p_k$ as desired.

For the remainder of the statement, given $p_iTp_i \in p_i \mathcal{T}_+(P)p_i$ a computation yields
$$
(\psi_{p_j,p_k} \circ \varphi_{p_i,p_j})(p_i T p_i) = p_k \psi(p_j) (\psi \circ \varphi)(p_i T p_i) \psi(p_j) p_k = 
$$
$$
p_k \psi(p_j) p_k (\psi \circ \varphi)(p_i T p_i) p_k \psi(p_j) p_k = p_k (\psi \circ \varphi)(p_i T p_i) p_k = (\psi \circ \varphi)_{p_i,p_k}(p_i T p_i)
$$
Where the second equality follows along the lines of the proof of multiplicativity of $\psi_{p_j,p_k}$ in item three of Proposition \ref{proposition:preservation-reducing}.
\end{proof}

\begin{proposition} \label{proposition:singularity-among-reducing}
Let $P$ and $Q$ be stochastic, let $i\in \St$ and let $\varphi: \mathcal{T}_+(P) \rightarrow \mathcal{T}_+(Q)$ be an isomorphism. Let $i'\in \St$ be such that that $p_{i'} = \Phi_0(\varphi(p_i))$. Then $p_i$ is $\varphi$-singular if and only if $p_{i'}$ is $\varphi^{-1}$-singular.
\end{proposition}

\begin{proof}
If $p_i$ is $\varphi$-regular then $p_i$ is $\widehat{\varphi}_{p_i,p_{i'}}$-regular. Thus $\widehat{\varphi}_{p_i,p_{i'}}$ is semi-graded by Proposition \ref{proposition:condition-for-semi-graded}. Since $\widehat{\varphi^{-1}}_{p_{i'},p_i} = (\widehat{\varphi}_{p_i,p_{i'}})^{-1}$, we have that $p_{i'}$ is  $\widehat{\varphi^{-1}}_{p_{i'},p_i}$-regular and so $p_{i'}$ is  $\varphi^{-1}$-regular. By Proposition \ref{proposition:isometric-one-dim} item (2) we get that $p_i = \Phi_0(\varphi^{-1}(p_{i'}))$ so that by applying the above we get that if $p_{i'}$ is $\varphi^{-1}$-regular, then $p_i$ is $\varphi$-regular.
\end{proof}

\subsection*{An equivalence between singular projections} For every stochastic matrix $P$ over a state set $\St$ we define the set of singular states of $P$
$$
I_P : = \{ \ j \in \St  \ | \ j\text{ is $\varphi$-singular for some isomorphism }\varphi: \mathcal{T}_+(P) \rightarrow \mathcal{T}_+(Q) \ \}
$$
Now, if $\varphi : \mathcal{T}_+(P) \rightarrow \mathcal{T}_+(Q)$ is \emph{any} algebraic isomorphism, and $j \in I_P$ then there is a projection $p_{j'}$ such that $p_{j'} = \Phi_0(\varphi(p_j))$ and $j' \in I_Q$ and $\widehat{\varphi}_{p_j, p_{j'}} : \mathcal{T}_+(p_jArv(Q)p_j) \rightarrow \mathcal{T}_+(p_{j'} Arv(P)p_{j'})$ is an isometric isomorphism of Banach algebras, by Proposition~\ref{proposition:isometric-one-dim}. Since $P_{jj} > 0$, and $Q_{j'j'} >0$, we must have that $\mathcal{T}_+(p_j Arv(P)p_j)$ and $\mathcal{T}_+(p_{j'}Arv(Q)p_{j'})$ are both isomorphic to the disc algebra $\mathbb{A}(\mathbb{D})$ so that the spectra of these commutative Banach algebras can be naturally identified with $\overline{\mathbb{D}}$ where the vacuum state is identified with the point $0$. Composition by $\widehat{\varphi}_{p_{j'}, p_j}$ induces a homeomorphism $f_{\varphi}^j = (\widehat{\varphi}_{p_j, p_{j'}})^* : \overline{\mathbb{D}} \rightarrow \overline{\mathbb{D}}$, so that when it is restricted to $\mathbb{D}$, it becomes a biholomorphic automorphism of $\mathbb{D}$ (See \cite[Lemma 4.4]{DRS} for this fact in much greater generality), and $j$ is $\varphi$-regular if and only if $(\widehat{\varphi}_{p_j, p_{j'}})^*$ preserves the vacuum state, that is, if and only if $f_j^{\varphi}(0) = 0$. 
Biholomorphisms of the disk have a well-known description
as Moebius transformations. Thus, we have proven the following statement.

\begin{proposition}\label{prop:fj}
Let $P,Q$ be stochastic matrices, let 
$\varphi : \mathcal{T}_+(P) \rightarrow \mathcal{T}_+(Q)$ be an algebraic isomorphism, and let $j \in I_P$. Let 
$p_{j'} = \Phi_0(\varphi(p_j))$.  Then
there exist $w_j \in \mathbb{D}$ and $\theta_j \in [0, 2\pi)$ such that $f_\varphi^j = (\widehat{\varphi}_{p_j, p_{j'}})^* : \overline{\disk} \to \overline{\disk}$ is
given by
$$
f_{\varphi}^j(z) = e^{i\theta_j} \cdot \frac{z - w_j}{1 - \overline{w_j}z}
$$
Furthermore, $p_j$ is $\varphi$-regular if and only if
$f_j^{\varphi}(0) = 0$.
\end{proposition}

The preceeding discussion and Proposition~\ref{proposition:singularity-among-reducing} motivate the following definition

\begin{defi}
Let $P$ and $Q$ be stochastic matrices over $\St$. An isomorphism $\varphi :\tensor(P) \rightarrow \tensor(Q)$  is said to be \emph{regular} if for all $i\in \St$, the projection $p_i$ is $\varphi$-regular.
\end{defi}

Due to Proposition \ref{proposition:singularity-among-reducing}, we have that $\varphi$ is regular if and only if $\varphi^{-1}$ is regular and, by Proposition~\ref{proposition:condition-for-semi-graded}, 
 $\varphi$ is regular and admissible if and only if $\varphi$ is semi-graded.

\begin{defi}
Let $P$ be a stochastic matrix over $\St$. Define $\sim_P$ to be the equivalence relation given by pairs $(i,k)\in I_P^2$ such that there exists a finite chain $(i=j_0 , j_1,..., j_s = k)$ with the property that any two subsequent pairs $(j_{t}, j_{t+1})$ there exist $i',k' \in \St$ with two \streamlined paths from $i'$ to $k'$ in the graph of $P$, one through $j_t$ and the other through $j_{t+1}$. Let $\mathcal{R}_P$ be a distinct set of representatives for $\sim_P$,  henceforth taken to be fixed.
\end{defi}

\begin{remark} \label{remark:three-properties}
Let $P$ and $Q$ be stochastic matrices. We shall show that $\sim_P$ satisfies the following three properties:
\begin{enumerate}
\renewcommand{\labelenumi}{($\Diamond_\theenumi$)}
\item \label{diamond-1}
If  $\varphi : \mathcal{T}_+(P) \rightarrow \mathcal{T}_+(Q)$ is an isomorphism
and $p_j$ is $\varphi$-regular for all $j \in \rp$, then
$p_j$ is $\varphi$-regular for all $j \in \Omega$.
\item \label{diamond-2}
For every $\Lambda = \{ \lambda_j \} \in \mathbb{T}^{\mathcal{R}_P}$ there exists an isometric isomorphism $\alpha_{\Lambda} : \mathcal{T}_+(P) \rightarrow \mathcal{T}_+(P)$ such that 
for every $j\in I_P$ we have $f^{j}_{\alpha_{\Lambda}}(z) = \lambda_{\mu}z$ for the unique $\mu\in \mathcal{R}_P$ such that $j \sim_P \mu$.
\item \label{diamond-3}
Every isomorphism $\varphi : \mathcal{T}_+(P) \rightarrow \mathcal{T}_+(Q)$ induces an equivalence preserving bijection $\Upsilon_{\varphi} : I_P \rightarrow I_Q$, and hence
a bijection $\overline{\Upsilon}_j: \rp \to \calr_Q$. Moreover,
$\Upsilon_{\varphi^{-1}} = (\Upsilon_{\varphi})^{-1}$.
\end{enumerate}
These three properties of $\sim_P$ will suffice to show that every isomorphism can be ``deformed" into a regular isomorphism.
\end{remark}

First we prove a few auxiliary results about the relation $\sim_P$.

\begin{proposition} \label{proposition:forcing-non-singular}
Let $P$ and $Q$ be stochastic matrices, and let $\varphi:\mathcal{T}_+(P) \rightarrow \mathcal{T}_+(Q)$ be an isomorphism. Assume that $j\in \St$ is $\varphi$-singular. Assume further that for $i,k$ there exists a streamlined path $\gamma$ from $i$ to $k$ through $j$ of length $\ell \geq 2$. Then for any path $\gamma'$ from $i$ to $k$ and through $j'$, we have that $p_{j'}$ is reducing for $Arv(P)$. If in addition $P^{(n)}_{j'j'} > 0$ for some $n \geq 1$, then $p_{j'}$ must be $\varphi$-singular.
\end{proposition}

\begin{proof}
If $p_{j'}$ is non-reducing for $Arv(P)$, there exists a \streamlined cycle through $j'$ of length $r \geq 2$ and thus $\mindeg(\varphi(S^{(nr)}_{E_{j'j'}})) \geq n \cdot \mindeg(\varphi(S^{(r)}_{E_{j'j'}})) \geq nr$ for all $n \geq 1$ due to Lemma \ref{lemma:perpendicular-subproduct}. Now, if $\gamma$ is of length $\ell$ with $\gamma(m) = j$, and $\gamma'$ is of length $\ell'$ with $\gamma'(m') = j'$ both assumed to be \emph{\streamlined}, then up to some positive scalar, which we denote by $\approx$, we have
$$
S^{(\ell+s)}_{E_{ik}} \approx S^{(1)}_{E_{i\gamma(1)}} \cdot ... \cdot S^{(1)}_{E_{\gamma(m-1)j}} \cdot S^{(s)}_{E_{jj}} \cdot ... \cdot S^{(1)}_{E_{\gamma(\ell - 1) k}}
$$
Now if $\mindeg(\varphi(S^{(\ell)}_{E_{ik}})) = t$, due to the above and the fact $\Phi_0(\varphi(S^{(s)}_{E_{jj}}))$ is a non-zero scalar multiple of $p_j$, we still have the equality $\mindeg(\varphi(S^{(\ell + s )}_{E_{ik}})) = t$ for all $s\geq 0$.

On the other hand for $n\geq 1$, 
$$
S^{(\ell'+n\ell)}_{E_{ik}} \approx S^{(1)}_{E_{i\gamma'(1)}} \cdot ... \cdot S^{(1)}_{E_{\gamma'(m'-1)j'}} \cdot S^{(nr)}_{E_{j'j'}} \cdot ... \cdot S^{(1)}_{E_{\gamma'(\ell' - 1) k}}
$$
Due to Proposition \ref{proposition:singular-is-reducing} we have that $p_{j'}$ is $\varphi$-regular, so we get that $\mindeg(S^{(nr)}_{E_{j'j'}}) \geq nr$, concluding that $\mindeg(\varphi(S^{(\ell'+nr)}_{E_{ik}})) \geq \ell' + nr$ for all $n \geq 1$, obtaining a contradiction. It follows that $p_{j'}$ must be reducing for $Arv(P)$.

For the second part, if by negation $p_{j'}$ is $\varphi$-regular, and $1 \leq r \in \mathbb{N}$ is minimal such that $P^{(r)}_{j'j'} > 0$, with the same assumptions on $\gamma$ and $\gamma'$ as before, we have the following equality for $n \geq 1$,
$$
S^{(\ell' + n r)}_{E_{ik}} \approx S^{(1)}_{E_{i\gamma'(1)}} \cdot ... \cdot S^{(1)}_{E_{\gamma'(m'-1)j'}} \cdot (S^{(r)}_{E_{j'j'}})^{n} \cdot ... \cdot S^{(1)}_{E_{\gamma'(\ell'-1)k}}
$$
So we will have $\mindeg(\varphi(S^{(\ell'+n r)}_{E_{ik}})) \geq \ell' +nr$ for all $n \geq 1$, obtaining a contradiction again.
\end{proof}

\begin{proposition} \label{proposition:almost-isomorphic-graphs-reducing}
Let $P$ be stochastic and $i\neq j$ both in $\St$ such that $(i,j)\in E(P)$, with $p_i,p_j$ both reducing for $Arv(P)$. Assume $\varphi : \mathcal{T}_+(P) \rightarrow \mathcal{T}_+(Q)$ is an isomorphism. Then
$$\mindeg(\varphi(S^{(1)}_{E_{ij}})) = 1$$
\end{proposition}

\begin{proof}
Surely it is the case that $\mindeg(\varphi(S^{(1)}_{E_{ij}})) \geq 1$ due to Lemma \ref{lemma:perpendicular-subproduct}. Since $p_i,p_j$ are reducing for $Arv(P)$, there exist $i',j'\in \St$ such that $p_{i'},p_{j'}$ are reducing for $Arv(Q)$, and $\varphi_{p_i,p_{i'}}$ and $\varphi_{p_j,p_{j'}}$ are bounded isomorphisms. Assume by negation that $k \geq 2$ is minimal such that $p_{i'} \varphi(S^{(1)}_{E_{ij}}) p_{j'} = \sum^{\infty}_{n=k}a_nS^{(n)}_{E_{i'j'}}$ with $a_k\neq 0$ and $S^{(k)}_{E_{i'j'}} \neq 0$ (which also shows that a path of length at most $k$ from $i'$ to $j'$ exists in $Q$). Thus $\mindeg(\varphi(S^{(1)}_{E_{ij}})) \geq 2$, and since for all $n> 1$ we have that $S^{(n)}_{E_{ij}}$ is either a multiple of at least two operators of minimal degree 1, or a multiple of $S^{(1)}_{E_{ij}}$, it must be that $\mindeg(\varphi(S^{(n)}_{E_{ij}})) \geq 2$ for all $n\geq 1$. Write $p_i \varphi^{-1}(S^{(1)}_{E_{i'j'}}) p_j = \sum^{\infty}_{n=m}b_nS^{(n)}_{E_{ij}}$ for $b_m \neq 0$ and compute
\begin{align*}
S^{(1)}_{E_{i'j'}} & = \Phi_1(p_{i'} \varphi(\varphi^{-1}(S^{(1)}_{E_{i'j'}})) p_{j'}) = \Phi_1(\varphi(p_i \varphi^{-1}(S^{(1)}_{E_{i'j'}})p_j)) = 
\Phi_1(\varphi(\sum^{\infty}_{n=m}b_nS^{(n)}_{E_{ij}}))  
\\
& = \sum^{\infty}_{n=m}b_n\Phi_1(\varphi(S^{(n)}_{E_{ij}})) = 0
\end{align*}
This shows that $S^{(1)}_{E_{i'j'}} = 0 $ and in particular that $(i',j') \notin E(Q)$. Let $\ell \geq 2$ be the minimal length of a path from $i'$ to $j'$. Since for all $n\geq \ell$ the operator $S^{(n)}_{E_{i'j'}}$ can be written as a multiple of at least $2$ operators of minimal degree 1, we have that $\mindeg(\varphi^{-1}(S^{(n)}_{E_{i'j'}})) \geq 2$ for all $n\geq \ell$. By a similar computation to the one above we obtain
\begin{align*}
S^{(1)}_{E_{ij}} & = \Phi_1(p_{i} \varphi^{-1}(\varphi(S^{(1)}_{E_{ij}})) p_{j}) = \Phi_1(\varphi^{-1}(p_{i'} \varphi(S^{(1)}_{E_{ij}})p_{j'})) = 
\Phi_1(\varphi^{-1}(\sum^{\infty}_{n=k}a_nS^{(n)}_{E_{i'j'}}))\\
& =  \sum^{\infty}_{n=k}a_n\Phi_1(\varphi^{-1}(S^{(n)}_{E_{i'j'}})) = 0
\end{align*}
Which contradicts $(i,j)\in E(P)$.
\end{proof}

\begin{proposition}\label{proposition:forcing-length}
Let $P$ and $Q$ be stochastic matrices, and let $\varphi:\mathcal{T}_+(P) \rightarrow \mathcal{T}_+(Q)$ be an isomorphism. Assume that $j\in \St$ is $\varphi$-singular. Assume further that for $i,k \in \St$ there exists a \streamlined path from $i$ to $k$ through $j$ of length $\ell$ in the directed graph of $P$. Then any \streamlined path from $i$ to $k$ in the directed graph of $P$ is of length not greater than $\ell$.
\end{proposition}
\begin{proof} 
Before we begin, we note that all states considered here are reducing due to Proposition \ref{proposition:forcing-non-singular}, thus making the use of Proposition \ref{proposition:almost-isomorphic-graphs-reducing} possible.
Assume by negation that while $\gamma$ is a \streamlined path of length $\ell$ from $i$ to $k$ with $\gamma(m) = j$, we have a \streamlined path $\gamma'$ of length $r > \ell$ from $i$ to $k$.
Thus, up to a positive scalar multiple, which we denote by $\approx$, we have the following chain of equalities
$$
S^{(1)}_{E_{i\gamma'(1)}} \cdot ... \cdot S^{(1)}_{E_{\gamma'(r -1)k}} \approx S^{(r)}_{E_{ik}} \approx S^{(1)}_{E_{i\gamma(1)}} \cdot ... \cdot S^{(1)}_{E_{\gamma(m-1)j}} \cdot (S^{(1)}_{E_{jj}})^{r- \ell} \cdot S^{(1)}_{E_{j\gamma(m+1)}} \cdot ... \cdot S^{(1)}_{E_{\gamma(\ell -1)k}}
$$
But since $j$ is $\varphi$-singular, after applying $\varphi$ to both sides, the right hand side of the equation would be of minimal degree \emph{equal} to $\ell$ while the left hand side of the equation would be of minimal degree $r > \ell$, obtaining a contradiction.
\end{proof}

\begin{corollary} \label{corollary:well-defined-rotation}
Let $P$ and $Q$ be stochastic matrices, and let $\varphi:\mathcal{T}_+(P) \rightarrow \mathcal{T}_+(Q)$ be an isomorphism. Assume that $j_1, j_2 \in \St$ are $\varphi$-singular, such that there exist $i,k \in \St$ and two \streamlined paths from $i$ to $k$, one through $j_1$ of length $\ell_1$ and the other through $j_2$ of length $\ell_2$. Then $\ell_1 = \ell_2$.
\end{corollary}

\begin{proposition}[Property $\Diamond_1$] \label{proposition:reduction-to-equivalence-classes}
Let $P$ and $Q$ be stochastic matrices and $\varphi: \mathcal{T}_+(P) \rightarrow \mathcal{T}_+(Q)$ an isomorphism. Assume that for each $j \in \mathcal{R}_P$ we have that $p_j$ is $\varphi$-regular. Then for all $j\in \St$ we have that $p_j$ is $\varphi$-regular and thus $\varphi$ is regular.
\end{proposition}
\begin{proof}
Suppose towards a contradiction that there there exists 
$j \in I_P$ and $\varphi$-singular, and let $j' \in \mathcal{R}_P$ be such that $j \sim_P j'$. Then there exists a finite chain  $(j = s_0 , s_1 , ... , s_{\ell-1} , s_{\ell} = j')$  of elements in $I_P$, with $(s_k, s_{k+1})$ such that there exist $i,k \in \St$ with two \streamlined paths from $i$ to $k$, one through $s_k$ and the other through $s_{k+1}$ in the graph of $P$. By Remark \ref{remark:sigular-is-cyclic} every $s\in I_P$ must satisfy that $P_{ss} > 0$, so due to the second part of Proposition \ref{proposition:forcing-non-singular} we must have that $s_k$ and $s_{k+1}$ are either both $\varphi$-regular or both $\varphi$-singular. But $j' \in \mathcal{R}_P$ is $\varphi$-regular by assumption while $j$ is $\varphi$-singular, which leads to a contradiction.
\end{proof}

Let $P$ be a stochastic matrix. We say that a triple 
$(i,k, n) \in \Omega\times\Omega\times\nn$ is \emph{$P$-proper} (or just \emph{proper} when $P$ is
determined by the context) if $(i,k) \in E(P^n)$
and there exists a \streamlined path $\gamma$ 
from $i$ to $k$ in the graph of $P$ with length $\ell < n$ and there exists $0 \leq s \leq \ell$ such that $\gamma(s)\in I_P$. Note that by definition of $\sim_P$,
if $(i,k, n)$ is proper, it determines a unique element
$\mu \in \rp$ corresponding to the class of the element
$\gamma(s) \in I_P$. In this case we will say that
$\mu\in \rp$ is the singular class 
associated to $(i,k,n)$. 

\begin{theorem}[Property $\Diamond_2$]
Let $P$ be a stochastic matrix, and let $\Lambda = (\lambda_\mu)_{\mu \in \mathcal{R}_P}\in \mathbb{T}^{\mathcal{R}_P}$ be given. 
There exists a unique isometric (Id-)automorphism 
$V^{\Lambda} = \{ V_{n}^{\Lambda} \}$ 
of the subproduct system $Arv(P)$ satisfying the following
condition: for every $n\geq 1$ and for each pair $(i,k) \in E(P^n)$, 
$$
V_{n}^{\Lambda}(E_{ik}) = 
   \begin{cases} 
       \lambda_{\mu}^{n-\ell}E_{ik}, & (i, k, n) \text{ is proper, with associated singular class } \mu\in\rp\\
       E_{ik}, &  \text{otherwise }
   \end{cases}
$$
Moreover, the (Id)-similarity $\alpha_{\Lambda} : \mathcal{T}_+(P) \rightarrow \mathcal{T}_+(P)$ given by 
$\alpha_\Lambda = \Ad_{V^\Lambda}$ is an isometric
automor\-phism satisfying 
$f^{j}_{\alpha_{\Lambda}}(z) = \lambda_{\mu}z$
for all $j \in I_P$, where $j \sim_P \mu \in \rp$.
\end{theorem}
\begin{proof}
By Corollary~\ref{corollary:well-defined-rotation} we see that 
for each $n\in\nn$ the map $V_{n}^{\Lambda}$ on 
$\{ E_{ik}  \mid  (i,k) \in E(P^n)  \}$ is 
well-defined and clearly extends uniquely to a unitary 
correspondence morphisms on 
$\Span \{ E_{ik} \mid (i,k) \in E(P^n) \}$. Thus we obtain
a unique unitary W*-correspondence morphism  $V_n^\Lambda$ on $Arv(P)_n$.

We need to show that $U^P_{n,m}(V_{n}^{\Lambda} \otimes V_{m}^{\Lambda}) = V_{n+m}^{\Lambda} U^P_{n,m}$ for every $n,m \in \mathbb{N}$. It suffices to show that for all $(i,j,k) \in E(P^n,P^m)$ we have
\begin{equation} \label{eq:subproduct-indicators}
U^P_{n,m}(V_{n}^{\Lambda}(E_{ij}) \otimes V_{m}^{\Lambda}(E_{jk})) = V_{n+m}^{\Lambda} U^P_{n,m}(E_{ij} \otimes E_{jk})
\end{equation}
Let a triple $(i,j,k) \in E(P^n,P^m)$ be given for $n,m\in \nn$. We split the proof into two cases.

Case 1: Suppose  that $(i,k, n+m)$ is proper, with associated
singular class $\mu \in\rp$. 
Let $\gamma$ be a \streamlined 
path of length $\ell< n + m$ from $i$ to $k$ such that $\gamma(s) \sim_P \mu$ for some $0\leq s \leq \ell$. By definition $V_{n+m}^\Lambda(E_{ik}) = \lambda_\mu^{n+m-\ell} E_{ik}$.
We first show that
at least one of $(i,j, n)$ and $(j,k, m)$
must be proper with associated singular class $\mu$.  Clearly there exist
two paths $\gamma^{(n)}$ from $i$ to $j$ of length $n$ and $\gamma^{(m)}$ from $j$ to $k$ of length $m$. Notice that the concatenation $\gamma'$ 
of $\gamma^{(n)}$ and $\gamma^{(m)}$ provides
a path from $i$ to $k$ which by Proposition 
\ref{proposition:forcing-length} cannot be streamlined,
 since $\ell < n+m$. Therefore, there exists some repeated index 
$j'$ in the path $\gamma'$, hence $P^{(r)}_{j'j'}>0$ for some $r\geq 1$.
By Proposition~\ref{proposition:forcing-non-singular}, we
must have that $j'\in I_P$, and by definition of the equivalence
relation we have therefore $j'\sim_P \mu$. Let us suppose without
loss of generality that $j'$ is in $\gamma^{(n)}$. Then $(i,j, n)$ is proper with associated singular class $\mu$. Let $\ell_1$ denote the length of the streamlined path
obtained from $i$ to $j$ by culling repeated vertices from $\gamma^{(n)}$. We obtain $V_n^\Lambda(E_{ij}) = \lambda_\mu^{n-\ell_1} E_{ij}$. 
Let $\ell_2$ be the length of the streamlined path from $j$ to
$k$ obtained from the culling procedure applied to $\gamma^{(m)}$.
We have $\ell =\ell_1 + \ell_2$ by Corollary~\ref{corollary:well-defined-rotation}. If $(j,k, m)$ is also proper (necessarily with associated singular class $\mu$),
we have $V_m^{\Lambda}(E_{jk}) = \lambda_\mu^{m-\ell_2} E_{jk}$.
On the other hand, notice that by the same argument we used above, 
if $\gamma^{(m)}$ is not streamlined, 
we must have that $(j,k, m)$ is proper.
Therefore, if $(j,k, m)$ is improper 
we must have $\ell_2 = m$, i.e. $\gamma^{(m)}$ already streamlined,
and we also have for this case the formula $V_m^{\Lambda}(E_{jk}) = \lambda_\mu^{m-\ell_2} E_{jk}$.  Since $n+m - \ell = (n-\ell_1) + (m-\ell_2)$, we obtain \eqref{eq:subproduct-indicators}.

Case 2: Suppose that $(i,k, n+m)$ is improper, so that
$V_{n+m}^\Lambda(E_{ik}) = E_{ik}$. Then
we clearly must have that $(i,j,n)$ and $(j,k,m)$ are also improper. It follows that
$V_{n}^\Lambda(E_{ij}) = E_{ij}$ and 
$V_{m}^\Lambda(E_{jk}) = E_{jk}$, and therefore
 \eqref{eq:subproduct-indicators} holds.

Finally, the map $\alpha_\Lambda =\Ad_{V^\Lambda}$, 
given by Corollary~\ref{corollary:isometric-isomorphism},  clearly satisfies the stated properties since it leaves the vacuum state
invariant and if $j\in I_P$, we have $(\widehat{\alpha_\Lambda})_{p_j,p_j}(S_{E_{jj}}) =  \lambda_\mu S_{E_{jj}}$ where $\mu\in \rp$ is the unique element with 
$j \sim_P \mu$. This is because any \streamlined path from $j$ to $j$ must be of length $\ell = 0$.
\end{proof}

\begin{proposition}[Property $\Diamond_3$] \label{proposition:equivalene-preserving}
Let $P$ and $Q$ be stochastic matrices and let $\varphi : \mathcal{T}_+(P) \rightarrow \mathcal{T}_+(Q)$ be an isomorphism. Then $\varphi$ induces an equivalence preserving bijection $\Upsilon_{\varphi}:I_P \rightarrow I_Q$ uniquely determined by the identity  $p_{\Upsilon(j)} = \Phi_0(\varphi(p_j)) $ for every $j \in I_P$. Furthermore, we have $\Upsilon_{\varphi}^{-1} = \Upsilon_{\varphi^{-1}}$.
\end{proposition}

\begin{proof}
By Proposition \ref{proposition:singular-is-reducing}, we have that $p_j$ is a reducing projection for $Arv(P)$
for every $j\in I_P$. By Propositions \ref{proposition:isometric-one-dim} and  \ref{proposition:singularity-among-reducing} we see that for all $j\in I_P$ we must have that there is a unique $j' \in I_Q$ such that $p_{j'} = \Phi_0(\varphi(p_j))$, hence $\Upsilon_\varphi$ is a well-defined injection. Furthermore,
by symmetry and Proposition~\ref{proposition:isometric-one-dim}, we have that it is onto and $\Upsilon_{\varphi}^{-1} = \Upsilon_{\varphi^{-1}}$. To show that $\Upsilon_{\varphi}$ preserves equivalence, let $j_1,j_2$ be two elements in $I_P$ such that there exist two \streamlined paths (which must be of the same length due to Corollary \ref{corollary:well-defined-rotation}) $\gamma$ from $i$ to $k$ of length $\ell$ with $\gamma(t) = j_1$, and $\gamma'$ from $i$ to $k$ of length $\ell$ with $\gamma'(r) = j_2$.
Thus, we have that up to a positive scalar, which we denote by $\approx$,
$$
S^{(1)}_{E_{i\gamma(1)}} \cdot ... \cdot S^{(1)}_{E_{\gamma(t-1)j_1}} \cdot S^{(1)}_{E_{j_1\gamma(t+1)}} \cdot ... \cdot S^{(1)}_{E_{\gamma(\ell-1)k}} \approx S^{(\ell)}_{E_{ik}} \approx S^{(1)}_{E_{i\gamma'(1)}} \cdot ... \cdot S^{(1)}_{E_{\gamma'(r-1)j_2}} \cdot S^{(1)}_{E_{j_2\gamma'(r+1)}} \cdot ... \cdot S^{(1)}_{E_{\gamma'(\ell-1)k}}
$$
Denote $p_{j'_1} = \Phi_0(\varphi(p_{j_1}))$, $p_{j'_2} = \Phi_0(\varphi(p_{j_2}))$, $p_{i'} = \Phi_0(\varphi(p_{i}))$ and $p_{k'} = \Phi_0(\varphi(p_{k}))$. By Proposition \ref{proposition:almost-isomorphic-graphs-reducing} we have the equality $\mindeg(\varphi(S^{(\ell)}_{E_{ik}})) = \ell$. Thus, write uniquely $\varphi(S^{(\ell)}_{E_{ik}}) = \sum_{n=\ell}^{\infty}S^{(n)}_{\xi_n}$
For $\xi_n \in Arv(Q)_n$ with $\xi_{\ell} = p_{i'} \xi_{\ell} p_{k'} \neq 0$. By applying $\Phi_\ell$ to the equation above, we have that up to a positive scalar,
$$
p_{i'} \Phi_t(\varphi(S^{(1)}_{E_{i\gamma(1)}} \cdot ... \cdot S^{(1)}_{E_{\gamma(t-1)j_1}})) p_{j_1'} \Phi_{\ell-t}(\varphi(S^{(1)}_{E_{j_1\gamma(t+1)}} \cdot ... \cdot S^{(1)}_{E_{\gamma(\ell-1)k}}) )p_{k'} \approx 
$$
$$
p_{i'} \Phi_r(\varphi(S^{(1)}_{E_{i\gamma'(1)}} \cdot ... \cdot S^{(1)}_{E_{\gamma'(r-1)j_2}})) p_{j'_2} \Phi_{\ell-r}(\varphi(S^{(1)}_{E_{j_2\gamma'(r+1)}} \cdot ... \cdot S^{(1)}_{E_{\gamma'(\ell-1)k}})) p_{k'} \approx
 \Phi_{\ell}(\varphi(S^{(\ell)}_{E_{ik}})) = S^{(\ell)}_{\xi_{\ell}} \neq 0
$$
Thus we obtain two \streamlined paths from $i'$ to $k'$, one through $j'_1$ and the other through $j'_2$ and we are done.
\end{proof}

\subsection*{Main results}
We now use the three properties of $\sim_P$ to establish the existence of a regular isomorphism, from that of a general isomorphism. We were inspired by \cite[Proposition 4.7]{DRS}.

\begin{theorem}\label{theorem:main-tool}
Let $P$ and $Q$ be stochastic matrices over $\St$. If there exists an algebraic isomorphism $\varphi : \mathcal{T}_+(P) \rightarrow \mathcal{T}_+(Q)$ then there exists a \emph{regular} isomorphism from $\mathcal{T}_+(P)$ to $\mathcal{T}_+(Q)$, which can be taken to be isometric if $\varphi$ is isometric.
\end{theorem}

\begin{proof}
Let $\varphi : \mathcal{T}_+(Q) \rightarrow \mathcal{T}_+(P)$ be
an algebraic isomorphism. For each 
$(\Lambda, \Theta) \in \torus^\rp \times \torus^{\calr_Q}$,
let $\alpha_\Lambda$ and $\alpha_\Theta$ be the
automorphisms of $\tensor(P)$ and $\tensor(Q)$, respectively,
provided by property $\Diamond_2$, and consider
the isomorphism from $\mathcal{T}_+(Q)$ to $\mathcal{T}_+(P)$ given by
$$
\Psi_{\Lambda, \Theta} = 
\varphi \circ \alpha_\Lambda \circ \varphi^{-1} \circ
\alpha_{\Theta} \circ \varphi 
$$
We will prove that
there exists a pair $(\Lambda_0, \Theta_0)$ such that 
$\psi  = \Psi_{\Lambda_0, \Theta_0}$ has the property that 
$f^j_\psi(0)=0$ for every $j \in \rp$. By Proposition~\ref{prop:fj} and property $\Diamond_1$, such an isomorphism $\psi$ is regular. Furthermore, if $\varphi$ is isometric, then $\psi$ is isometric by property $\Diamond_2$.

For simplicity, for every $j \in \mathcal{R}_P$ and for each pair $(\Lambda, \Theta) \in \torus^\rp \times \torus^{\calr_Q}$,
let us denote $f_{\Lambda, \Theta}^j = 
f_{\Psi_{\Lambda, \Theta}}^j$. Let 
$j'=\overline{\Upsilon}_\varphi(j)$, and notice that 
$j = \overline{\Upsilon}_{\varphi^{-1}}(j') = (\overline{\Upsilon}_\varphi)^{-1}(j')$ by property 
$\Diamond_3$. Therefore, by contravariance, 
we have that
$$
f_{\Lambda, \Theta}^j = f_\varphi^j \circ 
f^{j'}_{\alpha_\Theta} \circ
f_{\varphi^-1}^{j'} \circ
f^{j}_{\alpha_\Lambda} \circ
f_\varphi^j
$$
Let us denote by $T_j = f^j_\varphi$ and 
$\theta'_j = \theta_{j'}$. By Proposition~\ref{proposition:isometric-one-dim} and 
property $\Diamond_3$, we have that
$f_{\varphi^-1}^{j'} = T_j^{-1}$. Thus we have
that for every $j\in \rp$,
$$
f_{\Lambda, \Theta}^j(0) = 
f_\varphi^j \circ 
f^{j'}_{\alpha_\Theta} \circ
f_{\varphi^-1}^{j'} \circ
f^{j}_{\alpha_\Lambda} \circ
f_\varphi^j (z) = T_j( \theta'_{j} T_j^{-1}( \lambda_j( T_j(0))))
$$
Since $T_j$ is a Moebius transformation, it is an
elementary fact that there exist $\theta'_j , \lambda_j \in \torus$ such that $f_{\Lambda, \Theta}^j(0)=0$. Indeed,
if $T_j(0)=0$ this is trivial. If $T_j(0)\neq 0$, then 
$C_j = \torus \cdot T_j(0)$ is a circle centered at
origin (and not containing it). On the other hand, since $T_j$ is a Moebius transformation, $T_j^{-1}(C_j)$ is a circle and it clearly
contains the origin, since $T_j(0) \in C_j$. Therefore, the
larger circle $C'_j = \torus \cdot T_j^{-1}(C_j)$ obtained by its rotation contains the interior of $T_j^{-1}(C_j)$. It follows that $T_j(C'_j)$ is a disk that will contain the interior of $C_j$, and
therefore it contains the origin.
\end{proof}

\begin{corollary} \label{corollary:isometric-case}
Let $P$ and $Q$ be stochastic matrices over $\St$. If there exists an isometric isomorphism $\varphi : \mathcal{T}_+(P) \rightarrow \mathcal{T}_+(Q)$ then there exists a \emph{graded} completely isometric isomorphism $\check{\varphi} : \mathcal{T}_+(P) \rightarrow \mathcal{T}_+(Q)$.
\end{corollary}

\begin{proof}
If there exists an isometric isomorphism $\varphi : \mathcal{T}_+(P) \rightarrow \mathcal{T}_+(Q)$ then there exists a \emph{regular} isometric isometric isomorphism and since it must be admissible by Lemma \ref{lemma:isometric-automatic-admissibility}, Proposition \ref{proposition:condition-for-semi-graded} implies that it is semi-graded. Thus, by Proposition \ref{proposition:isometric-semi-graded-turns-graded} there exists a \emph{graded} completely isometric isomorphism.
\end{proof}

Further, if one takes $\St$ to be a finite set in Theorem \ref{theorem:main-tool}, then one obtains

\begin{corollary}\label{corollary:bounded-case}
Let $P$ and $Q$ be stochastic matrices over \emph{finite} $\St$. If there exists an isomorphism $\varphi : \mathcal{T}_+(P) \rightarrow \mathcal{T}_+(Q)$ then there exists a \emph{graded} completely bounded isomorphism $\check{\varphi} : \mathcal{T}_+(P) \rightarrow \mathcal{T}_+(Q)$.
\end{corollary}

\begin{proof}
Since $\St$ is finite, the regular $\varphi$ assured by Theorem \ref{theorem:main-tool} must be admissible by Proposition \ref{proposition:bounded-automatic-admisibility}, and by an appeal to Proposition \ref{proposition:condition-for-semi-graded} we see that $\varphi$ is semi-graded. Finally, by Proposition \ref{proposition:bounded-semi-graded-turns-graded} we arrive at the desired isomorphism $\check{\varphi}$.
\end{proof}

We now rephrase our results in a unified manner. Recall 
Theorems~\ref{theorem:telling-apart} and 
\ref{theorem:finite-essential-subproduct}.

\begin{theorem}[Isometric tensor algebra isomorphism] \label{theorem:A}
Let $P$ and $Q$ be stochastic matrices over a set $\St$. Then the following are equivalent
\begin{enumerate}
\item
There exists a $\rho$-unitary isomorphism from $Arv(P)$ to $Arv(Q)$ for some *-automorphism $\rho$ of $\vnMaS$.
\item
There exists a graded completely isometric isomorphism $\varphi: \mathcal{T}_+(P) \rightarrow \mathcal{T}_+(Q)$
\item
There exists an isometric isomorphism $\varphi : \mathcal{T}_+(P) \rightarrow \mathcal{T}_+(Q)$
\end{enumerate}
If moreover $P$ and $Q$ are \emph{recurrent}, then the above conditions are equivalent to $P \cong_{\sigma} Q$ for some permutation $\sigma$ of $\St$.
\end{theorem}

\begin{theorem}[Algebraic tensor algebra isomorphism for finite matrices] \label{theorem:B}
Let $P$ and $Q$ be \emph{finite} stochastic matrices over a set $\St$. Then the following are equivalent
\begin{enumerate}
\item
There exists a $\rho$-similarity isomorphism from $Arv(P)$ to $Arv(Q)$ for some *-automorphism $\rho$ of $\vnMaS$.
\item
There exists a graded completely bounded isomorphism $\varphi: \mathcal{T}_+(P) \rightarrow \mathcal{T}_+(Q)$
\item
There exists an algebraic isomorphism $\varphi : \mathcal{T}_+(P) \rightarrow \mathcal{T}_+(Q)$
\end{enumerate}
If moreover $P$ and $Q$ are \emph{essential}, then the above conditions are equivalent to $P\sim_{\sigma}Q$ for some permutation $\sigma$ of $\St$.
\end{theorem}

\begin{example}
For every $0<r<1/2$ let $P_r = \begin{bmatrix}
r & (1-r) \\
r & (1-r) 
\end{bmatrix}.
$
Then it follows from the previous theorem that 
$\tensor(P_r)$ is isomorphic to $\tensor(P_s)$ for
every $r\neq s \in (0,1/2)$, however the two algebras
are isometrically isomorphic only when $r=s$.
\end{example}

We note that when $P$ and $Q$ are essential (and $\St$ is possibly 
infinite), it is possible to prove more directly, and
without recourse to Theorem \ref{theorem:main-tool},
that the existence of an \emph{admissible} algebraic (resp. isometric)
isomorphism from $\tensor(P)$ to $\tensor(Q)$
implies the existence of a graded algebraic
(resp. isometric) isomophism from $\tensor(P)$ to $\tensor(Q)$.
For instance, under those conditions one can proceed as in the proof of Corollary~\ref{corollary:essential-case}, after 
fixing the $1 \times 1$ irreducible blocks by using Theorem~\ref{theorem:finite-essential-subproduct}. 
 The main point of Theorem \ref{theorem:main-tool} was to deal with the general, non-essential case.

\section*{Acknowledgments} We would like to thank Orr Shalit for his many helpful 
remarks which substantially refined this work. We also
thank Kenneth Dadvison, Ilan Hirshberg, N. Christopher Phillips 
and Baruch Solel for many useful comments. 
The first author would also like to thank Ariel Yadin, for many 
discussions and insights on stochastic matrices.

%\vspace{1cm}
%\bibliographystyle{amsalpha}
%\bibliography{subproducts}{}

\providecommand{\bysame}{\leavevmode\hbox to3em{\hrulefill}\thinspace}
\providecommand{\MR}{\relax\ifhmode\unskip\space\fi MR }
\providecommand{\MRhref}[2]{%
  \href{http://www.ams.org/mathscinet-getitem?mr=#1}{#2}
}
\providecommand{\href}[2]{#2}

\end{document}